\documentclass[12pt,a4paper,reqno]{amsart}
\usepackage{amssymb}
\usepackage{amscd}
\usepackage{enumerate}
\usepackage{graphicx}
\usepackage{siunitx}
\usepackage{tikz-cd}
\usepackage{color}
\usetikzlibrary{arrows}
\numberwithin{equation}{section}

\usepackage{mathabx}

\usepackage{mathtools}
\usepackage[tableposition=top]{caption}
\usepackage{booktabs,dcolumn}

\DeclareFontFamily{OT1}{rsfs}{}
\DeclareFontShape{OT1}{rsfs}{n}{it}{<-> rsfs10}{}
\DeclareMathAlphabet{\mathscr}{OT1}{rsfs}{n}{it}

\theoremstyle{plain}

\newtheorem{theorem}{Theorem}[section]
\newtheorem{proposition}[theorem]{Proposition}
\newtheorem{lemma}[theorem]{Lemma}

\theoremstyle{definition}

\newtheorem{definition}[theorem]{Definition}
\newtheorem{remark}[theorem]{Remark}

\newcommand\R{\mathbb{R}}
\newcommand\Z{\mathbb{Z}}

\newcommand\eps{\varepsilon}

\newcommand\M{\mathcal{M}}

\usepackage{algorithm, algorithmic}

\begin{document}

\title[Finite time blowup for modified Euler]{Finite time blowup for Lagrangian modifications of the three-dimensional Euler equation}

\author{Terence Tao}
\address{UCLA Department of Mathematics, Los Angeles, CA 90095-1555.}
\email{tao@math.ucla.edu}



\begin{abstract}  In the language of differential geometry, the incompressible inviscid Euler equations can be written in vorticity-vector potential form as
\begin{align*}
\partial_t \omega + {\mathcal L}_u \omega &= 0\\
u &= \delta \tilde \eta^{-1} \Delta^{-1} \omega
\end{align*}
where $\omega$ is the vorticity $2$-form, ${\mathcal L}_u$ denotes the Lie derivative with respect to the velocity field $u$, $\Delta$ is the Hodge Laplacian, $\delta$ is the codifferential (the negative of the divergence operator), and $\tilde \eta^{-1}$ is the canonical map from $2$-forms to $2$-vector fields induced by the Euclidean metric $\eta$.  In this paper we consider a generalisation of these Euler equations in three spatial dimensions, in which the vector potential operator $\tilde \eta^{-1} \Delta^{-1}$ is replaced by a more general operator $A$ of order $-2$; this retains the Lagrangian structure of the Euler equations, as well as most of its conservation laws and local existence theory.  Despite this, we give three different constructions of such an operator $A$ which admits smooth solutions that blow up in finite time, including an example on $\R^3$ which is self-adjoint and positive definite.  This indicates a barrier to establishing global regularity for the three-dimensional Euler equations, in that any method for achieving this must use some property of those equations that is not shared by the generalised Euler equations considered here.
\end{abstract}

\maketitle


\section{Introduction}

\subsection{Formal theory of the generalised Euler equations}

In this paper we will consider finite time blowup for generalised equations of Euler type on Euclidean spaces $\R^d$, and more generally\footnote{A substantial portion of the discussion here could in fact be extended to arbitrary smooth Riemannian manifold domains, but we will not need to do so here.} on flat cylinders $\R^m \times (\R/\Z)^{d-m}$ for $0 \leq m \leq d$, thus $d$ represents the total number of spatial dimensions, and $m$ the total number of unbounded spatial dimensions.  We will restrict attention primarily to the domains $\R^3$, $\R^2 \times \R/\Z$, and $\R^2$.  In particular we shall assume that $m \geq 2$, in order to avoid some technical issues involving the Biot-Savart law at low frequencies.

Recall that if $\M = \R^m \times (\R/\Z)^{d-m}$ with $d \geq m \geq 2$, the Euler equations for incompressible, inviscid fluids on $\M$ can be written as
\begin{equation}\label{euler-eq}
\begin{split}
\partial_t u + (u \cdot \nabla) u &= - \nabla p \\
\nabla \cdot u &= 0
\end{split}
\end{equation}
where $u\colon \R \times \M \to \R^d$ is the velocity field and $p\colon \R \times \M \to \R$ is the pressure field.  For now, we shall only interpret the system \eqref{euler-eq} at the formal level, ignoring issues of regularity or decay, and also ignoring all cohomology by assuming that closed forms are automatically exact; we will return to these issues later, when we discuss local existence theory.

It will be convenient in this paper to use the language of differential geometry, in order to minimise the reliance on the Euclidean metric $\eta$ on $\M$; this will become useful later when we exploit the properties of the Lie derivative ${\mathcal L}_u$ (which will not preserve the Euclidean metric in general), as well as when we temporarily switch over to (modified) cylindrical coordinates in Section \ref{period-sec}.  See for instance \cite{aubin} for a basic introduction to the differential geometry concepts used in this paper.

We begin with setting out notation for the standard Cartesian coordinates on $\M$, though we emphasise that the differential geometry constructions introduced here are coordinate-independent (although some of them will rely on the standard volume form $d\operatorname{vol}$ on $\M$). We let $x^1,\dots,x^d$ denote\footnote{We use superscripts here instead of the more customary subscripts $x_1,\dots,x_d$ in order to be compatible with the raising and lowering conventions of differential geometry.} the usual coordinates on $\M = \R^m \times (\R/\Z)^{d-m}$ (thus $x^1,\dots,x^m \in \R$ and $x^{m+1},\dots,x^d \in \R/\Z$).  Taking differentials, we obtain the standard $1$-forms $dx^1,\dots,dx^d$ on $\M$, and then on taking wedge products we obtain the standard volume form
$$ d\operatorname{vol} \coloneqq dx^1 \wedge \dots \wedge dx^d.$$
Dually, we have the standard vector fields
$$ \frac{d}{d x^1}, \dots, \frac{d}{d x^d}.$$
We have suggestively written these vector fields to resemble first-order differential operators, but in order to reduce confusion, we will use the symbol $\partial_i$ (as opposed to $\frac{\partial}{\partial x^i}$) for $i=1,\dots,d$ to denote the partial differentiation operation in the $x^i$ direction, to distinguish this partial differentiation operation from the associated vector field $\frac{d}{dx^i}$.

We let $\Lambda_0(\M)$ denote the space of (formal) scalar functions from $\M$ to $\R$.  More generally, for any $k \geq 0$, let $\Lambda_k(\M) = \Gamma( \bigwedge^k T^* \M)$ denote the space of (formal) $k$-forms on $\M$, thus for instance $dx^1,\dots,dx^d \in \Lambda_1(\M)$ and $d\operatorname{vol} \in \Lambda_d(\M)$.  The coordinates of a $k$-form $\omega \in \Lambda_k(\M)$ will be denoted 
$\omega_{i_1 \dots i_k}$,
where the indices $i_1,\dots,i_k$ range from $1$ to $d$ (with the usual summation conventions), and $\omega_{i_1 \dots i_k}$ is antisymmetric in $i_1,\dots,i_k$.  Of course, $\Lambda_k(\M)$ is trivial for $k > d$.  The standard basis for $\Lambda_k(\M)$ (as a $\Lambda_0(\M)$-module) is given by the constant $k$-forms
$$ dx^{i_1} \wedge \dots \wedge dx^{i_k}$$
for $1 \leq i_1 < \dots < i_k \leq d$; thus with the usual summation conventions we have
$$ \omega = \frac{1}{k!} \omega_{i_1 \dots i_k} dx^{i_1} \wedge \dots \wedge dx^{i_k}$$
(the $\frac{1}{k!}$ factor reflecting the fact that the $i_1,\dots,i_k$ are not necessarily in increasing order with the usual summation conventions).
Dual to the space $\Lambda_k(\M)$ of $k$-forms is the space $\Gamma^k(\M) = \Gamma(\bigwedge^k T \M)$ of (formal) $k$-vector fields on $\M$; the coordinates of an element $\alpha \in \Gamma^k(\M)$ will be denoted
$\alpha^{i_1 \dots i_k}$ and is antisymmetric in the $i_1,\dots,i_k$; again, $\Gamma^k(\M)$ is trivial for $k>d$, and we adopt the convention that it is trivial for $k<0$ also.  We also make the identification $\Lambda_0(\M) \equiv \Gamma^0(\M)$.  The standard basis for $\Gamma^k(\M)$ (as a $\Lambda_0(\M) = \Gamma^0(\M)$-module) is given by the constant $k$-vector fields
$$ \frac{d}{d x^{i_1}} \wedge \dots \wedge \frac{d}{d x^{i_k}}$$
for $1 \leq i_1 < \dots < i_k \leq d$, thus
$$ \alpha=  \frac{1}{k!} \alpha^{i_1 \dots i_k} \frac{d}{d x^{i_1}} \wedge \dots \wedge \frac{d}{d x^{i_k}}.$$
We have the usual pairing operation $\langle, \rangle\colon \Lambda_k(\M) \times \Gamma^k(\M) \to \Lambda_0(\M)$, defined in coordinates as
$$ \langle \omega, \alpha \rangle \coloneqq \frac{1}{k!} \omega_{i_1 \dots i_k} \alpha^{i_1 \dots i_k},$$
thus for instance $\langle dx^i, \frac{d}{dx^j} \rangle$ equals $1$ when $i=j$, and $0$ otherwise.

We have the usual exterior derivative operator $d\colon \Lambda_k(\M) \to \Lambda_{k+1}(\M)$, defined in coordinates as
$$ (d\omega)_{i_1 \dots i_{k+1}} \coloneqq \sum_{j=1}^{k+1} (-1)^{j-1} \partial_{i_j} \omega_{i_1 \dots i_{j-1} i_{j+1} \dots i_{k+1}};$$
this is of course compatible with our notation $dx^1,\dots,dx^d$ for the standard $1$-forms (viewing each coordinate function $x^i$, locally at least, as an element of $\Lambda_0(\M)$).  Dually, we have the codifferential\footnote{In the usual Hodge theory literature, one uses a Riemannian metric to identify $\Gamma^k$ with $\Lambda_k$ as per \eqref{teta-def}, so that the codifferential acts on $k$-forms rather than $k$-vector fields.  However, it will be more convenient here to avoid using the metric identification to define the codifferential, because the Euclidean metric $\eta$ will not in general be preserved by flowing along the velocity field $u$.} $\delta\colon \Gamma^{k+1}(\M) \to \Gamma^k(\M)$ defined in coordinates as
$$ (\delta \alpha)^{i_2 \dots i_{k+1}} \coloneqq -\partial_{i_1} \alpha^{i_1 \dots i_{k+1}}.$$
Thus, for instance, if $X \in \Gamma^1(\M)$ is a vector field, then $\delta X = - \operatorname{div} X$ is the negative divergence of $X$.
As is well known, we have
\begin{equation}\label{dd}
d^2 = 0
\end{equation}
and
\begin{equation}\label{deldel}
\delta^2 = 0;
\end{equation}
see e.g. \cite[\S 5.17]{aubin}.
We let $B_k(\M) \coloneqq \{ \omega \in \Lambda_k(\M): d \omega = 0 \}$ denote the space of closed $k$-forms, and similarly let $B^k(\M) \coloneqq \{ \alpha \in \Gamma^k(\M): \delta \alpha = 0 \}$ denote the space of divergence-free $k$-vector fields.

The Euclidean metric $\eta$ on $\M$ is given by its first fundamental form
$$ d\eta^2 = (dx^1)^2 + \dots + (dx^d)^2.$$
It can be viewed in coordinates as a $(0,2)$-tensor $\eta_{ij}$, or after inversion as a $(2,0)$-tensor $\eta^{ij}$.  It provides an identification $\tilde \eta\colon \Gamma^k(\M) \to \Lambda_k(\M)$ of $k$-vector fields with $k$-forms, defined in coordinates by
\begin{equation}\label{teta-def}
 (\tilde \eta T)_{i_1 \dots i_k} \coloneqq \eta_{i_1 j_1} \dots \eta_{i_k j_k} T^{j_1 \dots j_k},
\end{equation}
thus for instance
$$ \tilde \eta\left( \frac{d}{d x^{i_1}} \wedge \dots \wedge \frac{d}{d x^{i_k}} \right) = dx^{i_1} \wedge \dots \wedge dx^{i_k}$$
or upon inverting
$$ \tilde \eta^{-1}\left( dx^{i_1} \wedge \dots \wedge dx^{i_k} \right) =  \frac{d}{d x^{i_1}} \wedge \dots \wedge \frac{d}{d x^{i_k}}.$$ 

Suppose that $u,p$ (formally) solve \eqref{euler-eq}.  For each time $t$, $u(t)$ and $p(t)$ can be viewed as elements of $B^1(\M)$ and $\Lambda_0(\M)$ respectively; in coordinates with the usual summation conventions, \eqref{euler-eq} becomes
\begin{align*}
\partial_t u^i + u^j \partial_j u^i &= - \eta^{ij} \partial_j p \\
\partial_i u^i &= 0.
\end{align*}
If we define the \emph{covelocity} $v(t) \in \Lambda_1(\R^d)$ to be the $1$-form $v = \tilde \eta u$, thus in coordinates
$$ v_i \coloneqq \eta_{ij} u^j,$$
then we see that
\begin{equation}\label{vip}
\partial_t v_i + u^j \partial_j v_i +(\partial_i u^j) v_j = - \partial_i \tilde p 
\end{equation}
where the \emph{modified pressure} $\tilde p(t) \in \Lambda_0(\M)$ is given by the formula
$$ \tilde p \coloneqq p - \frac{1}{2} u^j v_j.$$
Recalling (see e.g. \cite[\S 3.4]{aubin}) that the \emph{Lie derivative} ${\mathcal L}_u$ along a vector field $u$ acts on $k$-forms $\omega \in \Lambda_k(\M)$ by the Cartan formula 	
\begin{equation}\label{cartan}
 {\mathcal L}_u \omega = \iota_u(d\omega) + d(\iota_u \omega)
\end{equation}
where $\iota_u\colon \Lambda_{k+1}(\M) \to \Lambda_k(\M)$ is the contraction operator
$$ (\iota_u \omega)_{i_2 \dots i_{k+1}} \coloneqq u^{i_1} \omega_{i_1 \dots i_{k+1}},$$
we see that
$$ ({\mathcal L}_u v)_i = u^j \partial_j v_i + (\partial_i u^j) v_j$$
and hence \eqref{vip} can be written in coordinate-free notation as
$$ \partial_t v + {\mathcal L}_u v = - d \tilde p.$$
If we define the \emph{vorticity} $\omega(t) \in \Lambda_2(\M)$ to be the exterior derivative 
$$\omega \coloneqq dv = d \tilde \eta u$$
of the covelocity $v$, and use the basic commutativity identity 
\begin{equation}\label{dcom}
d {\mathcal L}_u = {\mathcal L}_u d
\end{equation}
(see e.g. \cite[Proposition 3.6]{aubin}) and \eqref{dd}, we conclude that $\omega(t)$ in fact lies in $B_2(\M)$ (i.e. it is closed) and obeys the \emph{vorticity equation}
\begin{equation}\label{vort}
 \partial_t \omega + {\mathcal L}_u \omega = 0.
\end{equation}

\begin{remark}\label{hodge-rem}  The standard volume form $d\operatorname{vol} \coloneqq dx^1 \wedge \dots \wedge dx^d \in \Gamma^d(\M)$, induces the \emph{Hodge duality operator} $*\colon \Gamma^k(\M) \to \Lambda_{d-k}(\M)$ for $0 \leq d \leq k$, defined by the formula
$$ \omega \wedge (*\alpha) = \langle \omega,\alpha \rangle d\operatorname{vol}$$
for $\omega \in \Lambda_k(\M)$ and $\alpha \in \Gamma^k(\M)$.  Thus for instance we have
\begin{equation}\label{dim}
\delta = * d *^{-1}.
\end{equation}
The dual $*^{-1} \omega \in \Gamma^{d-2}(\M)$ of the vorticity is then a scalar function in two dimensions and a vector field in three dimensions, and in the Euler equation literature it is common to refer to this scalar or vector field, rather than the $2$-form $\omega$, as the vorticity (i.e. one replaces exterior derivative $d$ with a suitable curl operator).  The vorticity equation then becomes the familiar equation $\partial_t \omega + (u \cdot \nabla) \omega = 0$ (in the two-dimensional case) or $\partial_t \omega + (u \cdot \nabla) \omega = (\omega \cdot \nabla) u$ (in the three-dimensional case).  However, we will adopt a more differential-geometric viewpoint in this paper by interpreting the vorticity as a $2$-form rather than a scalar or vector field.  This distinction becomes particularly important when applying Lie derivatives such as ${\mathcal L}_u$, as these derivatives act on $2$-forms in a different fashion than on scalars or vector fields (this is related to the fact that the velocity field $u$ will almost never be a Killing vector field for the Euclidean metric $\eta$, so that ${\mathcal L}_u \eta \neq 0$).  Interpreting the vorticity as a $2$-form will also make it easier to change to curvilinear coordinate systems, such as cylindrical coordinates, as we will do in Section \ref{period-sec}.
\end{remark}

The velocity field $u$ can be (formally) recovered from the vorticity $\omega$ by the \emph{Biot-Savart law}
$$ u = \delta \tilde \eta^{-1} \Delta^{-1} \omega$$
where the \emph{Hodge Laplacian} $\Delta\colon \Lambda_k(\M) \to \Lambda_k(\M)$ is given by
$$\Delta \coloneqq d \tilde \eta \delta \tilde \eta^{-1} + \tilde \eta \delta \tilde \eta^{-1} d,$$
which in the Euclidean metric coordinates simplifies to the familiar formula\footnote{Note here the negative sign in our definition of the Laplacian, which differs from the usual conventions for the Laplacian in the Euler equation literature.  In particular, our Laplacian $\Delta$ will be positive semi-definite rather than negative semi-definite.  In the differential geometry literature it is common to refer to $\tilde \eta \delta \tilde \eta^{-1}$ rather than $\delta$ as the codifferential, so that $\Delta = d \delta + \delta d$ in this notation; however we prefer in this paper to make the dependence on the metric $\eta$ more explicit.} $\Delta = - \eta^{ij} \partial_i \partial_j$.  Note that $\Delta$ preserves $B_k(\M)$, and so the inverse operator $\Delta^{-1}$ does so also (formally, at least).  We make the technical remark that when $m=2$, the operator $\Delta^{-1}$ is only well defined up to constants, even when applied to forms that are smooth and compactly supported, unless one arbitrarily fixes a convention for defining $\Delta^{-1}$.  However this will not be a major issue in practice because the operator $\delta \tilde\eta^{-1} \Delta^{-1}$ will remain canonically defined.

By slight abuse of notation, we refer to the $2$-vector field $\tilde \eta^{-1} \Delta^{-1} \omega(t)$ as the \emph{vector potential} (also known as the \emph{stream function} in the two-dimensional case $d=2$), and refer to the operator $\tilde \eta^{-1} \Delta^{-1}\colon B_2(\M) \to \Gamma^2(\M)$ as the \emph{vector potential operator} for the Euler equations.  We observe that the vector potential operator $\tilde \eta^{-1} \Delta^{-1}$ is \emph{formally self-adjoint} in the sense that we have the (formal) integration by parts identity
$$ \int_{\R^d} \langle \omega, \tilde \eta^{-1} \Delta^{-1} \omega' \rangle\ d\operatorname{vol}
= \int_{\R^d} \langle \omega', \tilde \eta^{-1} \Delta^{-1} \omega \rangle\ d\operatorname{vol}$$
for $\omega, \omega' \in B_2(\M)$.

We refer to the system
\begin{align*}
 \partial_t \omega + {\mathcal L}_u \omega &= 0 \\
u &= \delta \tilde \eta^{-1} \Delta^{-1} \omega
\end{align*}
as the \emph{vorticity-vector potential formulation} of the Euler equations.  We now generalise this system to other choices of vector potential operator:

\begin{definition}[Generalised Euler equations]\label{gee}  Let $\M = \R^m \times (\R/\Z)^{d-m}$ for some $d \geq 2$ and $0 \leq m \leq d$, and let $A\colon B_2(\M) \to \Gamma^2(\M)$ be a (formal) linear operator from the space of closed $2$-forms to the space of $2$-vector fields.  We refer to the (formal) system of equations
\begin{align}
 \partial_t \omega + {\mathcal L}_u \omega &= 0 \label{vort-1}\\
u &= \delta A \omega,\label{vort-2}
\end{align}
where $\omega(t) \in B_2(\M)$ and $u(t) \in B^1(\M)$ for each time $t$, as the \emph{generalised Euler equations} with vector potential operator $A$.  We say that the vector potential operator $A$ is \emph{formally self-adjoint} if one formally has
\begin{equation}\label{fsa}
 \int_{\M} \langle \omega, A \omega' \rangle\ d\operatorname{vol}
= \int_{\M} \langle \omega', A \omega \rangle\ d\operatorname{vol}
\end{equation}
for all $\omega,\omega' \in B_2(\M)$.
\end{definition}

The vorticity-vector potential formulation of the Euler equations (which we will now call the \emph{true Euler equations} for emphasis) are thus the generalised Euler equations associated to the vector potential operator 
\begin{equation}\label{asper}
\tilde \eta^{-1} \Delta^{-1}.
\end{equation}
Another example of a system that can be (formally) written as the generalised Euler equation is the (inviscid) \emph{surface quasi-geostrophic (SQG)} equations
\begin{align*}
\partial_t \theta + (u \cdot \nabla) \theta &= 0 \\
u &= (-\partial_2 \Delta^{-1/2}, \partial_1 \Delta^{-1/2}) \theta
\end{align*}
in two spatial dimensions $d=2$, 
where $\theta\colon \R \times \M \to \R$ is a scalar field and $u\colon \R \times \M \to \R^2$ is a vector field.  This equation arises in atmospheric science and can be considered as a toy model for the three-dimensional Euler equations; see \cite{cmt} for further discussion.  If we set
$$ \omega \coloneqq \theta d\operatorname{vol} = \theta dx_1 \wedge dx_2$$
and define the vector potential operator $A\colon B_2(\M) \to \Gamma^2(\M)$ by $A \coloneqq \Delta^{-1/2}$, or in coordinates	
\begin{equation}\label{sqg-a}
 A( \theta dx^1 \wedge dx^2 ) \coloneqq \Delta^{-1/2} \theta \frac{d}{d x^1} \wedge \frac{d}{d x^2} 
\end{equation}
then we see that the SQG equations become the generalised Euler equations in two dimensions with the choice \eqref{sqg-a} of vector potential operator.  Later, in Section \ref{form}, we will give an alternate way of interpreting SQG as a generalised Euler equation, this time in three dimensions, and with a vector potential operator of order $-2$ (like $\tilde \eta^{-1} \Delta^{-1}$).  

\begin{remark}
The \emph{modified SQG} equations, in which the exponent $-1/2$ appearing in \eqref{sqg-a} is replaced by $-\alpha/2$ for some parameter $\alpha$ between $1$ and $2$, is a family of interpolants between SQG and the two-dimensional Euler equations which have also been studied in the literature; see e.g. \cite{kn}.  However, we will not study these equations further in this paper, though we will note the recent paper \cite{kryz} in which finite time blowup was established for patch solutions to the generalised SQG equations in a half-plane.
\end{remark}

\begin{remark}  The formalism in Definition \ref{gee} does not directly use the Euclidean metric $\eta$ on $\M$; one only needs the structure $(\M, d\operatorname{vol})$ of $\M$ as a smooth manifold equipped with a volume form $d\operatorname{vol}$ (in order to define the codifferential $\delta$).  However, when one works with the true Euler equations, the Euclidean metric $\eta$ is needed to define the vector potential operator $A = \tilde \eta^{-1} \Delta^{-1}$.  Thus we see that the role of Euclidean geometry (beyond the volume form) in the true Euler equations has been completely captured in this formalism by the operator $A$.
\end{remark}

\begin{remark}\label{covelocity}  One can rewrite the generalised Euler equations in a form resembling the traditional form \eqref{euler-eq} of the true Euler equations by formally defining the covelocity $v \in \Lambda_1(\M)$ to solve the system
$$ dv = \omega; \quad \delta \tilde \eta^{-1} v = 0$$
and then the generalised Euler equations may be rewritten as
\begin{align*}
\partial_t v + {\mathcal L}_u v &= d\tilde p \\
u &= \delta \tilde \eta^{-1} A d v \\
\delta \tilde \eta^{-1} v &= 0.
\end{align*}
\end{remark}

The generalised Euler equations (formally) obey many of the conservation laws that the true Euler equations do, particularly if the vector potential operator $A$ is formally self-adjoint and commutes with a suitable symmetry.  More precisely, we have

\begin{proposition}[Formal conservation laws]\label{Fcl}  Let $\M = \R^m \times (\R/\Z)^{d-m}$ for some $d \geq 2$ and $0 \leq m \leq d$, and let $A\colon B_2(\M) \to \Gamma^2(\M)$ be a (formal) linear operator.  Let $\omega, u$ solve the generalised Euler equations with vector potential operator $A$.
\begin{itemize}
\item[(i)] (Kelvin circulation theorem)  If $S = S(t)$ is a (time-dependent, oriented) surface with boundary that evolves along the (time-dependent) velocity field $u = u(t)$, then the quantity\footnote{In the case of the true Euler equations, this quantity $\int_S \omega$ can be expressed via Stokes' theorem as $\int_{\partial S} \tilde \eta u$, which is the physical circulation of velocity along the boundary $\partial S$ of $S$.  For the generalised Euler equations, this quantity is not quite the physical circulation, but is instead the quantity $\int_{\partial S} v$ where $v$ is the covelocity from Remark \ref{covelocity}. Nevertheless we shall abuse notation and refer to the quantity $\int_S \omega$ as the \emph{circulation} around the surface $S$.  We thank Peter Constantin for pointing out this subtle distinction between the circulation conserved by Kelvin's theorem and physical circulation in the context of generalised Euler equations.} $\int_S \omega$ is formally conserved in time.
\item[(ii)]  (Preservation of vortex streamlines)  If $d=3$, then the curves formed by integrating the vector field $*^{-1}\omega(t) \in \Gamma^1(\M)$ (i.e., the vortex streamlines) are transported by the velocity field $u$.
\item[(iii)]  (Conservation of helicity)  If $\M = \R^3$, define the helicity $H(t)$ to be the quantity $H(t) \coloneqq \int_{\M} v(t) \wedge \omega(t)$, where $v(t) \in \Lambda_1(\M)$ is an arbitrary $1$-form with $dv = \omega$; observe from Stokes' theorem that this quantity does not depend on the choice of $v$.  Then $H$ is formally conserved in time.
\item[(iv)]  (Conservation of Hamiltonian) Suppose $A$ is formally self-adjoint.  Define the energy $E(t)$ to be the quantity 
\begin{equation}\label{eht}
E(t) \coloneqq \frac{1}{2} \int_{\M} \langle \omega, A \omega \rangle\ d \operatorname{vol}.
\end{equation}
Then $E$ is formally conserved in time.
\item[(v)]  (Conservation of impulse)  Suppose $A$ is formally self-adjoint.  Let $X \in B^1(\M)$ be a (time-independent) divergence-free vector field such that the Lie derivative ${\mathcal L}_X$ commutes with $A$: ${\mathcal L}_X A = A {\mathcal L}_X$.  Suppose that $\alpha \in \Gamma^2(\M)$ is a time-independent $2$-vector field such that $\delta \alpha = X$.  Then the quantity $\int_{\M} \langle \omega, \alpha \rangle\ d\operatorname{vol}$ is formally conserved in time.
\end{itemize}
\end{proposition}

In the case of the true Euler equations with $\M = \R^d$, examples of (formal) conservation laws arising from Proposition \ref{form}(v) include the total vorticity
$$ \Omega_{ij} \coloneqq \int_{\M} \omega_{ij}\ d\operatorname{vol}$$
(corresponding to the zero vector field), the impulse
$$ I_j \coloneqq \frac{-1}{d-1} \int_{\M} x^i \omega_{ij}\ d\operatorname{vol}$$
(corresponding to the translation vector fields $\frac{d}{d x^j}$), and the moment of impulse
$$ M_{jk} \coloneqq \frac{-1}{d} \int_{\M} \eta_{lk} x^i x^l \omega_{ij} - \eta_{lj} x^i x^l \omega_{ik}\ d\operatorname{vol}$$
(corresponding to the rotation vector field $x^i \frac{d}{d x^j} - x^j \frac{d}{d x^i}$).  If the velocity field $u$ has sufficient decay, then $\Omega_{ij}$ vanishes, $I_j$ is equal to the total momentum $\int_{\M} u^j\ d\operatorname{vol}$ (after contracting by $\eta$), and $M_{jk}$ is equal to the total angular momentum $\int_{\M} (x^j u^k - x^k u^j)\ d\operatorname{vol}$ (again after contracting by $\eta$); however, the Biot-Savart law does not always give enough decay on $u$ to justify these computations, even when $\omega$ is smooth and compactly supported.  See \cite{majda}, \cite{shankar} for further discussion of these conservation laws.

We prove Proposition \ref{Fcl} in Section \ref{fcl-sec} by direct computation, relying mostly on the standard properties of the Lie and exterior derivatives.  One can also interpret these conservation laws as instances of Noether's theorem, using the Euler-Poincar\'e interpretation of generalised Euler equations (in the spirit of \cite{arnold}) as formal geodesic flow in the infinite-dimensional manifold of volume-preserving diffeomorphisms of $\M$, with the vector potential operator $A$ determining the formal (right-)invariant Riemannian metric to place on this manifold; see Section \ref{lag-sec}.  It is certainly possible to prove rigorous versions of Proposition \ref{Fcl} assuming sufficient regularity and decay of the solution (and assuming that there are no cohomological obstructions), but we will not need to do so here (except for Proposition \ref{Fcl}(i), which is used to prove Theorem \ref{first-blow} below).

\subsection{Local existence theory}

Thus far, all of our discussion has been purely formal, ignoring all requirements of decay and smoothness.  We now turn to the rigorous existence theory of the generalised Euler equations.  For this we will need to place the fields $u$ and $\omega$ in appropriate (high regularity) function spaces; we will also now work in coordinates, abandoning any pretense of coordinate invariance.  As mentioned previously, we will assume that the number $m$ of non-compact directions is at least two, in order to avoid problems with defining the inverse $\Delta^{-1}$ of the Hodge Laplacian.

For any $1 \leq p \leq \infty$, we let $L^p \cap \Lambda_k(\M)$ be the space of $k$-forms that are $p^{\operatorname{th}}$ power integrable, with the usual norm
$$ \| \omega \|_{L^p(\M)} \coloneqq \left(\int_\M |\omega|^p\ d\operatorname{vol}\right)^{1/p}$$
with the usual Euclidean norm on tensors to define $|\omega|$, and with the usual modifications for $p=\infty$.  Similarly define $L^p \cap \Gamma^k(\M)$, $L^p \cap B^k(\M)$, and $L^p \cap B_k(\M)$, where in the latter two cases we interpret the differential operators $d,\delta$ in the distributional sense, thus for instance $L^p \cap B^k(\M)$ consists of those $\alpha$ in $\Gamma^k(\M)$ that are $p^{\operatorname{th}}$ power integrable with $\delta \alpha = 0$ in the sense of distributions.  For any $k \geq 0$ and $s \in \R$, we define $\dot H^s \cap \Lambda_k(\M)$ to be the space of tempered distributional $k$-forms $\omega \in \Lambda_k(\M)$ whose (tempered distributional) Fourier transform
$$ \hat \omega(\xi) \coloneqq \int_{\M} \omega(x) e^{-2 \pi i x \cdot \xi}\ dx$$
(computed by working in the standard coordinate basis and taking the tempered distributional Fourier transform of each component of $\omega$ separately) is such that $|\xi|^s \hat \omega$ is square-integrable, thus
$$ \|\omega\|_{\dot H^s(\M)}^2 \coloneqq \int_{\hat \M} |2 \pi \xi|^{2s} |\hat \omega(\xi)|^2\ d\xi < \infty$$
(here we use the standard Euclidean norm on tensors to define $|\hat \omega(\xi)|$, and $\hat \M = \R^m \times \Z^{d-m}$ denotes the Pontryagin dual of $\M$).  The factors of $2\pi$ are of very minor importance and can be ignored for a first reading.
The space $\dot H^s \cap \Lambda_k(\M)$ can be easily verified to be a Hilbert space.  Similarly define $\dot H^s \cap \Gamma^k(\M)$, $\dot H^s \cap B_k(\M)$, and $\dot H^s \cap B^k(\M)$.  As is usual, we write $H^s \coloneqq L^2 \cap \dot H^s$ (thus for instance $H^s \cap \Lambda_k = L^2 \cap \dot H^s \cap \Lambda_k$) and
$$ \| \omega \|_{H^s(\M)}^2 \coloneqq \| \omega \|_{L^2(\M)}^2 + \| \omega \|_{\dot H^s(\M)}^2.$$
We also define $C^\infty_c \cap \Lambda_k(\M)$ to be the space of $k$-forms that are smooth and compactly supported, and similarly define $C^\infty_c \cap \Gamma^k(\M)$, etc..

Fix an integer $s > \frac{d}{2}+1$, and let $1 < p \leq 2$ be an exponent\footnote{In particular, when $m \geq 3$ we can take $p=2$, which simplifies some of the discussion below.  On the other hand, these hypotheses are not satisfiable if $m=0$ or $m=1$.  The reason we need the $L^p$ integrability for the vorticity $\omega$ is in order to make sense of the velocity $u = \delta A \omega$ as a continuous function, and not merely as a distribution.}
 with $p < m$, where we recall that $m$ is the number of non-compact dimensions in $\M$.
It turns out that a convenient space to place the vorticity field $\omega(t)$ for a given time $t$ is
$$ L^p \cap H^s \cap B_2(\M)$$
More precisely, to construct solutions on the time interval $[0,T]$, we will place $\omega$ in the space
\begin{equation}\label{X-def}
 X^{s,p} \coloneqq C^0( [0,T] \to L^p \cap H^s \cap B_2(\M) ) \cap C^1( [0,T] \to H^{s-1} \cap B_2(\M) ),
\end{equation}
thus the map $t \mapsto \omega(t)$ will be required to be a continuous map into $L^p \cap H^s \cap B_2(\M)$, and a continuously differentiable map into $H^{s-1} \cap B_2(\M)$, where we of course give $L^p \cap H^s \cap B_2(\M)$ the topology generated by the $L^p$ and $H^s$ norms, and similarly for $H^{s-1} \cap B_2(\M)$.  Note from Sobolev embedding and the hypothesis $s > \frac{d}{2}+1$ that this implies that $\omega \in C^1_{t,x}([0,T] \times \M)$.  Similarly, we will place the velocity field $u$ in the space
\begin{equation}\label{Y-def}
 Y^{s,p} \coloneqq C^0( [0,T] \to \dot W^{1,p} \cap \dot H^{s+1} \cap B^1(\M) ),
\end{equation}
where $\dot W^{1,p}$ is the Sobolev space of functions (or vector fields, in this case) whose distributional derivative lies in $L^p$; for technical reasons relating to the slow decay of the Biot-Savart law (and its generalisations) at infinity, we do not insist that $u$ itself lies in $L^2$ or $L^p$.  Note that the hypothesis $u \in Y^{s,p}$ and Sobolev embedding\footnote{More precisely, observe from Bernstein's inequality and the hypothesis $p<m$ that $\dot W^{1,p}$ embeds into $C^1_x$ at low frequencies, and $\dot H^{s+1}$ embeds into $C^1_x$ at high frequencies.} implies that $u \in C^0_t C^1_x([0,T] \times \M)$.  This is sufficient regularity to interpret the equation \eqref{vort-1} in the classical sense, as a ``strong'' solution rather than merely a weak distributional solution.

To interpret \eqref{vort-2}, we will of course need some regularity hypotheses on the operator $A\colon B_2(\M) \to \Gamma^2(\M)$.  We will adopt the following choice of hypotheses.  We use $X \lesssim Y$ or $X=O(Y)$ to denote the estimate $|X| \leq CY$, where $C$ is a constant; if $C$ is to depend on one or more parameters, we indicate this by subscripting the $\lesssim$ or $O()$ notation appropriately.

\begin{definition}[Reasonable operator]\label{reason-def}  Let $M$ be a natural number, and let $\M = \R^m \times (\R/\Z)^{d-m}$ for some $2 \leq m \leq d$.  A vector potential operator $A\colon C^\infty_c \cap B_2(\M) \to \Gamma^2(\M)$ is said to be \emph{$M$-reasonable} if it has an integral representation 
\begin{equation}\label{kernel-rep}
 A \omega(x) = \int_{\M} K(x,y) \omega(y)\ d\operatorname{vol}(y)
\end{equation}
where the (tensor-valued) kernel $K$ is smooth for $x \neq y$ and obeys the estimates
\begin{equation}\label{nij}
 |\nabla^i_x \nabla^j_y K(x,y)| \lesssim_{d,M,A} \max( |x-y|^{-i-j-d+2}, |x-y|^{-i-j-m+2} )
\end{equation}
for all $x \neq y$ and all $0 \leq i,j \leq M$ with $i+j \geq 1$, where $|x-y|$ denotes the distance between $x$ and $y$ in $\M$ with respect to the Euclidean metric $\eta$; furthermore we assume that
\begin{equation}\label{pass}
\| \nabla^2 A \omega \|_{H^{k}(\M)} \lesssim_{d,M,A} \| \omega \|_{H^k(\M)}
\end{equation}
for all $0 \leq k \leq M$ and all $\omega \in C^\infty_c \cap B_2(\M)$.  In particular, $\delta A$ can be continuously extended to a map from $L^2 \cap B_2(\M)$ to $\dot H^1 \cap \Gamma^1(\M)$.
\end{definition}

\begin{remark}
The right-hand side of \eqref{nij} has the geometric interpretation of being comparable to $\frac{1}{|x-y|^{i+j-2} \operatorname{vol}( B_\M( 0, |x-y| ) )}$, where $\operatorname{vol}( B_\M( 0, |x-y| ) )$ is the volume of the ball in $\M$ centred at the origin with radius $|x-y|$.  In particular, the operator $\nabla^2 A$ is a singular integral operator whose kernel obeys estimates of Calder\'on-Zygmund type, which is of course consistent with the hypothesis \eqref{pass}.
\end{remark}

From Plancherel's theorem and the fundamental solution for the Laplacian on $\M$, we see that the vector potential operator $\tilde \eta^{-1} \Delta^{-1}$ associated to the true Euler equations obeys these requirements whenever $m \geq 2$.  On the other hand, the vector potential operator \eqref{sqg-a} associated to the SQG equations do not, as in this case $A$ is only smoothing of order $1$ rather than $2$.  With the assumption that $A$ is $M$-reasonable for some sufficiently large $M$, one can now interpret \eqref{vort-2} rigorously when $\omega \in X$ and $u \in Y$.

Using mostly standard ``quasilinear well-posedness'' energy methods (following the basic approach of Bona and Smith \cite{bona}, as described in the survey \cite{tzvetkov}), we can prove the following classical local existence theorem:

\begin{theorem}[Local existence]\label{lest}  Let $\M = \R^m \times (\R/\Z)^{d-m}$ for some $d \geq 2$ and $0 \leq m \leq d$.  Let $s > \frac{d}{2}+1$ be an integer, let $1 < p \leq 2$ be such that $p<m$, and let $A\colon C^\infty_c \cap B_2(\M) \to \dot H^1 \cap \Gamma^2(\M)$ be a $s+1$-reasonable vector potential operator.  Then for any $M > 0$ there exists $T>0$ such that for any $\omega_0 \in L^p \cap H^s \cap B_2(\M)$ with $\|\omega_0\|_{L^p(\M)} + \| \omega_0 \|_{H^s(\M)} < M$, there exists a unique classical solution $\omega \in X^{s,p}$ and $u \in Y^{s,p}$ (with $X^{s,p},Y^{s,p}$ defined in \eqref{X-def}, \eqref{Y-def} respectively) obeying the generalised Euler equations \eqref{vort-1}, \eqref{vort-2}.  Furthermore the solution $\omega$ depends continuously on $\omega_0$ in the indicated topologies.

Finally, we have the Beale-Kato-Majda blowup criterion \cite{bkm}: if the solution constructed above cannot be continued beyond a time $0 < T_* < \infty$ in the indicated function spaces, then
$$ \int_0^{T_*} \| \omega(t) \|_{L^\infty(\M)}\ dt = \infty.$$
\end{theorem}

We prove this theorem in Section \ref{lest-sec}.  The argument is straightforward when $m \geq 3$, in which case the $L^p$ norm plays no essential role.  However, the situation becomes delicate in the $m=2$ case, basically because the generalised Biot-Savart operator  $\delta A$ that appears in the vorticity-vector potential formulation no longer maps $H^s$ into $L^\infty$ at low frequencies, and one must take advantage of the $L^p$ norm and Littlewood-Paley decomposition to close the argument.  A slightly different energy method approach to these equations is also given in \cite[Chapter 3]{majda}.  There is also a particle trajectory method to construct solutions to the true Euler equations using the contraction mapping theorem rather than quasilinear method; see e.g. \cite[Chapter 4]{majda}.  However, we were unable to extend it to this general context unless one imposed a translation-invariance hypothesis on the vector potential operator $A$, as the estimates required for the contraction mapping theorem appeared to fail if this hypothesis was not enforced.  It may also be possible to extend the local existence arguments in \cite{ebin-marsden} for the true Euler equations, based on the aforementioned interpretation of these equations as a geodesic flow, to the setting of the generalised Euler equations.

\begin{remark}
There are several refinements of the Beale-Kato-Majda blowup criterion in the literature \cite{cfm}, \cite{chae-blow}, \cite{dhy}, \cite{preston}.  It seems likely to the author that analogues of at least some of these criteria can also be established for the generalised Euler equations (since the generalised Biot-Savart law obeys very similar estimates to the true Biot-Savart law), although we have not attempted to do so here.
\end{remark}

\subsection{Finite time blowup}  We now turn to the main focus of this paper, namely the establishment of finite time blowup results for generalised Euler equations.

It is a notorious open problem as to whether smooth solutions to the three-dimensional true Euler equations (with suitable decay at infinity) can be extended globally in time, although it is widely expected that finite time blowup can occur for this system; see for instance the surveys \cite{chae}, \cite{constantin} and recent numerical evidence for blowup in \cite{lh, lh2}, as well as a proposed blowup mechanism in \cite{brenner}.

As the global regularity problem for the true Euler equations is difficult to resolve directly, there have been a number of studies of more tractable models of the Euler equations.  In particular, finite time blowup has been established for a number of equations that capture some, though not all, of the features of the true Euler-type equations.  For instance:

\begin{itemize}
\item In \cite{kp} a dyadic ``shell model'' of the Euler equations was introduced, and shown to have solutions that blow up in finite time; see also \cite{tao} for a variant of this construction that allows for Navier-Stokes type dissipation.  These shell models have the same scaling features as the true Euler equations in three dimensions, as well as energy conservation, but do not have the vorticity transport equation.  
\item In \cite{clm}, a one dimensional model for the vorticity equation of the true Euler equations was introduced, and again shown to have solutions that blow up in finite time; see also the later papers \cite{greg1}, \cite{greg2}, \cite{osw}, \cite{w} for further analysis of this model and its variants.  These equations capture many of the features of the Euler equations, such as energy conservation, vorticity stretching and an Euler-Poincar\'e Lagrangian formulation, but do not correspond to incompressible flows (the formal Euler-Poincar\'e geodesic flow is on the space of all diffeomorphisms of a manifold, rather than all volume-preserving diffeomorphisms).  
\item In \cite{hl}, \cite{hsw}, a model of the axially symmetric true Euler equations with swirl was studied in which the convection term was removed, and solutions constructed that blow up in finite time.  This system of equations still conserves energy, but does not appear to have an Euler-Poincar\'e formulation, or a vorticity transport equation analogous to \eqref{vort}.  
\item Further 1D models of axially symmetric true Euler equations outside of a cylindrical obstacle were studied in \cite{cky}, \cite{hl}, \cite{chklsy}, again with a number of finite time blowup results; these systems have some remnant of circulation conservation (through the transport of the ``temperature'' field $\theta$), but do not appear to have an Euler-Poincar\'e formulation that involves an incompressible flow.
\end{itemize}

In this paper we establish some finite time blowup results in three spatial dimensions for generalised Euler equations, with reasonable vector potential operator $A$.  In order to maximise the resemblance of these generalised Euler equations to the true Euler equations, it is desirable to ensure that $A$ be formally self-adjoint, and for $A$ to furthermore be ``positive definite'' in the sense that the conserved energy \eqref{eht} to be comparable to $\| \omega \|_{\dot H^{-1}}^2$ (or to $\| u \|_{L^2}^2$).  It would also be desirable to construct blowup solutions that are well localised in space, for instance by requiring the initial vorticity to be compactly supported.  Finally, one would like to demonstrate some stability in the blowup, by showing that blowup persists under some reasonable perturbation of the initial data.

Unfortunately, we were not able to construct a blowup solution in which all of these desirable criteria were satisfied simultaneously.  However, we were able produce three different constructions which enjoyed various subsets of this set of desiderata. Taken together, they suggest that one should not be able to establish global regularity properties for the true Euler equations merely by using properties that are shared with the generalised Euler equations, such as energy conservation, the Kelvin circulation theorem, and function space estimates for the vector potential operator.

Our first construction has compactly supported initial data (and a stable blowup), but a non-self-adjoint (and non-positive definite) vector potential operator $A$:

\begin{theorem}[Stable non-self-adjoint blowup]\label{first-blow}  Let $\M = \R^3$.  Then there exists a $100$-reasonable vector potential operator $A\colon C^\infty_c \cap B_2(\M) \to \Gamma^2(\M)$ and initial data $\omega_0 \in C^\infty_c \cap B_2(\M)$ such that there is no solution $\omega \in X^{10,2}$, $u \in Y^{10,2}$ with initial data $\omega_0$ on the time interval $[0,1]$.
\end{theorem}

We prove this result in Section \ref{3d-nonself}.  The exponents $10,100$ here have no particular significance and are chosen primarily for sake of concreteness. The blowup is probably\footnote{Because our argument will be a proof by contradiction, we will not actually be able to guarantee that the solution blows up as intended; it may blow up at an earlier time than the formation of the neck pinch due to other instabilities in the dynamics.  However, the ``neck pinch'' scenario is what the blowup \emph{should} be, if it is not pre-empted by some earlier, unforeseen blowup.  Similarly for the other finite time blowup results in this paper.} of a ``neck pinch'' nature, in which the vortex lines focus at a point (see Figure \ref{fig:firstblow}); the non-self-adjoint vector potential $A$ is designed to keep transporting the vorticity ever closer to that point.  We will not be able to achieve any fine level of control on the dynamics of this finite time blowup, but fortunately we can use the conservation of circulation, combined with a careful choice of $A$, to evaluate the velocity field $u$ near the blowup point and close the argument.  As can be seen from the proof, the blowup in Theorem \ref{first-blow} is stable in the sense the initial data $\omega_0$ can range in an open set in $C^\infty_c \cap B_2$; any smooth closed perturbation of the data supported in a slight enlargment of the support of $\omega_0$ will still lead to a solution that blows up in finite time.

\begin{figure} [t]
\centering
\includegraphics{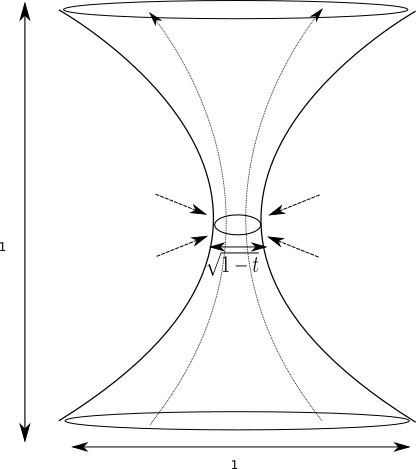}
\caption{A schematic depiction of a  ``neck pinch'' blowup of the type expected from the construction in Theorem \ref{first-blow} (ignoring a technical parameter $M$ appearing in the proof).  At times $t$ close to the blowup time (assumed here to be $T_*=1$), the vortex streamlines (shown here as dotted curves) travel through a truncated hyperboloid region that is of unit diameter at its ends, but ``pinches'' through a narrow disk of radius comparable to $\sqrt{1-t}$, so that the vorticity increases to be comparable to $\frac{1}{1-t}$ near this disk, as per the Kelvin circulation theorem.  (The vortex streamlines may connect back to themselves outside of this hyperboloid region; this is not pictured in the figure.)  The velocity field, depicted here as dashed arrows, points inwards with magnitude comparable to $\frac{1}{\sqrt{1-t}}$ in the pinching region.  At time $t=1$, the vorticity becomes infinite at a point, causing blowup.}
\label{fig:firstblow}
\end{figure}

As the initial data $\omega_0$ and operator $A$ constructed in Theorem \ref{first-blow} are compactly supported in space, it is an easy matter to extend the above theorem to $\R^2 \times \R/\Z$; by adding compact dummy dimensions one can also extend to the case $d \geq 3$ and $2 \leq m \leq 3$.  It is likely that one can in fact obtain a result of the above form for arbitrary $d \geq 3$ and $2 \leq m \leq d$ (increasing the exponents $10$ and $100$ as necessary), but we will not do so here.

The blowup in Theorem \ref{first-blow} is perhaps unsurprising, given that the vector potential operator $A$ was not self-adjoint and so did not even have a conserved energy.  Our second blowup result involves a vector potential operator $A$ which is now self-adjoint and positive definite.  However, to retain compact support of the data, it becomes convenient to work in the domain $\R^2 \times \R/\Z$ rather than $\R^3$; also, the blowup is less stable, as we require the initial data to be translation-invariant in the $\R/\Z$ direction (in order to reduce matters to a two-dimensional problem).

\begin{theorem}[Self-adjoint partially periodic blowup]\label{second-blow}  Let $\M = \R^2 \times \R/\Z$ and let $\eps>0$.  Then there exists a $100$-reasonable, formally self-adjoint vector potential operator $A\colon C^\infty_c \cap B_2(\M) \to \Gamma^2(\M)$ obeying the positive definiteness property
\begin{equation}\label{posdef}
(1-\eps) \| \omega \|_{\dot H^{-1}(\M)}^2 \leq \int_{\M} \langle \omega, A \omega \rangle\ d \operatorname{vol}
\leq (1+\eps) \| \omega \|_{\dot H^{-1}(\M)}^2 
\end{equation}
for all $\omega \in C^\infty_c \cap B_2(\M)$, as well as an initial vorticity $\omega_0 \in C^\infty_c \cap B_2(\M)$ such that there is no solution $\omega \in X^{10,2}$, $u \in Y^{10,2}$ with initial data $\omega_0$ on the time interval $[0,1]$.
\end{theorem}

We prove this result in Section \ref{embed-sec}.  The main idea is to work in a ``two-and-a-half-dimensional'' ansatz in which the velocity field $u$ and vorticity $\omega$ are invariant with respect to translations in the $x^3$ direction, with the $dx^1 \wedge dx^2$ component of the vorticity vanishing, but in which the third component $u^3$ of the velocity is allowed to be non-zero.  If the vector potential operator $A$ is chosen properly, it turns out that this component $u^3$ obeys an SQG-like active scalar equation on $\R^2$; furthermore, the vector potential operator $A_0$ for this SQG-like equation is no longer required to be self-adjoint.  It is then possible to modify the three-dimensional construction in Theorem \ref{first-blow} to create a two-dimensional blowup for this component $u^3$, which then implies blowup for the original fields $u,\omega$; in fact the two-dimensional case is a little easier than the three-dimensional one, and is carried out in Section \ref{nonself}.  Due to the dimensional reduction, the solution constructed in Theorem \ref{second-blow} will (probably) blow up on a one-dimensional set, namely a copy of $\R/\Z$ in $\R^2 \times \R/\Z$, in contrast to the solution in Theorem \ref{first-blow} which (probably) blows up at a point.  A schematic depiction of what the blowup should look like in this construction is given in Figure \ref{fig:secondblow}.
  
	\begin{figure} [t]
\centering
\includegraphics[width=0.7\textwidth]{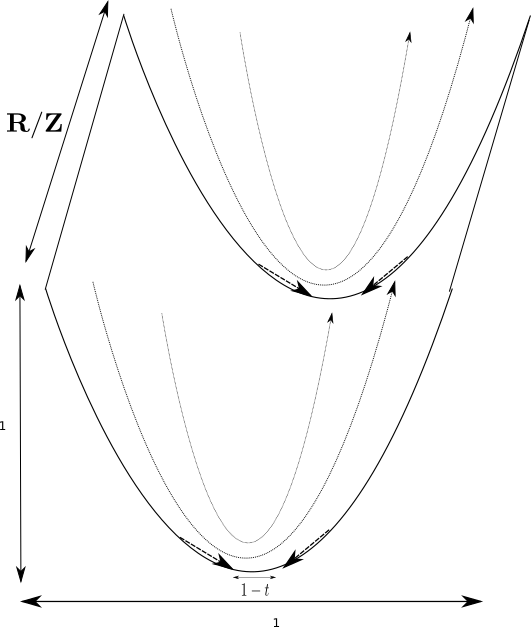}
\caption{A schematic depiction of a  ``two-and-a-half-dimensional'' blowup of the type expected from the construction in Theorem \ref{second-blow} (again ignoring a technical parameter $M$).  At times $t$ close to the blowup time (assumed here to be $T_*=1$), the vortex streamlines (shown here as dotted curves) are horizontal curves that are pinched into a narrow region of horizontal diameter about $1-t$, wherein the vorticity increases to about $\frac{1}{1-t}$ as per the Kelvin circulation theorem.  The vorticity is invariant with respect to vertical translations, which in the case of the true Euler equations would mean that the velocity field would be purely vertical (and invariant along streamlines), and the solution would be stationary.  Here, we work with a perturbation of the true Euler equations that creates some horizontal velocity in the pinching region (of magnitude comparable to $1$) that causes the vorticity to pinch further.  At time $t=1$, the vorticity becomes infinite on a vertical line (or more precisely, a copy of $\R/\Z$), causing blowup.}
\label{fig:secondblow}
\end{figure}

Finally, we remove the periodic dimension from Theorem \ref{second-blow}:

\begin{theorem}[Self-adjoint non-periodic blowup]\label{third-blow}  Let $\M = \R^3$ and let $\eps>0$.  Then there exists a $100$-reasonable, formally self-adjoint vector potential operator $A\colon C^\infty_c \cap B_2(\M) \to \Gamma^2(\M)$ obeying the positive definiteness property
\begin{equation}\label{posdef-2}
(1-\eps) \| \omega \|_{\dot H^{-1}(\M)}^2 \leq \int_{\M} \langle \omega, A \omega \rangle\ d \operatorname{vol}
\leq (1+\eps) \| \omega \|_{\dot H^{-1}(\M)}^2 
\end{equation}
for all $\omega \in C^\infty_c \cap B_2(\M)$, as well as an initial vorticity $\omega_0 \in C^\infty_c \cap B_2(\M)$ such that there is no solution $\omega \in X^{10,2}$, $u \in Y^{10,2}$ with initial data $\omega_0$ on the time interval $[0,1]$.
\end{theorem}

We will prove Theorem \ref{third-blow} in Section \ref{period-sec}; it will essentially be deduced from Theorem \ref{second-blow} by embedding $\R^2 \times \R/\Z$ into $\R^3$ using (modified) cylindrical coordinates.  The resulting dynamics resembles that of axisymmetric Euler equations with swirl (particularly when viewed in the coordinates used in \cite{benjamin}, \cite{turkington}.  For the true Euler equations, the assumption of axial symmetry does not completely reduce matters to an active scalar equation (in contrast to the situation with an assumption of translation symmetry, as used in the proof of Theorem \ref{second-blow}) due to the non-constant-coefficient nature of the metric $\eta$ in cylindrical coordinates; however, using the freedom to select the vector potential $A$, we can replace $\eta$ locally with a nearby metric which is constant coefficient in cylindrical coordinates on the support of $\omega$, at which point one can adapt the argument used to prove Theorem \ref{second-blow}.  Due to the use of cylindrical coordinates, the solution should now blow up on a circle; see Figure \ref{fig:thirdblow}.

	\begin{figure} [t]
\centering
\includegraphics[width=\textwidth]{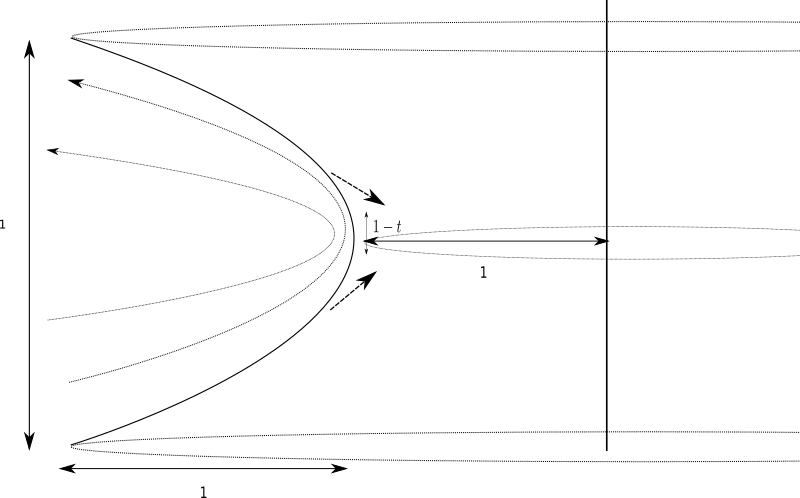}
\caption{A schematic depiction of a  ``axisymmetric with swirl'' blowup of the type expected from the construction in Theorem \ref{third-blow}, which is essentially the same blowup as Theorem \ref{second-blow} if the Cartesian coordinates were replaced with cylindrical ones.  At times $t$ close to the blowup time (assumed here to be $T_*=1$), the vortex streamlines (shown here as dotted curves) are curves coplanar with the vertical axis (drawn here as a thick line) that are pinched into a toroidal region that is roughly a $1-t$-neighbourhood of a circle of radius comparable to $1$.  Within this region, the vorticity increases to about $\frac{1}{1-t}$ as per the Kelvin circulation theorem.  The vorticity and velocity fields are axially symmetric; the latter has magnitude comparable to $1$ in the toroidal region, with significant inward components as well as some ``swirl'' around the vertical axis.  At time $t=1$, the vorticity becomes infinite on a circle, causing blowup.}
\label{fig:thirdblow}
\end{figure}

\begin{remark}  Another potential type of blowup scenario would be a self-similar blowup (here one would need the vector potential operator $A$ to commute with spatial dilations in an appropriate fashion).  In the case of the true Euler equations, the arguments of Chae \cite{chae-ss07, chae-ss10, chae-ss11} preclude non-trivial self-similar solutions in which the vorticity decays rapidly at infinity; however, this leaves open the possibility of non-trivial self-similar solutions that decay only slowly at infinity.  However, we were not able to construct such solutions while keeping $A$ self-adjoint and positive semi-definite.
\end{remark}

\begin{remark} There are numerous issues preventing one from adapting these generalised Euler equation blowup results to the case of the generalised Navier-Stokes systems (in the spirit of \cite{tao}).  The most proximate issue is that all of the results rely in one way or another on the conservation of circulation (Proposition \ref{Fcl}(i)), which is no longer true for Navier-Stokes type equations.  However, even if one could obtain bounds on circulation for Navier-Stokes that were of the same order as what one obtains for Euler, the blowup results in Theorem \ref{second-blow} and Theorem \ref{third-blow} would still create solutions that presumably blow up on a one-dimensional set, which cannot occur for Navier-Stokes equations by the well known results of Caffarelli, Kohn, and Nirenberg \cite{ckn}.   The blowup result in Theorem \ref{first-blow}, which concentrates only at a point, avoids this problem; however, the scaling is still (barely) unfavorable due to the assumption of bounded circulation (which turns out to be a stronger condition, from the perspective of scaling analysis, than conservation of energy).  The numerology is as follows.  At a time $t$ close to the blowup time $T_*$, one expects the vortex lines to pinch in a disk of radius comparable to $\sqrt{T_*-t}$ (see Figure \ref{fig:firstblow}).  As this disk has area comparable to $T_*-t$, conservation (or at least boundedness) of circulation suggests that the vorticity $\omega$ is of size about $\frac{1}{T_*-t}$ on this disk, which corresponds on the level of scaling to a velocity comparable to $\frac{1}{\sqrt{T_*-t}}$.  Thus far the numerology is self-consistent, but in the case of Navier-Stokes, a viscosity term $\Delta u$ in \eqref{euler-eq} would now be expected to be comparable to $\frac{1}{(T_*-t)^{3/2}}$, which is also the order of the transport term $(u \cdot \nabla) u$.  Thus we expect the viscosity effects to be comparable to the nonlinear effects, creating a ``critical'' scenario (analogous to two-dimensional Navier-Stokes) which leads one to predict that the blowup scenario will not occur.  It may be possible to still obtain blowup by weakening the viscosity term to something like $\Delta^\alpha u$ for some $0 < \alpha < 1$, but with the full viscosity term $\Delta u$, it appears that this blowup scenario can only occur if either the viscosity somehow causes a significant increase in circulation, or if there is a lot of ``looping'' of the vortex lines that causes the circulation through a small disk to become very large due to the vortex lines passing through the disk multiple times.
\end{remark}

\begin{remark} The blowup mechanisms in this paper behave a little differently from the locally (approximately) discretely self-similar solutions proposed in \cite{brenner} (which is in turn modeled on the dynamics seen in \cite{tao}).  We have already discussed the numerology of the blowup in Theorem \ref{first-blow} in the previous remark; now we discuss the numerology in Theorem \ref{second-blow} (the situation for Theorem \ref{third-blow} is similar).  At time $t$ close to the final blowup time $T_*$, one expects a particularly strong amount of vorticity (with $\omega$ comparable to $\frac{1}{T_*-t}$) in a narrow tube of radius about $T_*-t$ and sidelength $1$ (the tube will be a neighbourhood of a copy of $\R/\Z$ in $\R^2 \times \R/\Z$); see Figure \ref{fig:secondblow}.  This vortex tube will only capture a small fraction (about $T_*-t$) of the original circulation; the remainder will come from a ``wake'' of larger vortex tubes trailing this narrow tube.  The velocity field $u$ will be comparable to $1$ throughout, and will pinch the narrowest vortex tube to a line (or more precisely, a copy of $\R/\Z$) by time $T_*$.  This scenario has some resemblance to that in \cite{brenner}, which also involves increasingly narrow vortex tubes that carry less and less circulation, but which have increasingly large pointwise vorticity; however, the tubes in \cite{brenner} are not completely linear but have some curvature (and their length goes to zero as $t$ approaches $T_*$); furthermore there is a complicated dynamic in \cite{brenner} in which pairs of vortex tubes attract and deform to become vortex planes, which then destabilise back into thinner vortex tubes, which is probably not present in the blowup constructed in Theorem \ref{second-blow} (or Theorem \ref{third-blow}).  Also, the scaling exponents in \cite{brenner} are more flexible than the ones here, for instance the width of the vortex tubes in \cite{brenner} is not constrained to decay linearly in $T_*-t$, nor is the vorticity constrained to behave inversely like $\frac{1}{T_*-t}$.  We do not know how to implement the blowup scenario proposed in \cite{brenner} using a generalised Euler equation, even if one drops the requirements of self-similarity and positive definiteness for the vector potential operator.
\end{remark}

\begin{remark} The blowup scenarios here are not of the ``tube collapse'' form ruled out in \cite{cf}, in which the volume of a vortex tube locally collapses to zero; instead, only a small portion of the volume is either pinched in a neck, or creased on a line or circle.  As the vorticity is expected to blow up like $1/(T_*-t)$, the Beale-Kato-Majda blowup criterion is satisfied (as it must be).  The blowup scenarios also do not appear to be compatible with the scenarios ruled out in \cite{cfm}, either because the velocity is unbounded or because the vorticity direction is changing too rapidly.  One could presumably use these blowup scenarios as test cases for any future blowup criterion results for the true or generalised Euler equations in a similar fashion.
\end{remark}

\begin{remark}  It is also tempting to construct blowup solutions by first choosing the fields $u,\omega$ blowing up in a specified fashion (with $\omega$ being transported by $u$) and then designing a vector potential operator $A$ to solve the generalised Euler equations with these choices of fields, in the spirit of \cite{tao-nlw}, \cite{tao-high}.  This seems achievable if one does not require $A$ to be self-adjoint.  If however self-adjointness is imposed, then this creates a nonlocal-in-time constraint on the fields $u,\omega$ which makes this approach difficult.  Namely, if one writes $\omega = dv$ for some $v \in \Lambda_1(\M)$, then an integration by parts using the self-adjointness of $A$ reveals that we must have the constraint
$$ \int_\M \langle v(t), u(t')\rangle\ d\operatorname{vol} = \int_\M \langle v(t'), u(t)\rangle\ d\operatorname{vol} $$
relating the velocity $u$ to the covelocity $v$ for all times $t,t'$.  We do not know how to design suitably blowing up fields $u,\omega$ obeying this constraint as well as \eqref{vort-1} other than by solving the generalised Euler equations.
\end{remark}

\begin{remark}  We have attempted to design the generalised Euler equations so as to capture as many of the known features of the true Euler equations as possible.  However, we should mention\footnote{We thank an anonymous referee for stressing this point.} two properties of the Euler equations which are not obeyed by the generalised Euler equations, namely translation invariance and rotation invariance; the operators $A$ we construct are inhomogeneous and non-isotropic (indeed, if one were to insist on these properties as well as dilation invariance, one would essentially be restricting the class of generalised Euler equations back to the true Euler equations up to some normalising constants, as the Biot-Savart law is basically determined by these symmetries).  In particular, we do not have conservation of momentum (impulse) or angular momentum (moment of impulse) for these equations.  On the other hand, the \emph{class} of generalised Euler equations remains invariant with respect to such symmetries, in particular most of the useful estimates on solutions to the Euler or generalised Euler equations involve function space norms which are invariant with respect to translations or rotations.  Also, the conservation laws of impulse and moment of impulse are very rarely used in the local or global regularity theory for the Euler equations, so their loss does not significantly reduce the body of results that should transfer over to the generalised Euler equation setting.
\end{remark}

\subsection{Acknowledgments}

The author is supported by NSF grant DMS-1266164 and by a Simons Investigator Award.  The author thanks Nets Katz for many useful conversations and encouragement, and Peter Constantin and the anonymous referees for many helpful suggestions and comments.

\section{Formal proof of conservation laws}\label{fcl-sec}

We now prove Proposition \ref{Fcl}.  In this section all calculations will be formal, in that we do not check that all fields involved are smooth enough and exhibit sufficient decay at infinity to justify invocations of identities such as Stokes' theorem; we also assume here that all closed forms are exact.  

Let $\omega,u$ solve the generalised Euler equations with some vector potential operator $A$. For future reference we observe from the divergence-free nature of the velocity field $u$ (or equivalently, that ${\mathcal L}_u d\operatorname{vol} = 0$) we (formally) have\footnote{See e.g. \cite[\S 3.4]{aubin} for a definition of the Lie derivative on $k$-vector fields, as well as a proof of the Leibniz rule \eqref{lieb-prod}.}
\begin{equation}\label{stok}
\int_\M {\mathcal L}_u f\ d\operatorname{vol} = 0 
\end{equation}
for any $f \in \Lambda_0(\M)$.  Applying this with $f = \langle \omega,\alpha \rangle$ for any  $\omega \in \Lambda_k(\M)$ and $\alpha \in \Gamma^k(\M)$ using the Leibniz rule
\begin{equation}\label{lieb-prod}
{\mathcal L}_u \langle \omega, \alpha \rangle = \langle {\mathcal L}_u \omega, \alpha \rangle + \langle \omega, {\mathcal L}_u \alpha \rangle,
\end{equation}
we (formally) conclude the integration by parts formula
\begin{equation}\label{integ}
\int_\M \langle {\mathcal L}_u \omega, \alpha \rangle\ d\operatorname{vol}
= - \int_\M \langle  \omega, {\mathcal L}_u \alpha \rangle\ d\operatorname{vol}.
\end{equation}

The proof of the Kelvin circulation theorem (i) is standard. For each time $t$, let $\Phi(t)\colon \M \to \M$ be the diffeomorphism formed by flowing along the vector field $u$, thus $\Phi(0)$ is the identity and
$$ \partial_t \Phi(t,x) = u(t, \Phi(t,x) )$$
for all $t\in \R$ and $x \in \M$.  Using the interpretation of a Lie derivative as an infinitesimal diffeomorphism, we have
\begin{equation}\label{phia}
\partial_t (\Phi(t)^* \alpha) = \Phi(t)^* (\partial_t \alpha(t)) + \Phi^*(t)( {\mathcal L}_{u(t)} {\alpha(t)} )
\end{equation}
for any time-dependent form or vector field $\alpha$, where $\Phi(t)^*$ denotes the pullback by $\Phi(t)$.  From \eqref{vort-1} we thus see that $\Phi(t)^* \omega(t)$ is conserved in time, thus giving the \emph{Cauchy vorticity formula}
\begin{equation}\label{cauchy}
 \omega(t) = \Phi(t)_* \omega(0)
\end{equation}
where $\Phi(t)_*$ is the pushforward by $\Phi(t)$ (the inverse of $\Phi(t)^*$).  The Kelvin circulation theorem (i) then follows from the change of variables formula.

From \eqref{vort-2} and \eqref{deldel}, $u$ is divergence-free, thus by \eqref{cartan} 
\begin{equation}\label{lv}
{\mathcal L}_u d\operatorname{vol} = 0
\end{equation}
and thus by \eqref{phia} $\Phi(t)^* d\operatorname{vol}$ is conserved in time, thus $\Phi(t)$ is volume-preserving:
$$ \Phi(t)^* d\operatorname{vol} = d\operatorname{vol}.$$
Since the Hodge duality operator $*$ from Remark \ref{hodge-rem} is defined using the volume form $\operatorname{vol}$, we conclude that $*$ commutes with $\Phi(t)^*$.  In particular, we see from the Cauchy vorticity formula \eqref{cauchy} that
$$ *\omega(t) = \Phi(t)_* (*\omega(0)).$$
In three dimensions, this gives the transport (ii) of the vortex stream lines.

Now we establish (iii).  Let $v(t)$ be a time-dependent $1$-form with $dv = \omega$, then from the product rule and differentiation under the integral sign we have
$$ \partial_t H(t) = \int_{\R^3} \partial_t v \wedge \omega + v \wedge \partial_t \omega.$$
Writing $\omega = dv$ and using the Leibniz rule for the exterior derivative, we have
$$ d( v \wedge \partial_t v ) = \omega \wedge \partial_t v - v \wedge \partial_t \omega.$$
As the wedge product is commutative between $1$-forms and $2$-forms, we have $\omega \wedge \partial_t v = \partial_t v \wedge \omega$.  We conclude upon integrating and using Stokes' theorem that
$$ \partial_t H(t) = 2 \int_{\R^3} v \wedge \partial_t \omega$$
and hence by \eqref{vort-1}
\begin{equation}\label{ot}
 \partial_t H(t) = -2 \int_{\R^3} v \wedge {\mathcal L}_u dv.
\end{equation}
Recall (see e.g. \cite[\S 2.25, 3.4]{aubin}) that the Lie derivative ${\mathcal L}_u$ and the exterior derivative $d$ obey the Leibniz rules
\begin{equation}\label{leib1}
{\mathcal L}_u ( \omega \wedge \lambda ) = ({\mathcal L}_u \omega ) \wedge \lambda + \omega \wedge ({\mathcal L}_u \lambda)
\end{equation}
and
\begin{equation}\label{leib2}
d( \omega \wedge \lambda) = (d\omega) \wedge \lambda + (-1)^k \omega \wedge (d\lambda)
\end{equation}
for all $\omega \in \Lambda_k(\M)$ and $\lambda \in \Lambda_l(\M)$.  From these Leibniz rules and \eqref{dcom} we have
$$ {\mathcal L}_u (v \wedge dv) + d( {\mathcal L}_u v \wedge v ) = 2 v \wedge {\mathcal L}_u dv $$
and the claim (iii) now follows from \eqref{ot}, \eqref{stok}, and Stokes' theorem.

Now we prove (iv).  From the Leibniz rule and the self-adjointness of $A$, we have
$$ \partial_t E(t) = \int_{\M} \langle \partial_t \omega, A \omega \rangle\ d\operatorname{vol}.$$
Using \eqref{vort-1}, \eqref{lv} and \eqref{integ}, we conclude
$$ \partial_t E(t) = \int_{\M} \langle \omega, {\mathcal L}_u A \omega \rangle\ d\operatorname{vol}.$$
As $\omega$ is closed, and we are working formally, we may write $\omega = dv$ for some $1$-form $v$.  By duality, we thus have
$$ \partial_t E(t) = \int_{\M} \langle v, \delta {\mathcal L}_u A \omega \rangle\ d\operatorname{vol}.$$
Since ${\mathcal L}_u$ annihilates $d\operatorname{vol}$, it commutes with $*$; by \eqref{dcom} and \eqref{dim} it therefore commutes with $\delta$, thus by \eqref{vort-2}
$$ \delta {\mathcal L}_u A \omega  = {\mathcal L}_u u = [u,u] = 0$$
(where $[X,Y]$ denotes the Lie bracket of two vector fields $X,Y$) and the claim follows.

Finally, we establish (v).  From \eqref{vort-1}, we have
$$ \partial_t \int_{\M} \langle \omega, \alpha \rangle\ d\operatorname{vol} = -\int_{\M} \langle {\mathcal L}_u \omega, \alpha \rangle\ d\operatorname{vol}.$$
By \eqref{lv} and \eqref{integ}, the right-hand side is equal to
$$ \int_{\M} \langle  \omega, {\mathcal L}_u \alpha \rangle\ d\operatorname{vol}.$$
Writing $\omega = dv$ as before, and using \eqref{integ}, \eqref{dcom} and $d\alpha = X$, we can write this as
$$ -\int_{\M} \langle v, {\mathcal L}_u X \rangle\ d\operatorname{vol}.$$
Since 
$$ {\mathcal L}_u X = [u,X] = -[X,u] = -{\mathcal L}_X u$$
we can use \eqref{vort-2} to write the previous expression as
$$ \int_{\M} \langle v, {\mathcal L}_X \delta A \omega \rangle\ d\operatorname{vol}.$$
As $X$ is divergence-free, ${\mathcal L}_X$ commutes with $*$ and thus with $\delta$.  By duality and $dv = \omega$, the above expression becomes
$$ \int_{\M} \langle \omega, {\mathcal L}_X A \omega \rangle\ d\operatorname{vol}.$$
Using \eqref{integ} and using the self-adjointness of $A$, this is equal to
$$ -\int_{\M} \langle A {\mathcal L}_X \omega, \omega \rangle\ d\operatorname{vol}.$$
But as $A$ and ${\mathcal L}_X$ commute, the previous two expressions are also negations of each other, and must thus be zero.  The claim follows.

\section{Formal Lagrangian formulation}\label{lag-sec}

As in the preceding section, our calculations here will be purely formal, without regard to issues of smoothness or decay.

Given a divergence-free velocity field $u\colon [0,T] \to B^1(\M)$, we can form the family of volume-preserving diffeomorphisms $\Phi(t)\colon \M \to \M$ for $t \in [0,T]$ by solving the ODE
$$ \partial_t \Phi(t,x) = u(t, \Phi(t,x))$$
with initial data $\Phi(0,x) = x$.  We can then deform this family to a two-parameter family $\Phi(s,t)\colon \M \to \M$ of volume-preserving diffeomorphisms with $s$ near zero by solving a further ODE
\begin{equation}\label{sphi}
 \partial_s \Phi(s,t,x) = v( t, \Phi(s,t,x) )
\end{equation}
for some additional divergence-free velocity field $v\colon [0,T] \to B^1(\M)$, with initial data $\Phi(0,t,x) = \Phi(t,x)$.  The velocity field $u$ then deforms in $s$ via the formula
\begin{equation}\label{dphi}
 \partial_t \Phi(t,x) = u(s, t, \Phi(t,x)).
\end{equation}
Differentiating \eqref{sphi} in $t$ and \eqref{dphi} in $s$ and comparing at $s=0$ gives the identity
$$ \partial_s u + (v \cdot \nabla) u = \partial_t v + (u \cdot \nabla) v; $$
since 
$$ {\mathcal L}_u v = [u,v] = (u \cdot \nabla) v - (v \cdot \nabla) u$$
we thus have\footnote{Another way of interpreting the identity \eqref{las} is as follows.  The diffeomorphisms $\Phi$ can be viewed as a trivialisation of the $\M$-bundle over the parameter space $\R^2$ of the variables $(s,t)$.  The flat connection associated to this trivialisation, when written in terms of the standard trivialisation, correspond to the differential operators $\partial_t + {\mathcal L}_u$ and $\partial_s + {\mathcal L}_v$.  As the connection is flat, these operators commute, giving \eqref{las}.}
\begin{equation}\label{las}
 \partial_s u = \partial_t v + {\mathcal L}_u v.
\end{equation}
Let $A\colon B_2(\M) \to \Gamma^2(\M)$ be formally self-adjoint.
We now introduce the formal Lagrangian
\begin{equation}\label{lphi}
 {\mathcal L}[\Phi] \coloneqq \frac{1}{2} \int_0^T \int_{\M} \langle \omega, A \omega \rangle\ d\operatorname{vol} dt,
\end{equation}
where $\omega \in B_2(\M)$ is closed and solves \eqref{vort-2}; we assume that $\delta A$ is invertible, so that $\omega$ is uniquely determined by $u$.  We claim that if $u$ also solves \eqref{vort-1}, then it is a (formal) critical point of the Lagrangian if one holds the endpoints $\Phi(0), \Phi(T)$ fixed; in other words, if one deforms $u$ and $\Phi$ as above  using a divergence-free field $v$ that vanishes at the endpoints $t=0, T$, then
$$ \partial_s {\mathcal L}[\Phi] = 0 $$
at $s=0$.  Indeed, from \eqref{lphi} and the self-adjointness of $A$, the left-hand side is
$$ \int_0^T \int_{\M} \langle \omega, \partial_s A \omega \rangle\ d\operatorname{vol} dt.$$
As $\omega$ is closed, we can write $\omega = d \alpha$ for some $1$-form $\alpha$.  Integrating by parts and using \eqref{vort-2}, this quantity can be rewritten as
$$ \int_0^T \int_{\M} \langle \alpha, \partial_s u \rangle\ d\operatorname{vol} dt$$
which by \eqref{las} and \eqref{integ} is equal to
$$ \int_0^T \int_{\M} \langle - \partial_t \alpha - {\mathcal L}_u \alpha, v \rangle\ d\operatorname{vol} dt.$$
As $v$ is divergence-free, we can (formally) write $v = \delta \beta$ for some $2$-vector field $\beta$.  Integrating by parts using $d \alpha = \omega$, we can rewrite the preceding expression as
$$ \int_0^T \int_{\M} \langle - \partial_t \omega - {\mathcal L}_u \omega, \beta \rangle\ d\operatorname{vol} dt.$$
But this vanishes by \eqref{vort-1}.

\begin{remark}  Suppose the operator $A$ is (formally) positive definite.  Then one can interpret the above calculation as asserting that the generalised Euler equations in Definition \ref{gee} are the formal geodesic flow on the infinite-dimensional manifold $\operatorname{Sdiff}(\M)$ of volume-preserving diffeomorphisms on $\M$, where we endow this manifold with a right-invariant metric whose quadratic form on the tangent space of the identity (which one can identify with divergence-free vector fields $u$) is given by $u \mapsto \int_{\M} \langle \omega, A \omega \rangle d\operatorname{vol}$, where $\omega$ obeys \eqref{vort-2}.  In the case of the true Euler equations, this fact was famously observed by Arnold \cite{arnold}, as a special case of the Euler-Poincar\'e formalism, and a rigorous version of these computations was used in \cite{ebin-marsden} to obtain a local existence theorem for the true Euler equations that is close to that in Theorem \ref{lest}.  See \cite{wash} for some analogous results for the SQG equation.  If one drops the requirement that the diffeomorphisms be volume preserving, then there are several \emph{compressible} fluid equations that also have a rigorous geodesic flow interpretation; see e.g. \cite{ck}, \cite{khesin}, \cite{lenells}, \cite{escher}, \cite{wunsch}.  It is thus likely that the formal computations in this section can similarly be made rigorous given suitable hypotheses on the vector potential operator $A$ and on the initial data, but we will not attempt to do so here.
\end{remark}

Given the above Lagrangian formulation of the generalised Euler equations in Definition \ref{gee}, it should come as no surprise that the conservation laws in Proposition \ref{Fcl} are associated to symmetries of the Lagrangian \eqref{lphi}, in accordance with Noether's theorem.  Indeed, the Kelvin circulation theorem (and hence helicity conservation and stream line conservation) come from the invariance of \eqref{lphi} with respect to the right action of $\operatorname{Sdiff}(\M)$, while conservation of the Hamiltonian comes (as usual) from time translation symmetry, and conservation of impulse comes from the symmetry arising from the diffeomorphisms $e^{sX}$ generated by the vector field $X$.  See \cite{tao-blog}, \cite{shankar} for further discussion.

\section{Classical local existence}\label{lest-sec}

In this section we leave the realm of formal calculations, and prove Theorem \ref{lest} rigorously.  We will rely primarily on the energy method, with some modification at low frequencies to deal with the failure of the generalised Biot-Savart operator $\delta A$ to map $H^s$ to $L^\infty$ in the $m=2$ case.   To abbreviate the notation, we write $\| \|_{\dot H^s}$ for $\| \|_{\dot H^s(\M)}$, and similarly for $\| \|_{L^p}$, $\|\|_{H^s}$, etc..  It will also be convenient to use the norm
$$
\| f \|_{L^p \cap H^s} \coloneqq \| f\|_{L^p} + \| f \|_{H^s}.
$$

Let $\M, d, A, M, p$ be as in Theorem \ref{lest}.  For brevity, we drop the dependence of constants on $d,M,A,p$ from the asymptotic notation.

We can of course write the system \eqref{vort-1}, \eqref{vort-2} as a single equation
\begin{equation}\label{vort-comb}
 \partial_t \omega + {\mathcal L}_{\delta A \omega} \omega = 0.
\end{equation}

In coordinates, the equation \eqref{vort-comb} becomes
\begin{equation}\label{dap}
 \partial_t \omega + (\delta A \omega \cdot \nabla) \omega = {\mathcal O}( (\nabla^2 A \omega) \omega )
\end{equation}
where we use ${\mathcal O}(X)$ to denote an expression that has the \emph{schematic form} of $X$ in the sense that it is a linear combination (with constant coefficients) of components of a tensor of the form $X$ (interpreting all products in $X$ as tensor products).  

We first establish uniqueness.  Given two solutions $\omega_1, \omega_2 \in X^{s,p}$ to \eqref{dap} with initial data $\omega_0$, the difference $\alpha \coloneqq \omega_1 - \omega_2$ lies in $X^{s,p}$ and obeys an equation of the form
\begin{equation}\label{diff}
 \partial_t \alpha + (\delta A \omega_2 \cdot \nabla) \alpha + (\delta A \alpha \cdot \nabla) \omega_1 
 = {\mathcal O}( (\nabla^2 A \alpha) \omega_1 ) + {\mathcal O}( (\nabla^2 A \omega_2) \alpha ).
\end{equation}
Taking inner products with $\alpha |\alpha|^{p-2}$ and integrating using the divergence-free nature of $\delta A \omega_2$, we obtain\footnote{To be more rigorous here, one could obtain instead a transport equation for $(\eps^2 + |\alpha|^2)^{p/2}$ for $\eps>0$, run the Gronwall argument below for the quantity $\| (\eps^2 + |\alpha|^2)^{1/2} \|_{L^p}$, and then send $\eps$ to zero; we leave the details to the interested reader.  Similarly for other arguments in this section involving derivatives of $L^p$ or $L^2$ norms.} the inequality
$$
\partial_t \| \alpha \|_{L^p} \lesssim \| (\delta A \alpha \cdot \nabla) \omega_1 \|_{L^p} +
\| (\nabla^2 A \alpha) \omega_1 \|_{L^p} + \| (\nabla^2 A \omega_2) \alpha \|_{L^p},$$
and hence by H\"older's inequality one has
\begin{align*}
\partial_t \| \alpha \|_{L^p} &\lesssim \| P_{\operatorname{lo}} \nabla A \alpha \|_{L^q} \| \nabla \omega_1 \|_{L^m} +
\| P_{\operatorname{hi}} \nabla A \alpha \|_{L^r} \| \nabla \omega_1 \|_{L^d} \\
&\quad + \| \nabla^2 A \alpha \|_{L^p} \| \omega_1 \|_{L^\infty} + \| \nabla^2 A \omega_2 \|_{L^\infty} \| \alpha \|_{L^p}
\end{align*}
where $\frac{1}{q} \coloneqq \frac{1}{p} - \frac{1}{m}$, $\frac{1}{r} \coloneqq \frac{1}{p} -\frac{1}{d}$, and $P_{\operatorname{lo}}, P_{\operatorname{hi}}$ are the Fourier projections to frequencies $|\xi| \leq 1$ and $|\xi| > 1$ respectively.
From Sobolev embedding in $\R^m \times (\R/\Z)^{d-m}$, \eqref{pass}, and the hypothesis $s > \frac{d}{2} + 1$, one has
\begin{align*}
 \| P_{\operatorname{lo}} \nabla A \alpha \|_{L^q},  \| P_{\operatorname{hi}} \nabla A \alpha \|_{L^r} &\lesssim \| \nabla^2 A \alpha \|_{L^p} \\
 \| \nabla \omega_1 \|_{L^m}, \| \nabla \omega_1 \|_{L^d}, \| \omega_1 \|_{L^\infty} &\lesssim \| \omega_1 \|_{H^s} \lesssim \| \omega_1 \|_{X^{s,p}} \\
\| \nabla^2 A \omega_2 \|_{L^\infty} &\lesssim \| \nabla^2 A \omega_2 \|_{H^s} \lesssim \| \omega_2 \|_{H^s} \lesssim \| \omega_2 \|_{X^{s,p}}.
\end{align*}
Also, from Definition \ref{reason-def}, the operator $\nabla^2 A$ is bounded on $L^2$ and has a kernel obeying Calder\'on-Zygmund estimates, so is bounded on $L^p$ by Calder\'on-Zygmund theory (see e.g. \cite{stein:large}).  We conclude that
$$  \partial_t \| \alpha \|_{L^p} \lesssim (\|\omega_1\|_{X^{s,p}} + \|\omega_2\|_{X^{s,p}}) \| \alpha \|_{L^p}.$$
Since $\alpha(0)=0$, we conclude from Gronwall's inequality that $\alpha=0$ identically, giving uniqueness.

Next, we show existence of (weak) solutions using a standard viscosity method which we briefly sketch here; later on we will upgrade the regularity of solutions from weak to strong.  For any $\eps>0$ we can consider the generalised Navier-Stokes equation
\begin{equation}\label{oma}
\partial_t \omega + {\mathcal L}_{\delta A \omega} \omega + \eps \Delta \omega = 0
\end{equation}
(recall in this paper that $\Delta$ denotes the Hodge Laplacian, which is positive semi-definite).  We can write this equation schematically as
\begin{equation}\label{stomp}
\partial_t \omega + \eps \Delta \omega = {\mathcal O}( \nabla A \omega \nabla \omega ) + {\mathcal O}( (\nabla^2 A \omega) \omega ).
\end{equation}
From repeated application of the H\"older and Sobolev inequalities, as well as \eqref{pass}, the $L^p$ boundedness of $\nabla^2 A$, and the hypothesis $s > \frac{d}{2}+1$, one can check that if $\omega_1, \omega_2 \in L^p \cap H^{s+1} \cap B_2(\M)$, then the expression ${\mathcal O}( \nabla A \omega_1 \nabla \omega_2 ) + {\mathcal O}( (\nabla^2 A \omega_1) \omega_2 )$ lies in $L^p \cap H^s(\M)$, and that this operation is locally Lipschitz\footnote{Indeed, $\nabla A \omega_1, \nabla^2 A \omega_1, \nabla \omega_2, \omega_2$ both lie in $H^s(\M)$ thanks to \eqref{pass}, and the product of two functions in $H^s(\M)$ lies in both $H^s(\M)$ and $L^p(\M)$ by the Leibniz rule and the H\"older and Sobolev inequalities.} in the $\omega_1$ and $\omega_2$ variables in the indicated norms.  By running a contraction mapping argument that places $\omega$ in the function space
$$ C^0( [0,T] \to L^p \cap H^{s} \cap B_2(\M) ) \cap L^2( [0,T] \to H^{s+1} \cap B_2(\M) ),$$
and using the parabolic smoothing effects of the heat equation (and also noting that the Lie derivative operator ${\mathcal L}_{\delta A \omega}$ and the Hodge Laplacian $\Delta$ both preserve the space $B_2(\M)$ of closed $2$-forms), one can then construct local solutions in \eqref{oma} in the above space that can be continued as long as one has a uniform bound on the quantity
\begin{equation}\label{ps}
\| \omega(t) \|_{L^p \cap H^s}.
\end{equation} 
But for any constant coefficient differential operator $D$ of order $k$ for some $0 \leq k \leq s$ (with no lower order terms), we have upon differentiating \eqref{oma} by $D$ that
$$
\partial_t D \omega + (\delta A \omega \cdot \nabla) D \omega + \eps \Delta D \omega = F$$
where
$$ F \coloneqq [\delta A \omega, D] \cdot \nabla \omega + {\mathcal O}( D( (\nabla^2 A \omega) \omega ) )$$
and $[A,B] \coloneqq AB - BA$.
Multiplying by $D\omega$ and integrating by parts, we obtain an energy inequality of the form
\begin{equation}\label{dow}
 \partial_t \| D\omega \|_{L^2}^2 \lesssim \| F \|_{L^2} \|D \omega \|_{L^2}
\end{equation}
where the implied constant does not depend on $\eps$.  However, using the Moser estimate
\begin{equation}\label{tools}
 \| D(uv) \|_{L^2} \lesssim \| u \|_{\dot H^k} \|v\|_{L^\infty} + \|u \|_{L^\infty} \|v\|_{\dot H^k} 
\end{equation}
(see e.g. \cite[(2.0.22)]{tools}), as well as the commutator estimate\footnote{The commutator estimate would usually have $\| u \|_{\dot H^k} \| \nabla v \|_{L^\infty}$ in place of $\|u\|_{\dot H^{k+1}} \| v\|_{L^\infty}$ here, but it is not difficult to adapt the standard (paraproduct-based) proof of the estimate to also establish \eqref{commutator} as written.  Indeed, one can use the Leibniz rule to write $[u,D] \nabla v$ as $\sum_{i=1}^k {\mathcal O}(\nabla^i u \nabla^{k+1-i} v )$; the contribution of the cases $i=1$, $i=k+1$ are trivial, and all intermediate cases can be handled by paraproducts or Littlewood-Paley decomposition.}
\begin{equation}\label{commutator}
\| [u,D] \nabla v \|_{L^2} \lesssim \|u\|_{\dot H^{k+1}} \| v\|_{L^\infty} + \|\nabla u \|_{L^\infty} \|v\|_{\dot H^{k}}
\end{equation}
(see e.g. \cite{km}, \cite{kponce}, noting the claim is trivial for $k=0$), we see that
\begin{equation}\label{x2}
\|F\|_{L^2} \lesssim  \| \nabla^2 A \omega \|_{L^\infty} \| \nabla^k \omega \|_{L^2} + \| \nabla^{k+2} A \omega \|_{L^2} \| \omega \|_{L^\infty}. 
\end{equation}
Applying Sobolev embedding then gives
$$ \|F\|_{L^2} \lesssim \| \nabla^2 A \omega \|_{H^s} \|\omega \|_{H^s}.$$
Hence by \eqref{pass} and summing over a suitable choice of $D$ we have
\begin{equation}\label{taba}
 \partial_t \| \omega \|_{H^s}^2 \lesssim \| \omega \|_{H^s}^3.
\end{equation}
Since $\|\omega(0)\|_{H^s} \lesssim M$, this gives an \emph{a priori} bound 
\begin{equation}\label{omam}
 \| \omega(t) \|_{H^s} \lesssim M
\end{equation}
for $0 \leq t \leq T$, if $T$ is sufficiently small depending on the bound $M$.  

Now we need to control the $L^p$ component in \eqref{ps}.  Since $\nabla^2 A$ is bounded in $L^p$, $\nabla^2 A \omega$ has a $L^p$ norm of $O( \| \omega \|_{L^p})$.  In particular, from Sobolev embedding and H\"older we have
\begin{equation}\label{sm}
\begin{split}
\| {\mathcal O}( \nabla A \omega \nabla \omega ) + {\mathcal O}( (\nabla^2 A \omega) \omega ) \|_{L^p} 
&\lesssim \| \nabla A \omega \|_{L^q} \| \nabla \omega \|_{L^2} + \| \nabla^2 A \omega \|_{L^\infty} \| \omega \|_{L^p} \\
&\lesssim \| \nabla^2 A \omega \|_{L^p \cap H^s} \| \omega \|_{H^s} + \| \nabla^2 A \omega \|_{H^s} \| \omega \|_{L^p} \\
&\lesssim \| \omega \|_{L^p \cap H^s} \| \omega \|_{H^s}
\end{split}
\end{equation}
where $1/q \coloneqq 1/p - 1/2$, and hence by \eqref{stomp}, \eqref{omam}, and the contractivity of the heat semigroup in $L^p$
$$ \partial_t \| \omega \|_{L^p} \lesssim_M 1 + \| \omega \|_{L^p} $$
and hence by Gronwall's inequality we have
\begin{equation}\label{samo}
\sup_{0 \leq t \leq T} \| \omega(t) \|_{L^p} \lesssim_{M,T} 1
\end{equation}
giving the required uniform bound on \eqref{ps}.  This \emph{a priori} bound allows us to continue the solution to \eqref{oma} up to a time $T>0$ that is independent of $\eps$. A standard compactness argument sending $\eps \to 0$ (and noting from the Rellich compactness theorem that weak convergence in $L^p \cap H^s$ implies strong convergence in $C^1$) then gives a (distributional) solution to the inviscid system \eqref{vort-1}, \eqref{vort-2} with the regularity
\begin{equation}\label{linfty}
 \omega \in L^\infty([0,T] \to L^p \cap H^s \cap B_2(\M) ).
\end{equation}
This is not quite smooth enough to place $\omega$ in $X^{s,p}$ (mainly because of the lack of continuity in time); we will upgrade the regularity of $\omega$ shortly.

To prove continuous dependence on the initial data, we use an argument originally due to Bona and Smith \cite{bona} (see also the survey of Tzvetkov \cite{tzvetkov}).  Let $\omega_0 \in L^p \cap H^s \cap B_2(\M)$ with 
$$ \| \omega_0 \|_{L^p \cap H^s} < M.$$
Let $\omega'_0 \in L^p \cap H^{s+1} \cap B_2(\M)$ be a suitable mollification of $\omega_0$ which also obeys the bound
$$ \| \omega'_0 \|_{L^p \cap H^s} < M;$$
we will choose $\omega'_0$ more precisely later.  
Let $\omega \in L^\infty( [0,T] \to L^p \cap H^s \cap B_2(\M) )$ be a solution to \eqref{vort-comb} with initial data $\omega_0$ constructed by the preceding compactness argument, and similarly define $\omega' \in L^\infty( [0,T] \to L^p \cap H^s \cap B_2(\M))$.
From \eqref{omam}, \eqref{samo} we have the bounds to be the solution to \eqref{vort-comb} with initial data $\omega'_0$.  Then we have
\begin{equation}\label{oso}
\| \omega(t) \|_{L^p \cap H^s}, \| \omega'(t) \|_{L^p \cap H^s} \lesssim_{M,T} 1
\end{equation}
for all $0 \leq t \leq T$.  A routine modification of the proof of \eqref{taba} yields the \emph{a priori} bound
$$ \partial_t \| \omega' \|_{H^{s+1}}^2 \lesssim \| \omega \|_{H^s} \| \omega' \|_{H^{s+1}}^2$$
which by Gronwall's inequality and \eqref{oso} leads to the bound
\begin{equation}\label{omb}
 \| \omega'(t) \|_{H^{s+1}} \lesssim_{M,T} \| \omega'_0 \|_{H^{s+1}}
\end{equation}
for all $t \in [0,T]$.

Next, we set $\alpha \coloneqq \omega' - \omega$.  As in \eqref{diff}, we have the difference equation
\begin{equation}\label{Diffa}
\partial_t \alpha + (\delta A \omega \cdot \nabla) \alpha + (\delta A \alpha \cdot \nabla) \omega' =  
{\mathcal O}( (\nabla^2 A \alpha) \omega' ) + {\mathcal O}( (\nabla^2 A \omega) \alpha ).
\end{equation}
Taking inner products with $\alpha |\alpha|^{p-2}$ as before, we see that
$$ \partial_t \| \alpha \|_{L^p} \lesssim \| \nabla A \alpha \nabla \omega' \|_{L^p} + \| \nabla^2 A \alpha \omega' \|_{L^p} + \| \nabla^2 A \omega \alpha \|_{L^p}.$$
Using Sobolev embedding and H\"older as in \eqref{sm}, as well as the boundedness of $\nabla^2 A$ on $L^p$, we conclude that
$$ \partial_t \| \alpha \|_{L^p} \lesssim \|\alpha \|_{L^p} (\| \omega' \|_{H^s} + \| \omega \|_{H^s})$$
and hence by \eqref{oso} and Gronwall's inequality we have
\begin{equation}\label{alp}
 \| \alpha(t) \|_{L^p} \lesssim_{M,T} \| \alpha(0) \|_{L^p}
\end{equation}
for all $t \in [0,T]$.  

Next, if $D$ is a constant coefficient operator of order $k$ for some $k \leq s$, then upon applying $D$ to \eqref{Diffa} we have
\begin{equation}\label{Diffb}
\partial_t D\alpha + (\delta A \omega \cdot \nabla) D \alpha  =  F'
\end{equation}
where
$$ F' \coloneqq
[\delta A \omega , D] \cdot \nabla \alpha + [\delta A \alpha, D] \cdot \nabla \omega'
- (\delta A \alpha \cdot \nabla) D\omega' 
+ {\mathcal O}( D((\nabla^{2} A \alpha) \omega )) + {\mathcal O}( D((\nabla^{2} A \omega') \alpha )).
$$
Multiplying \eqref{Diffb} by $D\alpha$ and then integrating by parts, we conclude that
$$
\partial_t \| D \alpha \|_{L^2}^2 \lesssim \| D \alpha \|_{L^2} \| F' \|_{L^2}$$
On the other hand, by using \eqref{tools}, \eqref{commutator}, \eqref{pass} as before, followed by \eqref{oso}, we have
\begin{align*}
\|F'\|_{L^2} &\lesssim \| \alpha\|_{H^k} ( \| \omega \|_{H^s} + \| \omega' \|_{H^s} ) + \| (\delta A \alpha \cdot \nabla) D\omega'\|_{L^2} \\
&\lesssim_{M,T} \| \alpha\|_{H^k} + \| \nabla A \alpha \|_{L^\infty} \| \omega'\|_{H^{k+1}}.
\end{align*}
Summing over a suitable set of $D$, we conclude that
$$
\partial_t \| \alpha \|_{H^k}^2 \lesssim_{M,T} \| \alpha\|_{H^k}^2 + \| \alpha\|_{H^k} \| \nabla A \alpha \|_{L^\infty} \| \omega'\|_{H^{k+1}}
$$
and thus
\begin{equation}\label{fod}
\partial_t \| \alpha \|_{H^k} \lesssim_{M,T} \| \alpha\|_{H^k} + \| \nabla A \alpha \|_{L^\infty} \| \omega'\|_{H^{k+1}}
\end{equation}
for any $0 \leq k \leq s$.

When $m \geq 3$, we can use Sobolev embedding and \eqref{pass} to bound
$$ \| \nabla A \alpha \|_{L^\infty} \lesssim \| \nabla^2 A \alpha \|_{H^{s-1}} \lesssim \| \alpha \|_{H^{s-1}}.$$
However when $m=2$ the situation is more delicate.  If $P_{hi}$ and $P_{lo}$ denote the Fourier projections used previously, we have
$$ \| P_{hi} \nabla A \alpha \|_{L^\infty} \lesssim \| \nabla^2 A \alpha \|_{H^{s-1}} \lesssim \| \alpha \|_{H^{s-1}}$$
For $P_{lo}$, we see for any $N \geq 2$ using the Bernstein and Cauchy-Schwarz inequalities, as well as Plancherel's theorem, that
\begin{align*}
 \| P_{lo} \nabla A \alpha \|_{L^\infty} &\lesssim \sum_{M \leq 1} \| P_M \nabla A \alpha \|_{L^\infty} \\
&\lesssim \sum_{M \leq 1} \| P_M \nabla^2 A \alpha \|_{L^2} \\
&\lesssim \sum_{M \leq N^{-C}} M^{1/2-1/p} \| \nabla^2 A \alpha \|_{L^p}
+ \sum_{N^{-C} < M \leq 1} \| P_M \nabla^2 A \alpha \|_{L^2} \\
&\lesssim \frac{1}{N^2} \| \nabla^2 A \alpha \|_{L^p}
+ \sqrt{\log N} \left(\sum_{N^{-C} < M \leq 1} \| P_M \nabla^2 A \alpha \|_{L^2}^2\right)^{1/2} \\
&\lesssim \frac{1}{N^2} \| \nabla^2 A \alpha \|_{L^p} + \sqrt{\log N} \| \nabla^2 A \alpha \|_{L^2}
\end{align*}
where $C \coloneqq 2 / (1/p -1/2)$, $M$ ranges over dyadic numbers $M = 2^m$, $m \in \Z$, and $P_M$ is a Littlewood-Paley type Fourier projection to frequencies comparable to $M$.  Since $\nabla^2 A$ is bounded on $L^p$, $L^2$, and $H^{s-1}$, we conclude for any choice of $m$ that
$$ \| \nabla A \alpha \|_{L^\infty} \lesssim \frac{1}{N^2} \| \alpha \|_{L^p} + \sqrt{\log N} \| \alpha \|_{L^2} + \| \alpha \|_{H^{s-1}}$$
Inserting these bounds and \eqref{alp} into \eqref{fod}, we have
\begin{equation}\label{fodo}
\partial_t \| \alpha \|_{H^k} \lesssim_{M,T} \| \alpha\|_{H^k} + \left(\frac{1}{N^2} \| \alpha(0) \|_{L^p} + \sqrt{\log N} \| \alpha \|_{L^2} + \| \alpha \|_{H^{s-1}}\right) \| \omega'\|_{H^{k+1}}
\end{equation}
for any $0 \leq k \leq s$ and $N \geq 2$.  We first apply this bound with $k=0$ using \eqref{oso} to obtain
$$
\partial_t \| \alpha \|_{L^2} \lesssim_{M,T} \frac{1}{N} \| \alpha(0) \|_{L^p} + \sqrt{\log N} \| \alpha \|_{L^2} $$
and hence by Gronwall's inequality
$$ 
\| \alpha(t) \|_{L^2} \lesssim_{M,T} \exp( O_{M,T}(\sqrt{\log N})) \left(\frac{1}{N^2} \| \alpha(0) \|_{L^p} + \| \alpha(0) \|_{L^2}\right) $$
which on re-insertion back into \eqref{fodo} and conceding some powers of $N$ give
\begin{equation}\label{fodo-2}
\partial_t \| \alpha \|_{H^k} \lesssim_{M,T} \| \alpha\|_{H^k} + \left(\frac{1}{N} \| \alpha(0) \|_{L^p} + N^{s-1} \| \alpha(0) \|_{L^2} + \| \alpha \|_{H^{s-1}}\right) \| \omega'\|_{H^{k+1}}.
\end{equation}
 Applying this bound with $k=s-1$ and using \eqref{oso}, we conclude that
$$
\partial_t \| \alpha \|_{H^{s-1}} \lesssim_{M,T} \| \alpha\|_{H^{s-1}} + \frac{1}{N} \| \alpha(0) \|_{L^p} + N^{s-1} \| \alpha(0) \|_{L^2} $$
and hence by Gronwall's inequality
$$ \| \alpha(t) \|_{H^{s-1}} \lesssim_{M,T} \| \alpha(0)\|_{H^{s-1}} + \frac{1}{N} \| \alpha(0) \|_{L^p} + N^{s-1} \| \alpha(0) \|_{L^2}$$
for any $0 \leq t \leq T$.  Inserting this back into \eqref{fodo} for $k=s$ and using \eqref{omb}, we have
$$
\partial_t \| \alpha \|_{H^s} \lesssim_{M,T} \| \alpha\|_{H^s} + \left(\frac{1}{N} \| \alpha(0) \|_{L^p} + N^{s-1} \| \alpha(0) \|_{L^2} + \| \alpha(0) \|_{H^{s-1}}\right) \| \omega'(0)\|_{H^{s+1}}
$$
and hence by Gronwall's inequality 
$$
\| \alpha(t) \|_{H^s} \lesssim_{M,T} \| \alpha(0)\|_{H^s} + \left(\frac{1}{N} \| \alpha(0) \|_{L^p} + N^{s-1} \| \alpha(0) \|_{L^2} + \| \alpha(0) \|_{H^{s-1}}\right) \| \omega'(0)\|_{H^{s+1}}
$$
for any $0 \leq t \leq T$.  Combining this with \eqref{alp} and the definition of $\alpha$, we have
\begin{equation}\label{frodo}
\begin{split}
\| \omega'(t)-\omega(t) \|_{L^p \cap H^s} &\lesssim_{M,T} \| \omega'_0 - \omega_0\|_{L^p \cap H^s} \\
&\quad + \left(\frac{1}{N} \| \omega'_0 - \omega_0 \|_{L^p} + N^{s-1} \| \omega'_0 - \omega_0 \|_{L^2} + \| \omega'_0 - \omega_0\|_{H^{s-1}}\right) \| \omega'_0\|_{H^{s+1}}.
\end{split}
\end{equation}

Let $\eps > 0$.  If we let $\tilde \omega_0$ be initial data in $L^p \cap H^s \cap B_2(\M)$ that is sufficiently close to $\omega_0$ in $L^p \cap H^s$ norm (depending on $\eps,N$), and let $\tilde \omega$ be the corresponding solution to \eqref{vort-comb}, the same argument (replacing $\omega$ with $\tilde \omega$) gives
\begin{align*}
\| \omega'(t)-\tilde \omega(t) \|_{L^p \cap  H^s} &\lesssim_{M,T}  \| \omega'_0 - \omega_0\|_{L^p \cap H^s} \\
&\quad + \left(\frac{1}{N} \| \omega'_0 - \omega_0 \|_{L^p} + N^{s-1} \| \omega'_0 - \omega_0 \|_{L^2} + \| \omega'_0 - \omega_0\|_{H^{s-1}}\right) \| \omega'_0\|_{H^{s+1}} + \eps
\end{align*}
and thus by the triangle inequality 
\begin{align*}
 \| \omega(t)-\tilde \omega(t) \|_{L^p \cap H^s} &\lesssim_{M,T} \| \omega'_0 - \omega_0\|_{L^p \cap H^s} \\
&\quad + \left(\frac{1}{N} \| \omega'_0 - \omega_0 \|_{L^p} + N^{s-1} \| \omega'_0 - \omega_0 \|_{L^2} + \| \omega'_0 - \omega_0\|_{H^{s-1}}\right) \| \omega'_0\|_{H^{s+1}} + \eps
\end{align*}
for $\tilde \omega_0$ sufficiently close to $\omega_0$ in $L^p \cap H^s \cap B_2(\M)$.

If we now let $\omega'_0$ be a smoothed Fourier projection of $\omega_0$ (of Littlewood-Paley type) to frequencies less than $N$, we see from Plancherel's theorem and dominated convergence that
\begin{align*}
\|\omega'_0 - \omega_0 \|_{L^p \cap H^s} &\lesssim \eps \\
 N^s \| \omega'_0- \omega_0 \|_{L^2} &\lesssim \eps \\
 N \| \omega'_0- \omega_0 \|_{H^{s-1}} &\lesssim \eps \\
\frac{1}{N} \| \omega'_0 \|_{H^{s+1}}  &\lesssim \eps
\end{align*}
for $N$ large enough, and thus 
\begin{equation}\label{sort}
\sup_{0 \leq t \leq T} \| \omega(t)-\tilde \omega(t) \|_{L^p \cap H^s} \lesssim_{M,T} \eps 
\end{equation}
whenever $\tilde \omega_0$ is sufficiently close to $\omega_0$.  Writing $\beta \coloneqq \omega - \tilde \omega$, we have the difference equation
$$
\partial_t \beta = - (\delta A \tilde \omega \cdot \nabla) \beta - (\delta A \beta \cdot \nabla) \omega +  
{\mathcal O}( (\nabla^2 A \beta) \omega ) + {\mathcal O}( (\nabla^2 A \tilde \omega) \beta ).$$
Using \eqref{sort}, \eqref{omam}, and \eqref{pass}, all the terms on the right-hand side can be computed to have an $H^{s-1}$ norm of $O_{M,T}(\eps)$, and so
\begin{equation}\label{sort-2}
 \sup_{0 \leq t \leq T}\| \partial_t(\omega(t)-\tilde \omega(t)) \|_{ H^{s-1}} \lesssim_{M,T} \eps.
\end{equation}
The estimates \eqref{sort}, \eqref{sort-2} will give continuous dependence of the solution map $\omega_0 \mapsto \omega$ from $L^p \cap H^s \cap B_2(\M)$ to $X^{s,p}$ as soon as we establish that the solution $\omega$ actually lies in $X^{s,p}$.  We already have the $L^\infty$ regularity \eqref{linfty}; we now upgrade this to $C^0$ regularity.  By approximating $\omega_0$ by initial data in $L^p \cap H^{s+1} \cap B_2(\M)$ and using the continuity estimates already established, it suffices to establish $C^0$ regularity under the hypothesis that $\omega_0 \in L^p \cap  H^{s+1} \cap B_2(\M)$.  By \eqref{omb} we then have
$$ \| \omega(t) \|_{L^p} + \| \omega(t) \|_{ H^{s+1}} \lesssim_{\omega_0} 1 $$
for all $0 \leq t \leq T$;  from \eqref{oma}, the H\"older and Sobolev inequalities, and \eqref{pass} we then have
$$ \| \partial_t \omega(t) \|_{H^s} \lesssim_{\omega_0} 1 $$
and from repeating the proof of \eqref{sm} we also have
$$ \| \partial_t \omega(t) \|_{L^p} \lesssim_{\omega_0} 1 $$
 and on integrating in $t$ this gives the desired continuity in time in the $L^p \cap H^s(\M)$ topology.  Finally, once we know that $\omega$ lies in $C([0,T] \to L^p \cap  H^s(\M))$, we can use \eqref{oma}, \eqref{pass}, and the H\"older and Sobolev inequalities to conclude that $\partial_t \omega$ lies in $C([0,T] \to H^{s-1}(\M))$, and so $\omega$ lies in $X^{s,p}$ as required.  Setting $u \coloneqq \delta A \omega$, we also have $u \in Y^{s,p}$ by \eqref{pass} and the boundedness of the Calder\'on-Zygmund operator $\nabla^2 A$ on $L^p$; this also gives continuous dependence of $u$ on $\omega_0$.  This concludes the proof of the local wellposedness portion of Theorem \ref{lest}.

Now we establish the Beale-Kato-Majda criterion.  Suppose \emph{a priori} that we have a solution $\omega \in X^{s,p}, u \in Y^{s,p}$ to \eqref{vort-1}, \eqref{vort-2} up to some (possibly large) time $0 < T_* < \infty$ with the bounds
\begin{equation}\label{hsm}
\| \omega(0)\|_{L^p} +  \| \omega(0) \|_{H^s} \leq M
\end{equation}
and
\begin{equation}\label{bkm-bound}
 \int_0^{T_*} \| \omega(t) \|_{L^\infty}\ dt \leq M
\end{equation}
for some $0 < M < \infty$.
Multiplying \eqref{dap} by $\omega$ and integrating, we have
$$
\partial_t \| \omega \|_{L^2}^2 \lesssim \| \nabla^2 A \omega \|_{L^2} \| \omega \|_{L^\infty} \| \omega \|_{L^2}$$
and hence by \eqref{pass}
$$
\partial_t \| \omega \|_{L^2}^2 \lesssim \| \omega \|_{L^\infty} \| \omega \|_{L^2}^2$$
and hence by Gronwall's inequality and \eqref{hsm}, \eqref{bkm-bound} one has
\begin{equation}\label{2b}
 \| \omega(t) \|_{L^2} \lesssim_{M,T_*} 1
\end{equation}
for all $0 \leq t \leq T_*$.

Next, let $D$ be a constant coefficient differential operator of order $k \leq s$, with no lower order terms.  From \eqref{dow}, \eqref{x2} we have
$$
\partial_t \| D\omega \|_{L^2}^2 \lesssim (\| \nabla^2 A \omega \|_{L^\infty} \| \nabla^k \omega \|_{L^2} + \| \nabla^{k+2} A \omega \|_{L^2} \| \omega \|_{L^\infty}) \|D\omega\|_{L^2} $$
and hence on summing over suitable $D$ and using \eqref{pass}
$$ \partial_t \| \omega \|_{H^s}^2 \lesssim \| \omega \|_{H^s}^2 ( \| \omega \|_{L^\infty} + \| \nabla^2 A \omega \|_{L^\infty} ).$$

We now establish a key logarithmic inequality, as in \cite{bkm}:

\begin{lemma}  We have
$$
\| \nabla^2 A \omega \|_{L^\infty} \lesssim 1 + \|\omega\|_{L^2} + \|\omega\|_{L^\infty} \log( 2 + \| \omega \|_{H^s} ).$$
\end{lemma}

\begin{proof}  Let $N \geq 2$ be a parameter to be chosen later.  From \eqref{pass} and Sobolev embedding we have
$$ \| \nabla^3 A \omega \|_{L^\infty} \lesssim \| \nabla^2 A \omega \|_{H^s} \lesssim \| \omega \|_{H^s} $$
and hence for any $x_0 \in \M$, we have
$$ |\nabla^2 A \omega(x_0)| \lesssim N^d \left|\int_{\M} \phi( N(x-x_0)) \nabla^2 A \omega(x)\ d\operatorname{vol}(x)\right| + \frac{1}{N} \| \omega \|_{H^s}$$
for some fixed bump function $\phi$ supported on the unit ball $B_\M(0,1)$, where of course
$$ B_\M(x_0,r) \coloneqq \{ x \in \M: |x-x_0| \leq r \}$$
is the ball of radius $r$ centred at $x_0$ with respect to the distance associated with the Euclidean metric $\eta$ on $\M$.
We split $\omega = \omega 1_{B_\M(x_0,2/N)} + \omega (1 - 1_{B_\M(x_0,2/N)})$.  The former term has an $L^2$ norm of $O( N^{-d/2} \| \omega \|_{L^\infty})$, hence by \eqref{pass} and Cauchy-Schwarz, we have
$$
N^d \left|\int_{\M} \phi( N(x-x_0)) \nabla^2 A (\omega 1_{B_\M(x_0,2/N)})(x)\ d\operatorname{vol}(x)\right| \lesssim \| \omega \|_{L^\infty}.$$
Now we turn to the contribution of $\omega (1 - 1_{B_\M(x_0,2/N)})$.  Using the kernel representation \eqref{kernel-rep} of $A$, we can bound
\begin{align*}
&N^d \left|\int_{\M} \phi( N(x-x_0)) \nabla^2 A (\omega (1-1_{B_\M(x_0,2/N)}))(x)\ d\operatorname{vol}(x)\right| \\
&\quad \lesssim N^3 \int_{B_\M(x_0,1/N)} \int_{\M \backslash B_\M(x_0,2/N)} |\nabla_x^2 K(x,y)| |\omega(y)| \ d\operatorname{vol}(y) d\operatorname{vol}(x).
\end{align*}
From \eqref{nij} one has $|\nabla_x^2 K(x,y)| \lesssim |x-y|^{-d} + |x-y|^{-m}$.  Using $L^\infty$ bounds on $\omega$ for $y \in B_\M(x_0,1)$ and $L^2$ bounds elsewhere, we can bound the above expression by
$$ \| \omega \|_{L^\infty} \log N  + \| \omega \|_{L^2}$$
and hence
$$
\| \nabla^2 A \omega \|_{L^\infty} \lesssim \frac{1}{N} \| \omega \|_{H^s} + \|\omega\|_{L^2} + \|\omega\|_{L^\infty} \log N.$$
Setting $N \coloneqq 2 + \| \omega \|_{H^s}$, we obtain the claim.
\end{proof}

Using this inequality and \eqref{2b}, we thus have
$$ \partial_t \| \omega \|_{H^s}^2 \lesssim_{M,T_*} \| \omega \|_{H^s}^2  \left(1 + \| \omega \|_{L^\infty} \log(2 + \| \omega \|_{H^s}) \right)$$
and hence by the chain rule
$$ \partial_t \log(2+\| \omega \|_{H^s})\lesssim_{M,T_*} (1 + \| \omega \|_{L^\infty}) \log(2 + \| \omega \|_{H^s}).$$
Using Gronwall's inequality and \eqref{hsm}, we conclude the \emph{a priori} bound
$$ \| \omega \|_{H^s} \lesssim_{M,T_*} 1$$
all the way up to $T_*$; repeating the proof of \eqref{samo} we also have
$$ \| \omega \|_{L^p} \lesssim_{M,T_*} 1$$
all the way to this time.  By the local existence theory already established, this allows one to continue the solution beyond the time $T_*$.  Taking contrapositives, we obtain the Beale-Kato-Majda blowup criterion.

\section{Non-self-adjoint blowup: a simple one-dimensional example}

Our three blowup theorems will rely on a ``non-self-adjoint blowup'' mechanism in which the velocity field $u$ depends on the dynamic field (which will either be a scalar $\theta$ or a $2$-form $\omega$, depending on the dimensionality) in a non-self-adjoint fashion (though for the last two of our theorems, we will rely on an embedding trick to make the vector potential operator self-adjoint again).  To illustrate this mechanism, we begin with a simple blowup result for a (compressible) one-dimensional equation (a variant of the inviscid Burgers equation).  This result will not be directly used elsewhere in the paper, but may help illustrate the basic strategy of the arguments in subsequent sections.

\begin{proposition}[One-dimensional non-self-adjoint blowup]  Let $\theta_0\colon \R \to \R$ be a smooth function with $\theta_0(0)=0$ and $\theta_0(x)=1$ for all $x > 1/4$.  Then there does not exist a smooth bounded solution $u, \theta\colon [0,1] \times \R \to \R$ to the system
\begin{align}
\partial_t \theta + u \partial_x \theta &= 0 \label{uth} \\
u(t,x) &= -\theta(t,2x)\label{uth-2}
\end{align}
with initial data $\theta(0,x) = \theta_0(x)$.
\end{proposition}

Note that the negative dilation map that sends a function $x \mapsto \theta(x)$ to the function $x \mapsto -\theta(2x)$ is non-self-adjoint.  The system \eqref{uth}, \eqref{uth-2} transports the field $\theta$ at a position $x$ with a velocity that depends on the value of the field $\theta$ at the position $2x$; however, due to the non-self-adjointness, the value of $\theta$ at $x$ has no direct impact on the dynamics of $\theta$ at $2x$.  This one-way causality makes it easy to force the $\theta=1$ portion of the solution to collide with the $\theta=0$ portion to create the desired singularity; the point is that the ``front'' of the $\theta=1$ portion is being driven by the ``bulk'' of that portion, without any feedback in the opposite direction. This basic dynamic will also power all the rest of the blowup arguments in this paper.

\begin{proof}  Suppose for contradiction that there are $u, \theta$ with the claimed properties.  We use the barrier method, introducing a time-varying barrier $\Omega(t)$ which, on its boundary, expands slower than the velocity field.  More precisely, for each time $t \in [0,1]$, let $\Omega(t) \subset \R$ denote the half-line
$$ \Omega(t) \coloneqq [(1-t)/2, +\infty),$$
thus $\Omega(t)$ expands outwards at speed $1/2$ until it reaches the origin at time $t=1$; see Figure \ref{fig:1d}.
Let $T_*$ denote the supremum of all the times $0 \leq T_* \leq 1$ such that $\theta(t,x)=1$ for all $0 \leq t \leq T_*$ and $x \in \Omega(t)$.  From the initial condition $\theta=\theta_0$, and the fact that $\theta$ is transported by the bounded velocity field $u$, we see that $0 < T_* \leq 1$.  By continuity we see that $\theta(T_*)$ equals $1$ on $\Omega(T_*)$.

From \eqref{uth}, \eqref{uth-2} applied to $x=0$, we have
$$ \partial_t \theta(t,0) = \theta(t,0) \partial_x \theta(t,0).$$
Since $\theta(0,0) = \theta_0(0)=0$ and $\theta$ is smooth, we conclude from Gronwall's inequality that $\theta(1,0)=0$.  Since $\Omega(1)$ contains $0$, we conclude that $T_*$ cannot equal $1$, thus $0 < T_* < 1$.

Let $x_* \coloneqq (1-T_*)/2$.  By construction of $T_*$ and continuity, we have $\theta(T_*,x_*)=1$, but $\theta(T_*,x) \neq 1$ for $x$ arbitrarily close to $x_*$.  On the other hand, we have $\theta(t, x) = 1$ for all $0 \leq t \leq T_*$ and $x \in \Omega(t)$.  Since $x_*$ lies on the boundary of $\Omega(T_*)$, which moves at a velocity of $-1/2$, and $\theta$ is transported by the velocity field $u$, we conclude (by the method of characteristics) that
\begin{equation}\label{utsx}
u(T_*,x_*) \geq -1/2,
\end{equation}
otherwise one could flow $\theta$ backwards in time from $T_*$ and conclude that $\theta(t,x) \neq 1$ for some $t$ slightly less than $T_*$ and some $x$ barely inside $\Omega(t)$.

On the other hand, we have $\theta(T_*,x) = 1$ for all $x \geq (1-T_*)/2$, and hence from \eqref{uth-2} we see that $u(T_*,x_*) = -1$.  This contradicts \eqref{utsx} and gives the claim.
\end{proof}

	\begin{figure} [t]
\centering
\includegraphics[width=\textwidth]{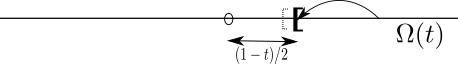}
\caption{The region $\Omega(t)$ (to the right of the solid bracket) and a slightly later region $\Omega(t+dt)$ (to the right of the dotted bracket).  The active scalar $\theta(t)$ is known to equal one on $\Omega(t)$, and to vanish at the origin (depicted here by a small circle).  The curved arrow from $2x$ to $x$ represents the one-way causality of the (non-self-adjoint) negative dilation operator in \eqref{uth-2} that sends $x \mapsto \theta(x)$ to $x \mapsto -\theta(2x)$. 
}
\label{fig:1d}
\end{figure}

\begin{remark}  The above argument suggests that, \emph{at best}, the solution $u$ will survive up to time $1$, and for times $t$ close to $1$ it will equal $1$ on the region $\Omega(t)$ and vanish at and to the left of $0$.  However, as the proof of the above proposition is by contradiction, it does not preclude the possibility that the solution $u$ in fact blows up sooner, and possibly with a qualitatively different dynamics\footnote{In the case when $\theta$ is non-negative and vanishing to the left of the origin, it may be possible to analyse the solution more carefully using some variant of the method of characteristics to obtain more definitive control on the blowup, for instance it seems possible to show that blowup at a point $x_0>0$ cannot occur if the solution remained regular in the region $x \geq 2x_0$, which on iteration suggests that blowup can only occur at (or to the left of) the origin.  We thank an anonymous referee for this observation.} than the one suggested here.  Similarly, the arguments used to prove the main theorems in our paper suggest a possible blowup mechanism, but do not ensure that this mechanism actually occurs because the solution may in fact blow up sooner, and in a different fashion, from that mechanism.
\end{remark}

\section{A non-self-adjoint blowup of an SQG-type equation}\label{nonself}

We now give a two-dimensional version of the argument in the previous section, establishing finite time blow up of an SQG type equation with a non-self-adjoint vector potential operator $A$.  The construction here can be viewed as a simplified version of the three-dimensional blowup construction used to establish Theorem \ref{first-blow}, and will also be embedded directly into the three-dimensional blowup constructions in Theorem \ref{second-blow} and Theorem \ref{third-blow}.

Consider the generalised Euler equation \eqref{vort-1}, \eqref{vort-2} on $\R^2$.  We formally write the vector potential operator $A\colon B_2(\R^2) \to \Gamma^2(\R^2)$ in coordinates as
$$ A ( \theta dx^1 \wedge dx^2 ) = (A_0 \theta) \frac{d}{d x^1} \wedge \frac{d}{d x^2}$$
for some linear operator $A_0\colon \Lambda_0(\R^2) \to \Lambda_0(\R^2)$ and all scalar functions $\theta\colon \R^2 \to \R$.  If we write the fields $\omega,u$ in coordinates as
$$ \omega = \theta dx^1 \wedge dx^2$$
and
$$ u = u^1 \frac{d}{d x^1} + u^2 \frac{d}{d x^2}$$
we thus arrive at the active scalar system
\begin{align}
\partial_t \theta + u^1 \partial_1 \theta + u^2 \partial_2 \theta &= 0 \label{ask}\\
u^1 &= \partial_2 (A_0\theta) \label{u1d}\\
u^2 &= - \partial_1 (A_0\theta).\label{u2d}
\end{align}
As noted in the introduction, the SQG equation corresponds to the case $A_0 = \Delta^{-1/2}$.
We now construct an operator $A_0$ which will behave like\footnote{In fact, it will be almost be a pseudodifferential operator in the exotic symbol class $S^{-1}_{1,1}$, as defined in \cite[Chapter VII]{stein:large}, in that the symbol obeys a large but finite number of the estimates required for this class.  We will not prove or use this fact here.} a non-self-adjoint variant of $\Delta^{-1/2}$, as follows.  We will need some cutoff functions:

\begin{itemize}
\item A Littlewood-Paley type cutoff $\gamma\colon \R \to \R$ which is smooth, non-negative, supported on $[1/2,2]$, and obeys the identity 
\begin{equation}\label{lp}
\sum_{j \in \Z} \gamma(2^j x) = 1
\end{equation}
for all $x>0$;
\item A smooth function $\psi\colon \R \to \R$ supported on $[-20,20]$ that equals $1$ on $[-10,10]$ and obeys the moment conditions $\int_\R \psi(x) P(x)\ dx = 0$ for all polynomials $P\colon \R \to \R$ of degree at most $1000$;
\item A smooth function $\varphi\colon \R^2 \to \R$ supported on $B_{\R^2}((0,10),1) \cup B_{\R^2}((0,-10),1)$ such that $\int_{B_{\R^2}((0,10),1)} \varphi(x)\ dx = 1$, but such that $\int_{\R^2} \varphi(x) P(x)\ dx = 0$ for all polynomials $P\colon \R \to \R$ of degree at most $1000$.
\end{itemize}
It is not difficult to construct examples of such cutoff functions $\gamma, \psi,\varphi$.  The moment conditions on $\psi, \varphi$ will not be needed in this section, but will become useful in Sections \ref{embed-sec}, \ref{period-sec}, when verifying that certain vector potential operators $A$ constructed using these functions are reasonable.

Let $M \geq 1$ be a sufficiently large constant (depending on $\gamma, \psi, \varphi$).  
The operator $A_0$ will now be defined for locally integrable $\theta$ as
\begin{equation}\label{stack}
 A_0(\theta)(x^1,x^2) \coloneqq\frac{2}{M} \sum_{k=0}^\infty 2^{2k} x^1 \gamma(2^k x^2) \psi( 2^k x^1 ) \int_{\R^2} \theta(y) \varphi(2^k y)\ d\operatorname{vol}(y).
\end{equation}
Note that if $\theta$ is supported on the upper half-plane $\{ (x^1,x^2): x^2 \geq 0 \}$, then the value of $A_0(\theta)$ near $(0,2^{-k})$ for some $k \geq 0$ is mostly driven by the behaviour of $\theta$ near $(0, 10 \times 2^{-k})$.  This is analogous to how, in the previous section, the value of the velocity field $u$ at a position $x$ was driven by the active scalar $\theta$ at position $2x$.  Roughly speaking, the operator $A_0$ is normalised so that it will produce a downward velocity $(\partial_2 A_0 \theta, -\partial_1 A_0 \theta)$ of $(0,-\frac{2}{M})$ near $(0,2^{-k})$ whenever the active scalar $\theta$ is equal to $1$ near $(0,10 \times 2^{-k})$.

We now have the following blowup result:

\begin{proposition}[Finite time blowup]\label{sqgb}  Let $\theta_0:\R^2 \to \R$ be smooth, compactly supported, vanishing in the half-plane $\{ (x^1,x^2): x^2 \leq 0 \}$, and equal to $1$ on the trapezoid $R \coloneqq \{ (x^1,x^2): \frac{1}{2M} \leq x^2 \leq 100; |x^1| \leq x^2 \}$.  Then there does not exist continuously differentiable and compactly supported fields $\theta, u^1, u^2\colon [0,1] \times \R^2 \to \R$ solving \eqref{ask}, \eqref{u1d}, \eqref{u2d} with $\theta(0)=\theta_0$.
\end{proposition}

We now prove this proposition.  Let $\theta_0$ be as in the proposition, and suppose for contradiction that such fields $\theta, u^1, u^2$ exist.  From \eqref{stack} we see that for all $0 \leq t \leq 1$, $A_0(\theta(t))$ vanishes on the half-space $\{ (x^1,x^2): x^2 < 0\}$, so by \eqref{u1d}, \eqref{u2d} we conclude that the velocity fields $u^1(t),u^2(t)$ do also; from \eqref{ask} and the vanishing of $\theta_0$ we conclude that $\theta(t)$ also vanishes here.  By continuity we thus have
\begin{equation}\label{thunk}
\theta(t,x^1,x^2)=0
\end{equation}
for all $0 \leq t \leq 1$ and $x^2 \leq 0$.

To obtain the required contradiction, we again use the barrier method.  For each time $t \in [0,1]$, let $\Omega(t) \subset \R^2$ denote the truncated hyperbolic region
\begin{equation}\label{omtdef}
 \Omega(t) \coloneqq \left\{ (x^1,x^2): \sqrt{((1+t)x^1)^2 + \left(\frac{1-t}{M}\right)^2} \leq x^2 \leq 20 \right\}
\end{equation}
(see Figure \ref{fig:parab}).  Informally, $\Omega(t)$ describes the region where we will be able to force $\theta(t)$ to take the value of $1$. 
Note that as $t$ increases from zero to one, the vertex $(0, \frac{1-t}{M})$ of this region is moving outwards (towards the origin) at a constant speed $\frac{1}{M}$, but the middle portion of the boundary (where $x^1,x^2$ are comparable to $1$) is instead moving inwards due to the narrowing of the hyperbola bounding $\Omega(t)$.  These dynamics are chosen to match the bounds we will be able to establish on the velocity field $u$ on the boundary of this domain.

	\begin{figure} [t]
\centering
\includegraphics[width=\textwidth]{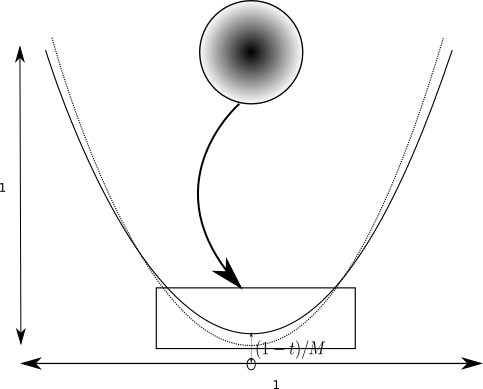}
\caption{A schematic depiction of $\Omega(t)$ (the region above the solid hyperboa, with the upper boundary $x^2 = 20$ out of view), together with a slightly later version $\Omega(t+dt)$ (the region above the dotted hyperbola).  The origin $(0,0)$ is marked by a small circle.  The rectangle and large disk represent the supports of the two functions $(x^1,x^2) \mapsto x^1 \gamma(2^k x^2) \psi( 2^k x^1 )$ and $y \mapsto \varphi(2^k y)$ respectively that occur in \eqref{stack}, for the value of $k$ that is of most importance at the time $t$.   (Actually, $y \mapsto \varphi(2^k y)$ also has a component supported below the $x^1$ axis, but this component will not be of relevance since $\theta$ vanishes there thanks to \eqref{thunk}.) The active scalar $\theta(t)$ is known to equal one on $\Omega(t)$, and to vanish at and below the $x^1$ axis.  The curved arrow signifies the one-way causality of the non-self-adjoint operator $A_0$.
}
\label{fig:parab}
\end{figure}

As in the previous section, let $T_*$ denote the supremum of all the times $0 \leq T_* \leq 1$ such that $\theta(t)$ is equal to $1$ on $\Omega(t)$ for all $0 \leq t \leq T_*$.  By hypothesis, $\theta_0$ equals $1$ on $\Omega(0)$, and so $0 \leq T_* \leq 1$.  In fact, since $\omega_0$ equals $1$ on a neighbourhood of $\Omega(0)$, and $\theta$ is transported by the continuous vector field $u^1 \frac{d}{dx^1} + u^2 \frac{d}{dx^2}$ thanks to \eqref{ask}, we have $T_* > 0$.  By continuity we see that $\theta(T_*)$ equals $1$ on $\Omega(T_*)$.  Since $\Omega(1)$ contains the origin, we conclude from \eqref{thunk} that $T_* < 1$.  Thus we have $0 < T_* < 1$.

As $\theta$ is transported continuously by $u \coloneqq u^1 \frac{d}{dx^1} + u^2 \frac{d}{dx^2}$, and $\Omega(t)$ is compact and varies continuously with $t$, there must exist a point $x_* = (x^1_*,x^2_*)$ on the boundary of $\Omega(T_*)$ which is also on the boundary of the set $\{x:\theta(T_*,x)=1\}$.  On the other hand, from \eqref{stack} we see that $A_0(\theta(t))$ is supported in the region $\{ (x^1,x^2): x^2 \leq 2 \}$ for all $t$, and so from \eqref{ask}, \eqref{u1d}, \eqref{u2d} we see that $\theta(t,x^1,x^2) = \theta_0(x^1,x^2)$ whenever $x^2 \geq 2$ and $0 \leq t \leq 1$.  Since $\theta_0=1$ on the trapezoid $R$, we conclude that $\theta(T_*)$ equals $1$ in a neighbourhood of $\{ (x^1,x^2) \in \Omega(T_*): x^2 \geq 2 \}$.  Thus we must have $x^2_* < 2$, and hence by \eqref{omtdef} we have
\begin{equation}\label{x21}
x^2_* = \sqrt{((1+T_*)x^1_*)^2 + \left(\frac{1-T_*}{M}\right)^2}.
\end{equation}
In particular,
\begin{equation}\label{xt}
\frac{1-T_*}{M} \leq x^2_* < 2.
\end{equation}
From \eqref{x21} we have
$$ \partial_t \left.\sqrt{((1+t)x^1)^2 + \left(\frac{1-t}{M}\right)^2}\right|_{t=T_*} = \frac{(1+T_*) (x^1_*)^2 - (1-T_*) / M^2}{x^2_*} $$
and thus $\Omega(T_*)$ expands outward at $(x^1_*,x^2_*)$ at velocity
\begin{equation}\label{nora}
\frac{(1+T_*) (x^1_*)^2 - (1-T_*) / M^2}{x^2_*} n^2
\end{equation}
where $n^2 < 0$ is the $x^2$ component of the outward unit normal $n$ of $\Omega(T_*)$ at $(x^1_*,x^2_*)$ (this expansion becomes negative for large $x^1_*$).  Since $\theta(t)$ is equal to $1$ on $\Omega(t)$ for $t \leq T_*$ and is transported by $u$, but $\theta(T_*,x^1,x^2)$ is not equal to $1$ for $(x^1,x^2)$ arbitrarily close to $(x^1_*,x^2_*)$, we conclude (on tracing characteristics backwards in time from $T_*$) the inequality
\begin{equation}\label{help}
n \cdot u(T_*,x^1_*,x^2_*) \leq \frac{(1+T_*) (x^1_*)^2 - (1-T_*) / M^2}{x^2_*} n^2,
\end{equation}
that is to say the outward normal velocity cannot exceed the expansion of the barrier at $(T_*,x_*)$.

To compute the left-hand side of \eqref{help}, we first compute $A_0(\theta)(T_*,x^1,x^2)$ for $(x^1,x^2)$ in a small neighbourhood of $(x^1_*,x^2_*)$.  We expand this quantity using \eqref{stack}.  From the support of $\gamma$, we need only restrict attention to those $k$ for which $2^{-k} \geq \frac{1}{2} x^2$; in particular, from \eqref{xt} and the restriction $k \geq 0$ we have
\begin{equation}\label{ja}
1 \geq 2^{-k} \geq \frac{1-T_*}{3M}
\end{equation}
for $(x^1,x^2)$ sufficiently close to $(x^1_*,x^2_*)$.
The function $y \mapsto \varphi(2^k y)$ in the integrand in \eqref{stack} is supported in $B_{\R^2}((0,10 \times 2^{-k}), 2^{-k}) \cup B_{\R^2}((0,-10 \times 2^{-k}), 2^{-k})$.  By \eqref{thunk}, $\theta$ vanishes on the latter ball $B_{\R^2}((0,-10 \times 2^{-k}), 2^{-k})$.  By \eqref{ja}, the ball $B_{\R^2}((0,10 \times 2^{-k}), 2^{-k})$ is contained in the truncated cone
$$ \left\{ (x^1,x^2): 9 \times \frac{(1-T_*)}{3M} \leq x^2 \leq 11; |x^1| \leq \frac{1}{9} x^2 \right\}$$
which can be seen in turn from \eqref{omtdef} and a brief calculation to lie in $\Omega(T_*)$.  By construction of $T_*$ and continuity, we have $\theta(T_*,y) = 1$ for all $y$ in $B_{\R^2}((0,10 \times 2^{-k}), 2^{-k})$, and hence
\begin{align*}
\int_{\R^2} \theta(y) \varphi(2^k y)\ d\operatorname{vol}(y) &= \int_{B_{\R^2}((0,10 \times 2^{-k}), 2^{-k})} \varphi(2^k y)\ d\operatorname{vol}(y)\\
&= 2^{-2k} \int_{B_{\R^2}((0,10),1)} \varphi(y)\ d\operatorname{vol}(y)\\
&= 2^{-2k}
\end{align*}
by construction of $\varphi$.  Inserting this into \eqref{stack}, we conclude that
$$ A_0(\theta)(T_*,x^1,x^2) = \frac{2}{M} \sum_{k=0}^\infty x^1 \gamma(2^k x^2) \psi( 2^k x^1 ).$$
From \eqref{x21} we have $|x^1_*| \leq x^2_*$, and hence 
$$ |x^1| \leq 2 |x^2|$$
for $(x^1,x^2)$ sufficiently close to $(x^1_*,x^2_*)$.  From the construction of $\gamma$ and $\psi$, we conclude that $\psi(2^k x^1)$ equals $1$ whenever $\gamma(2^k x^2)$ is non-zero.  Thus
\begin{equation}\label{ss}
 A_0(\theta)(T_*,x^1,x^2) = \frac{2}{M} x^1 \sum_{k=0}^\infty  \gamma(2^k x^2).
\end{equation}
If $x^2_* \leq 1/2$, then the constraint $k \geq 0$ can be dropped, and from \eqref{lp} we thus have $A_0(\theta)(x^1,x^2) = \frac{2}{M} x^1$.  From \eqref{u1d}, \eqref{u2d} we thus have
\begin{equation}\label{u1u2}
u(T_*,x^1_*,x^2_*) = \left(0, -\frac{2}{M}\right),
\end{equation}
and hence by \eqref{help} and the negativity of $n^2$
$$ \frac{2}{M} \leq  \frac{ (1-T_*) / M^2 - (1+T_*) (x^1_*)^2}{x^2_*}.$$
But this contradicts \eqref{xt} (discarding the negative term $(1+T_*) (x^1_*)^2$).  Thus we must have $x^2_* \geq 1/2$.  But then the quantity $\gamma(2^k x^2)$ only is non-zero for $k=0,1$.  Meanwhile, from \eqref{xt}, \eqref{x21} we have 
\begin{equation}\label{x12}
1 \lesssim |x^1_*|, x^2_* \lesssim 1.
\end{equation}
From \eqref{ss}, \eqref{u1d}, \eqref{u2d} we then have the crude bounds
$$ u(T_*,x^1_*,x^2_*) = O(1/M).$$
From \eqref{help} we thus have
$$ (-n^2) \frac{(1+T_*) (x^1_*)^2 - (1-T_*) / M^2}{x^2_*} \leq O(1/M).$$
On the other hand, from \eqref{x12} we have for $M$ large enough that $-n^2 \gtrsim 1$ and $\frac{(1+T_*) (x^1_*)^2 - (1-T_*) / M^2}{x^2_*} \gtrsim 1$, giving the required contradiction.  This concludes the proof of Proposition \ref{sqgb}.

\section{A stable, non-self-adjoint blowup}\label{3d-nonself}

In this section we prove Theorem \ref{first-blow}, using a three-dimensional variant of the argument\footnote{A simplified version of this argument, involving a non-compactly supported initial vorticity $\omega_0$, can be found at {\tt terrytao.wordpress.com/2016/02/01}.} used to prove Proposition \ref{sqgb}.  We will need a large constant $M > 1$ to be chosen later.  Now we select initial data $\omega_0 \in B_2(\R^3)$ with the following properties:

\begin{itemize}
\item $\omega_0$ is smooth and compactly supported.  When restricted to the ball $B_{\R^3}(0,100M)$, $\omega_0$ supported on the cylindrical region $\{ (x^1,x^2,x^3) \in B_{\R^3}(0,100M): (x^1)^2+(x^2)^2 \leq \frac{1}{M} \}$.
\item For any $-50M \leq x^3 \leq 50M$, one has the constant circulation 
\begin{equation}\label{const}
\int_{\{(x_1,x_2,x_3): (x^1)^2+(x^2)^2 \leq \frac{1}{M}\}} \omega_0 = 1
\end{equation}
where we give the disk $\{(x^1,x^2,x^3): (x^1)^2+(x^2)^2 \leq \frac{1}{M}\}$ the orientation of $\frac{d}{d x^1} \wedge \frac{d}{d x^2}$.
\end{itemize}

To create such an $\omega_0$, one can for instance set $\omega_0 = d \lambda$, where $\lambda \in C^\infty_c \cap \Lambda_1(\R^3)$ is chosen to be equal to the closed form
$$ \lambda = \frac{1}{2\pi} \frac{ x^1 dx^2 - x^2 dx^1 }{(x^1)^2+(x^2)^2}$$
in the region $\{ (x^1,x^2,x^3) \in B_{\R^3}(0,100M): (x^1)^2+(x^2)^2 \geq \frac{1}{M} \}$, but otherwise arbitrary outside of this region; the constant circulation \eqref{const} then follows from Stokes' theorem.  

Next, we construct the vector potential operator $A$.  We introduce the cylindrically radial variable
$$ r \coloneqq \sqrt{(x^1)^2 + (x^2)^2} $$
and the associated cylindrically radial vector field
$$
\frac{d}{d r} \coloneqq \frac{x^1}{r} \frac{d}{d x^1} + \frac{x^2}{r} \frac{d}{d x^2},
$$
defined away from the $x^3$ axis $\{ (0,0,x^3): x^3 \in \R\}$.
Our construction will be designed so that the velocity field $u = \delta A \omega$ will be equal to the inward cylindrically radial vector field 
$$-\frac{1}{Mr} \frac{d}{d r} = \frac{- x^1 \frac{d}{d x^1} - x^2 \frac{d}{d x^2}}{Mr^2}$$ 
in a certain key portion of physical space $\R^3$.  Observe that away from the $x^3$-axis, this field is divergence free, and can be written in turn as a divergence
\begin{equation}\label{dive}
 -\frac{1}{Mr} \frac{d}{d r} = \delta \frac{-x^1 x^3 \frac{d}{d x^1} \wedge \frac{d}{d x^3} 
 - x^2 x^3 \frac{d}{d x^2} \wedge \frac{d}{d x^3}}{Mr^2}.
\end{equation}
For technical reasons (having to do with ensuring that the vector potential operator $A$ we will construct is reasonable), we need to replace the $2$-vector field $\frac{-x^1 x^3 \frac{d}{d x^1} \wedge \frac{d}{d x^3} 
 - x^2 x^3 \frac{d}{d x^2} \wedge \frac{d}{d x^3}}{r^2}$ appearing on the right-hand side of \eqref{dive} by a variant $\alpha$ that enjoys better moment vanishing conditions.  More precisely, by inserting a suitable cutoff in the angular variable, one can find a $2$-vector field $\alpha \in \Gamma^2(\R^3)$ that is smooth away from the origin and homogeneous of degree zero, such that
$$ \alpha = \frac{-x^1 x^3 \frac{d}{d x^1} \wedge \frac{d}{d x^3} 
 - x^2 x^3 \frac{d}{d x^2} \wedge \frac{d}{d x^3}}{r^2} $$
and hence
\begin{equation}\label{ded}
 \delta \alpha = - \frac{1}{r}\frac{d}{d r} 
\end{equation}
in the exterior cone region $\{ (x^1,x^2,x^3): r >|x^3| \}$, and such that all moments of $\alpha$ vanish to order $1000$ (say) on each sphere, or in other words the three components $\alpha^{12}, \alpha^{13}, \alpha^{23}$ of $\alpha$ are such that
$$ \int_{S^2} \alpha^{ij}(\theta) P(\theta) d\theta = 0$$
for all polynomials $P\colon \R^3 \to \R$ of degree at most $1000$, where $S^2$ is the unit sphere in $\R^3$ and $d\theta$ denotes surface measure.

We introduce a smooth dyadic partition of unity of Littlewood-Paley type, writing
$$ 1 = \sum_{k \in \Z} \psi( 2^k x ) $$
for a suitable smooth, spherically symmetric function $\psi\colon \R^3 \to \R$ (not depending on $M$) supported on the annulus $\{ x: 1/2 \leq |x| \leq 2 \}$.  Clearly we can then decompose $\alpha = \sum_{k \in \Z} \alpha_k$, where $\alpha_k \in C^\infty_c \cap \Gamma^2(\R^3)$ is defined by the formula
$$ \alpha_k(x) \coloneqq \alpha(x) \psi(2^k x).$$
Next, we let $\varphi\colon \R^3 \to \R$ be a smooth compactly supported function (not depending on $M$) of the form
\begin{equation}\label{varph}
 \varphi(x^1,x^2,x^3) = \varphi_{12}( x^1,x^2) \varphi_3(x^3)
\end{equation}
where $\varphi_{12}\colon \R^2 \to \R$ is a smooth spherically symmetric function supported on the disk $B_{\R^2}(0,20)$ that equals one on the disk $B_{\R^2}(0,10)$ and obeys the moment conditions
\begin{equation}\label{12mom}
\int_{\R^2} \varphi_{12}(x^1,x^2) P(x^1,x^2)\ d\operatorname{vol}(x) = 0
\end{equation}
for any polynomial $P$ of degree at most $1000$, and $\varphi_3\colon \R \to \R$ is a smooth function supported on $[1,2]$ with the normalisation
\begin{equation}\label{ting}
 \int_1^2 \varphi_3(x^3)\ dx^3 = 1.
\end{equation}
We define the vector potential operator $A$ by the formula
\begin{equation}\label{Adef}
 A \omega(x) \coloneqq \sum_{k=0}^\infty \frac{2^k}{M^2} \alpha_k(x) \int_{\R^3} \omega_{12}(y) \varphi( 2^k y / M )\ d\operatorname{vol}(y)
\end{equation}
where $\omega_{12}$ is the $dx^1 \wedge dx^2$ component of $\omega$.  This operator $A$ is designed so that $A \omega$ will equal $\frac{1}{M} \alpha$ in regions where $\omega$ has circulation equal to one.  In particular, $u = \delta A \omega$ will equal $-\frac{1}{Mr} \frac{d}{dr}$ in these regions.

It is easy to see that the sum defining $A \omega$ is absolutely convergent for $\omega \in C^\infty_c \cap B_2(\R^3)$.  One can write $A$ as an integral operator
$$ A \omega(x) = \int_{\R^3} K(x,y) \omega_{12}(y)\ d\operatorname{vol}(y)$$
where the kernel $K$ is given by the formula
$$ K(x,y) =  \sum_{k=0}^\infty \frac{2^k}{M^2} \alpha \psi(2^k x) \varphi( 2^k y / M )$$
(here we exploit the hypothesis that $\alpha$ is homogeneous of degree zero).  Since $\alpha \psi, \varphi$ are smooth and compactly supported, we see that $K$ obeys the bounds \eqref{nij} for all $0 \leq i,j \leq 100$ (with implied constants depending on $M$); indeed, one can even replace the quantity $|x-y|$ in \eqref{nij} by the larger quantity $2\max(|x|, |y|)$, and obtain bounds for arbitrary $i,j \geq 0$ if one allows the implied constant to depend on these parameters.  Now we show \eqref{pass} (again with bounds depending on $M$).  It will suffice to establish the slightly stronger bounds
$$
\| A \omega \|_{\dot H^{r+2}(\M)} \lesssim_M \| \omega \|_{\dot H^rk(\M)}
$$
for all $0 \leq r \leq 100$.
By duality, it suffices to establish the bounds
$$ \int_{\R^3} \langle \beta(x), A \omega(x)\rangle\ d\operatorname{vol}(x) \lesssim_M \| \omega \|_{\dot H^s(\R^3)} \| \beta \|_{\dot H^{-s-2}(\R^3)}$$
for any $\beta \in \dot H^{-s-2} \cap \Lambda_2(\R^3)$ and any $0 \leq s \leq 100$.  By Littlewood-Paley decomposition and Schur's test, it suffices to show that
$$ \int_{\R^3} \langle \beta(x), A \omega(x)\rangle\ d\operatorname{vol}(x) \lesssim_M \min( N_1^{-2}, N_1^{200} N_2^{-202} )
\| \omega \|_{L^2(\R^3)} \| \beta \|_{L^2(\R^3)}$$
whenever $\omega, \beta \in C^\infty_c \cap \Lambda_2(\R^3)$ have Fourier transforms supported on the annuli $\{ \xi: |\xi| \sim N_1 \}$ and $\{ \eta: |\eta| \sim N_2\}$ respectively for some $N_1,N_2 > 0$.  The left-hand side may be expanded as $\sum_{k=0}^\infty \frac{2^k}{M^2} X_k Y_k$, where
$$ X_k \coloneqq \int_{\R^3} \langle \beta(x), \alpha \psi(2^k x) \rangle\ d\operatorname{vol}(x)$$
and
$$ Y_k \coloneqq \int_{\R^3} \omega_{12}(y) \varphi(2^k y / M)\ d\operatorname{vol}(y).$$
From the smoothness and moment conditions on $\alpha \psi$, the Parseval identity, and Cauchy-Schwarz, we see that
$$ X_k \lesssim \min(N_1/2^k, 2^k/N_1)^{300} 2^{-3k/2} \| \beta \|_{L^2(\R^3)}$$
for any $k$; similarly
$$ Y_k \lesssim_M \min(N_2/2^k, 2^k/N_2)^{300} 2^{-3k/2} \| \omega \|_{L^2(\R^3)}.$$
Inserting these bounds and summing in $k$, one obtains the claim.

We can now prove Theorem \ref{first-blow} with this choice of $\omega_0$ and $A$.  Suppose for contradiction that there is a solution $\omega \in X^{10,2}, u \in Y^{10,2}$ to \eqref{vort-1}, \eqref{vort-2} with $s=10$ on the time interval $[0,1]$.  This is enough regularity to interpret the equations \eqref{vort-1}, \eqref{vort-2} in the classical sense.  The velocity $u$ is bounded in $\R^3 \times [0,1]$, and the vorticity $\omega$ is transported by $u$ and is compactly supported at time zero, and is thus compactly supported in all of $\R^3 \times [0,1]$. From 
\eqref{Adef}, \eqref{vort-2} we see that $u$ is supported in the ball $B_{\R^3}(0,2)$, and thus by \eqref{vort-2} $\omega$ is stationary outside of this ball.

	\begin{figure} [t]
\centering
\includegraphics[width=\textwidth]{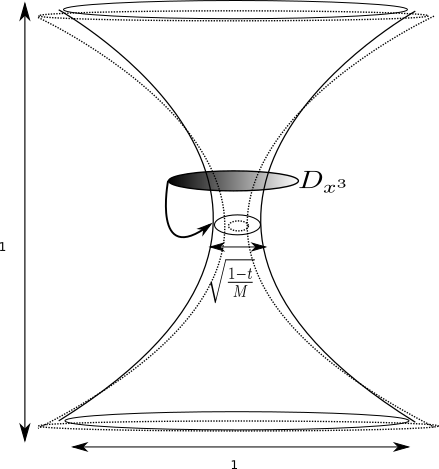}
\caption{A schematic depiction of $\Omega(t)$ (the region inside the hyperboloid, with the exterior of $B_{\R^3}(0,50M)$ out of view), together with a slightly later version $\Omega(t+dt)$ (the region inside the dotted hyperboloid).  The origin $(0,0)$ is marked by a small circle.  The vorticity $\omega$ is supported inside $\Omega(t)$, which allows one to use the Kelvin circulation theorem to calculate the circulation on disks $D_{x^3}$ such as the shaded one depicted here.  The curved arrow depicts the causal relationship in the non-self-adjoint vector potential $A$, which uses the circulation on disks such as $D_{x^3}$ to determine the velocity field in the ``neck'' of the hyperboloid.
}
\label{fig:squish}
\end{figure}

We once again use the barrier method.  For any time $0 \leq t \leq 1$, let $\Omega(t) \subset \R^3$ denote the region
\begin{equation}\label{omo}
 \Omega(t) \coloneqq (\R^3 \backslash B_{\R^3}(0,50M)) \cup \left\{ (x^1,x^2,x^3) \in \R^3: r \leq \sqrt{\frac{1-t}{M}+(1+t)(x^3)^2} \right\};
\end{equation}
inside the ball $B_{\R^3}(0,50M)$; this is the interior of a one-sheeted hyperboloid which pinches at the spatial origin $(0,0,0)$ at time $t=1$, while simultaneously becoming slightly wider away from this origin.  See Figure \ref{fig:squish}.  From the construction of $\omega_0$, we see that $\omega_0$ is supported in $\Omega(0)$; from continuity and the fact that the support of $\omega$ propagates at bounded speed, we see also that $\omega(t)$ is supported in $\Omega(t)$ for sufficiently small $t$.  Let $T_*$ be the supremum of all times $0 \leq T_* \leq 1$ for which $\omega(t)$ is supported in $\Omega(t)$ for all $0 \leq t \leq T_*$, then from the previous observation we have $0 < T_* \leq 1$, and from continuity $\omega(T_*)$ is supported in $\Omega(T_*)$.  We now claim the circulation identity
\begin{equation}\label{circ}
 \int_{D_{x^3}} \omega(t) = 1
\end{equation}
on the disk $D_{x^3} \coloneqq \{(x^1,x^2,x^3): r \leq 30M\}$
for all $-10M \leq x^3 \leq 10M$ and $0 \leq t \leq T_*$. For $t=0$, this follows from the construction of $\omega_0$.  The set of $0 \leq t \leq T_*$ for which the above bound holds is clearly closed in $t$.  Finally, if the above bound holds for some $0 \leq t < 1$, and $t'$ is a time slightly larger than $t$, then from conservation of circulation we have
$$ \int_S \omega(t') = 1$$
where $S$ is the image of the disk $D_{x^3}$ after flowing along the velocity field $u$ from time $t$ to time $t'$.  But if $t'$ is sufficiently close to $t$, $S$ is homologous to $D_{x^3}$ up to a thin annular strip outside of $\Omega(t')$, and so from Stokes theorem and the closed nature of $\omega$ we conclude that \eqref{circ} holds for all $t'$ slightly larger than $t$, and from a continuity argument we conclude that \eqref{circ} holds for all $0 \leq t \leq T_*$.

We can now exclude the case $T_*=1$, since in this case $\Omega(1)$ degenerates to a cone that only intersects the disk $D_0 = \{ (x^1,x^2,0): r \leq 30M \}$ at the origin $(0,0,0)$, contradicting \eqref{circ} and the regularity hypotheses on $\omega$.  Thus we have $0 < T_* < 1$.

By definition of $T_*$, and the continuity of $\omega$, there must be a point $x_* = (x^1_*,x^2_*,x^3_*)$ on the boundary of $\Omega(T_*)$ which is on the boundary of the support of $\omega(T_*)$.  Since $\omega(T_*)$ is equal to $\omega_0$ outside of $B_{\R^3}(0,2)$, and $\omega_0$ vanishes near the boundary of $\Omega(T_*)$, we must have $x_* \in B_{\R^3}(0,2)$.  From \eqref{omo} we conclude that the radial coordinate $r_* \coloneqq \sqrt{(x^1_*)^2 + (x^2_*)^2}$ is given by
$$
 r_* = \sqrt{\frac{1-T}{M}+(1+T)(x_*^3)^2}$$
which implies in particular that
\begin{equation}\label{cone}
\max\left( |x^3_*|, \sqrt{\frac{1-T_*}{M}} \right) \leq r \leq 2 \max\left( |x^3_*|, \sqrt{\frac{1-T_*}{M}} \right) 
\end{equation}
and hence by Pythagoras' theorem
\begin{equation}\label{cone-2}
\max\left( |x^3_*|, \sqrt{\frac{1-T_*}{M}} \right)  \leq  |x_*| \leq 3 \max\left( |x^3_*|, \sqrt{\frac{1-T_*}{M}} \right).
\end{equation}
On the other hand, if $n$ denotes the outward normal to $\Omega(T_*)$ at $x_*$, then since
\begin{align*}
\frac{d}{dt} \left.\sqrt{\frac{1-t}{M}+(1+t)(x^3_*)^2}\right|_{t=T_*} &= \frac{(x^3_*)^2 - \frac{1}{M}}{2} \frac{1}{\sqrt{\frac{1-T_*}{M}+(1+T_*)(x^3_*)^2}} \\
&= - \frac{\frac{1}{M} - (x^3_*)^2}{2r_*}
\end{align*}
we see that at $x_*$, $\Omega(T_*)$ is moving outwards at speed
$$ \frac{(x^3_*)^2 - \frac{1}{M}}{2r_*} n^r$$
where $n^r$ is the radial component of $n$ (note this component is negative, reflecting inwards motion, when $x^3$ is small).  Since $\omega$ is transported by $u$ and is supported on $\Omega(t)$ for all $t \leq T_*$, we thus have
\begin{equation}\label{dope}
n \cdot u(T_*,x_*) \geq \frac{(x^3_*)^2 - \frac{1}{M}}{2r_*} n^r.
\end{equation}

Now we compute the velocity field $u(T_*,x_*)$ at $(T_*,x_*)$.  By \eqref{vort-2}, \eqref{Adef} we have
\begin{equation}\label{uTx}
u(T_*,x_*) \coloneqq \sum_{j=0}^\infty \frac{2^j}{M^2} \delta \alpha_j(x_*) \int_{\R^3} \omega_{12}(T,y) \varphi( 2^j y / M )\ d\operatorname{vol}(y).
\end{equation}
The quantity $\delta \alpha_k(x_*)$ is only non-vanishing when
\begin{equation}\label{stam}
 2^{-k-1} \leq |x_*| \leq 2^{-k+1},
\end{equation}
so we may restrict to $k$ obeying these bounds.  By \eqref{varph}, the function $\varphi(2^k y/M)$ is only non-vanishing when
\begin{equation}\label{stub}
 2^{-k} M \leq y^3 \leq 2^{-k+1} M
\end{equation}
and
$$ r(y) \leq 20 M 2^{-k}$$
where $r(y) \coloneqq \sqrt{(y^1)^2+(y^2)^2}$ is the cylindrically radial component of $y$.
In particular $|y| \leq 30 M$ since $k \geq 0$.  Since $\omega(T_*)$ is supported in $\Omega(T_*)$, we conclude that $\omega_{12}(T,y) \varphi( 2^k y / M )$ is only non-vanishing when
$$  r(y) \leq \sqrt{\frac{1-T_*}{M}+(1+T_*)(y^3)^2}$$
which implies from the triangle inequality that
$$ r(y)   \leq \sqrt{\frac{1-T_*}{M}} + 2 |y^3|.$$
Using \eqref{stub}, \eqref{stam}, \eqref{cone-2} we have
$$ y^3 \geq \frac{1}{2} M |x_*| \geq \frac{1}{2} M \sqrt{\frac{1-T_*}{M}} $$
and hence by \eqref{stub}
$$ r(y)  \leq 3 y^3 \leq 6 M 2^{-k}.$$
Using \eqref{varph}, we then have
$$ \varphi(2^k y/M) = \varphi_3( 2^k y^3 / M )$$
and thus
$$
\int_{\R^3} \omega_{12}(T_*,y) \varphi( 2^k y / M )\ d\operatorname{vol}(y) = 
\int_\R \varphi_3(2^k y^3/M) \left(\int_{r(y) \leq 30M} \omega_{12}(T_*,y)\ dy^1 dy^2\right)\ dy^3.$$
Applying \eqref{circ} and \eqref{ting}, the right-hand side evaluates to $M/2^k$.  From \eqref{uTx} we conclude that
$$
u(T_*,x_*) \coloneqq \frac{1}{M} \sum_{k=0}^\infty \delta \alpha_k(x_*).$$
Suppose first that $|x_*| \leq 1/2$, then $\delta \alpha_k$ vanishes for $k<0$, and we conclude from \eqref{ded} that $u(T_*,x_*)$ is the inward vector field
$$
u(T_*,x) = -\frac{1}{Mr_*} \frac{d}{d r},
$$
and hence
$$ n \cdot u(T_*,x_*) = - \frac{1}{Mr_*} n^r.$$
Since $n^r$ is positive, this contradicts \eqref{dope}.  Thus we must have $1/2 \leq |x_*| \leq 2$, which from \eqref{cone-2} implies that $|x_3|$ is comparable to $1$.  Now we use the boundedness of $\alpha$ and its derivatives on this annulus to obtain the crude bound
$$
u(T_*,x_*) = O(1/M)
$$
$$ n \cdot u(T_*,x_*) = O(1/M).$$
On the other hand, in the region $1/2 \leq |x_*| \leq 2$, one checks from \eqref{omo} that $n^r$ is comparable to $1$, and this again contradicts \eqref{dope} for $M$ large enough.  This concludes the proof of Theorem \ref{first-blow}.

\section{Embedding SQG type equations into Euler type equations}\label{embed-sec}

\subsection{Formal calculations}\label{form}

To motivate our proof of Theorem \ref{second-blow}, we begin with the following observation that embeds solutions of SQG type equations on $\R^2$ into solutions of Euler type equations on $\R^2 \times \R/\Z$; a key feature of this embedding is that the vector potential operator $\tilde A$ on $\R^2 \times \R/\Z$ will always be formally self-adjoint, even when the vector potential operator $A$ on $\R^2$ is not.  In this subsection we ignore issues of regularity or decay in calculations, proceeding instead at a purely formal level.

Let $A\colon B_2(\R^2) \to \Gamma^2(\R^2)$ be a (formal) vector potential operator on $\R^2$.  We can write this operator in coordinates as
$$ A ( \theta dx^1 \wedge dx^2 ) = (A_0 \theta) \frac{d}{d x^1} \wedge \frac{d}{d x^2}$$
for all scalar functions $\theta\colon \R^2 \to \R$, and some linear operator $A_0\colon \Lambda_0(\R^2) \to \Lambda_0(\R^2)$.  Let $\omega, u$ solve the generalised Euler equations with vector potential operator $A$; writing in coordinates
$$ \omega = \theta dx^1 \wedge dx^2$$
and
$$ u = u^1 \frac{d}{d x^1} + u^2 \frac{d}{d x^2}$$
we thus arrive at the active scalar system \eqref{ask}, \eqref{u1d}, \eqref{u2d} from Section \ref{nonself}.
We can formally define the adjoint $A_0^*\colon \Lambda_0(\R^2) \to \Lambda_0(\R^2)$ of $A_0$ by requiring the formal identity
$$ \int_{\R^2} (A_0 \theta)(x) \theta'(x)\ d\operatorname{vol}(x) = \int_{\R^2} \theta(x) (A_0^* \theta')(x)\ d\operatorname{vol}(x)$$
for all $\theta, \theta' \in \Lambda_0(\R^2)$.

We now suppose we have a three-dimensional extension $\tilde A_0\colon \Lambda_0(\R^2 \times \R/\Z) \to \Lambda_0(\R^2 \times \R/\Z)$ of $A_0$, by which we mean a linear operator on $\Lambda_0(\R^2 \times \R/\Z)$ obeying the compatibility condition
\begin{equation}\label{ao}
 \tilde A_0( \theta \circ \pi ) \coloneqq (A_0 \theta) \circ \pi
\end{equation}
for all $\theta \in \Lambda_0(\R^2)$, where $\pi\colon \R^2 \times \R/\Z \to \R^2$ is the projection map $\pi(x^1,x^2,x^3) \coloneqq (x^1,x^2)$.  We also suppose that we have an adjoint operator $\tilde A_0^*\colon \Lambda_0(\R^2 \times \R/\Z) \to \Lambda_0(\R^2 \times \R/\Z)$ which extends $A_0^*$ in the sense that the analogue 
$$
 \tilde A_0^*( \theta \circ \pi ) \coloneqq (A_0^* \theta) \circ \pi
$$
of \eqref{ao} holds for all $\theta \in \Lambda_0(\R^2)$; we also assume that $\tilde A_0^*$ is the adjoint of $\tilde A_0$ in the sense that
\begin{equation}\label{adj}
 \int_{\R^2 \times \R/\Z} (\tilde A_0 \theta)(x) \theta'(x)\ d\operatorname{vol}(x) = \int_{\R^2} \theta(x) (\tilde A_0^* \theta')(x)\ d\operatorname{vol}(x)
\end{equation}
for all $\theta, \theta' \in \Lambda_0(\R^2 \times \R/\Z)$.  One could impose further properties on $\tilde A_0$ and $\tilde A_0^*$, for instance that they are invariant with respect to translations in the $x^3$ direction, but we will not need to do so for this formal calculation.  Heuristically, if $A_0$ (and hence $A_0^*$) are pseudodifferential operators of order $-1$, then we would expect to be able to select extensions $\tilde A_0, \tilde A_0^*$ to also be pseudodifferential operators of order $-1$; again, we will not enforce these requirements during this formal discussion.

We now formally define an operator $\tilde A\colon B_2(\R^2 \times \R/\Z) \to \Gamma^2(\R^2 \times \R/\Z)$ by the formula
\begin{equation}\label{awdef}
\begin{split}
\tilde A \omega & \coloneqq - \tilde A_0 \Delta^{-1} (\partial_1 \omega_{13} + \partial_2 \omega_{23} ) \frac{d}{d x^1} \wedge \frac{d}{d x^2} \\
&\quad + \partial_1 \Delta^{-1} \tilde A_0^* \omega_{12} \frac{d}{d x^1} \wedge \frac{d}{d x^3} \\
&\quad + \partial_2 \Delta^{-1} \tilde A_0^* \omega_{12} \frac{d}{d x^2}\wedge \frac{d}{d x^3} \\
&\quad + \Delta^{-1} \omega_{13} \frac{d}{d x^1} \wedge \frac{d}{d x^3} \\
&\quad + \Delta^{-1} \omega_{23} \frac{d}{d x^2} \wedge \frac{d}{d x^3} \\
&\quad + \Delta^{-1} \omega_{12} \frac{d}{d x^1} \wedge \frac{d}{d x^2}
\end{split}
\end{equation}
whenever $\omega \in B_2(\R^2 \times \R/\Z)$, where $\omega$ is expressed in coordinates as
$$ \omega = \omega_{12} dx^1 \wedge dx^2 + \omega_{13} dx^1 \wedge dx^3 + \omega_{23} dx^2 \wedge dx^3.$$
Here we pause to make a technical remark: because there are only two noncompact dimensions in $\R^2 \times \R/\Z$, the operator $\Delta^{-1}$ is not quite uniquely defined even on $C^\infty_c$ (the symbol $\frac{1}{4\pi |\xi|^2}$ is not absolutely integrable near the origin of the frequency space $\R^2 \times \Z$).  However, the ambiguity is only up to constant functions, which will not be an issue since every appearance of $\Delta^{-1}$ will eventually be combined with at least one spatial derivative.  For sake of concreteness, though, we fix an explicit choice\footnote{In the language of distributions, this corresponds to fixing an explicit interpretation of the symbol $\frac{1}{4\pi |\xi|^2}$ as a tempered distribution, which is well defined up to a constant multiple of the Dirac mass at the origin.} of $\Delta^{-1} \omega$ for $\omega \in C^\infty_c \cap \Lambda_0(\R^2 \times \R/\Z)$ by the formula
$$ \Delta^{-1} \omega(x) = \int_{\R^2 \times \R/\Z} \omega(x') K_1(x-x')\ d\operatorname{vol}(x')$$
where the fundamental solution $K_1(x)$ for $x \neq 0$ can be obtained via descent from the fundamental solution $\frac{1}{4\pi|x|}$ on $\R^3$ by the renormalised summation formula
\begin{equation}\label{kdef}
 K_1(x) \coloneqq \lim_{N \to \infty} \sum_{n=-N}^N \frac{1}{4\pi |\tilde x + (0,0,n)|} - \frac{\log N}{2\pi}
\end{equation}
where $\tilde x$ is an arbitrary lift of $x$ from $\R^2 \times \R/\Z$ to $\R^3$ (it is easy to see that the precise choice of lift is irrelevant).  Roughly speaking, this kernel behaves like $\frac{1}{4\pi |x|}$ when $|x|$ is small and like $\log |x|$ when $|x|$ is large.  Note that the convergence of the sum in \eqref{kdef} improves after taking at least one derivative; for instance, one has the absolutely convergent series representation
\begin{equation}\label{kdef-deriv}
 \nabla K_1(x) = \sum_{n=-\infty}^\infty \frac{-(\tilde x + (0,0,n))}{4\pi |\tilde x  + (0,0,n)|^3}.
\end{equation}

Since the Hodge Laplacian $\Delta$ is diagonalised by the basis $dx^1 \wedge dx^2, dx^1 \wedge dx^3, dx^2 \wedge dx^3$, one could also write the last three terms in \eqref{awdef} more compactly as $\tilde \eta^{-1} \Delta^{-1} \omega$, as per \eqref{asper}.  Observe that if $\tilde A_0$ is a pseudodifferential operator of order $-1$, then $\tilde A$ will be a pseudodifferential operator of order $-2$ (formally, at least); similarly, if $A_0$ and $A^*_0$ commute with translations in the $x^3$ direction, then so does $\tilde A$.

From definition and integration by parts it is clear that $\tilde A$ is formally self-adjoint in the sense of \eqref{fsa}.  Next, we introduce the $2$-form $\omega \in \Lambda_2(\R^2 \times \R/\Z)$ and the vector field $\tilde u \in \Gamma^1(\R^2 \times \R/\Z)$ at any given time by the formulae
\begin{align}
\omega &\coloneqq d(\tilde \theta dx^3) \nonumber \\
&= (\partial_1 \tilde \theta) dx^1 \wedge dx^3 + (\partial_2 \tilde \theta) dx^2 \wedge dx^3 \label{omega-def}\\
\tilde u &\coloneqq \tilde u^1 \frac{d}{d x^1} + \tilde u^2 \frac{d}{d x^2} - \tilde \theta \frac{d}{d x^3}.\label{u-def}
\end{align}
where $\tilde \theta \coloneqq \theta \circ \pi$, $\tilde u^1 \coloneqq u^1 \circ \pi$, $\tilde u^2 \coloneqq u^2 \circ \pi$ are the lifts of $\theta, u^1,u^2$ from $\R^2$ to $\R^2 \times \R/\Z$.  It is clear that $\omega$ is closed, and thus lies in $B^2(\R^2 \times \R/\Z)$.

We now claim

\begin{proposition}  $\omega$ and $\tilde u$ (formally) obey the generalised Euler equations \eqref{vort-1}, \eqref{vort-2} on $\R^2 \times \R/\Z$ with vector potential operator $\tilde A$.
\end{proposition}

\begin{proof}  We begin with \eqref{vort-1}.  By \eqref{awdef}, \eqref{omega-def} we have
\begin{align*}
\tilde A \omega &= - \tilde A_0 \Delta^{-1} (\partial_1 \partial_1 \tilde \theta + \partial_2 \partial_2 \tilde\theta )  \frac{d}{d x^1} \wedge \frac{d}{d x^2} \\
&\quad + \Delta^{-1} \partial_1 \tilde \theta \frac{d}{d x^1} \wedge \frac{d}{d x^3} \\
&\quad + \Delta^{-1} \partial_2 \tilde \theta \frac{d}{d x^2} \wedge \frac{d}{d x^3}.
\end{align*}
But since $\tilde \theta  = \theta \circ \pi$ is constant in the $x^3$ direction, we have from definition of the Hodge Laplacian that
\begin{equation}\label{hodge}
 \partial_1 \partial_1 \theta + \partial_2 \partial_2 \tilde \theta = - \Delta \tilde \theta.
\end{equation}
Taking divergences, and again noting that $\tilde \theta$ is constant in the $x^3$ direction, we have
\begin{align*}
\delta \tilde A \omega &= - \partial_1 \tilde A_0 \tilde \theta \frac{d}{d x^2} + \partial_2 \tilde A_0 \tilde \theta \frac{d}{d x^1} \\
&\quad + \partial_1 \Delta^{-1} \partial_1 \tilde \theta  \frac{d}{d x^3} \\
&\quad + \partial_2 \Delta^{-1} \partial_2 \tilde \theta  \frac{d}{d x^3}.
\end{align*}
From \eqref{u1d}, \eqref{u2d}, \eqref{ao} one has
$$ \tilde u^1 = \partial_2 \tilde A_0 \tilde \theta; \quad \tilde u^2 = - \partial_1 \tilde A_0 \tilde \theta; $$
inserting this and \eqref{hodge} into the above computation, we obtain \eqref{vort-2}.

Now we turn to \eqref{vort-1}.  From \eqref{omega-def} we have
$$ \omega = d\tilde \theta \wedge dx^3 $$
and hence by \eqref{dcom} and \eqref{leib1}, we have
$$ \partial_t \omega + {\mathcal L}_{\tilde u} \omega = d ( \partial_t \tilde \theta + {\mathcal L}_{\tilde u} \tilde \theta ) \wedge dx^3 - d\tilde \theta \wedge d({\mathcal L}_{\tilde u} x^3).$$
From \eqref{u-def} we have
$$ {\mathcal L}_{\tilde u} x^3 = - \tilde \theta $$
and hence
$$  d\tilde \theta \wedge d({\mathcal L}_{\tilde u} x^3) = - d\tilde  \theta \wedge d\tilde  \theta = 0.$$
Next, since $\tilde u = u \circ \pi - \tilde \theta \frac{d}{d x^3}$ and $\tilde \theta$ is constant in the $x^3$ variable, we have
$$  {\mathcal L}_{\tilde u} \theta = ( {\mathcal L}_{u} \theta) \circ \pi$$
and hence by \eqref{ask}
$$ \partial_t \tilde \theta + {\mathcal L}_{\tilde u} \tilde \theta = 0.$$
The claim \eqref{vort-1} follows.
\end{proof}

\subsection{Rigorous construction}

We now prove Theorem \ref{second-blow} rigorously.  Set $\M = \R^2 \times \R/\Z$, and let $\eps > 0$.  
Let $M > 1$ be sufficiently large depending on $\eps$.  Let $A_0\colon \Lambda_0(\R^2) \to \Lambda_0(\R^2)$ be the linear operator defined in 
\eqref{stack}, thus
$$
 A_0(\theta)(x^1,x^2) \coloneqq\frac{2}{M} \sum_{k=0}^\infty 2^{2k} x^1 \eta(2^k x^2) \psi( 2^k x^1 ) \int_{\R^2} \theta(y) \varphi(2^k y)\ d\operatorname{vol}(y).
$$
We then define the extension $\tilde A_0\colon \Lambda_0(\R^2 \times \R/\Z) \to \Lambda_0(\R^2 \times \R/\Z)$ by the formula
\begin{equation}\label{tao}
\tilde A_0(\theta)(x^1,x^2,x^3) \coloneqq\frac{2}{M} \sum_{k=0}^\infty 2^{2k} x^1 \eta(2^k x^2) \psi( 2^k x^1 ) \int_\R \int_{\R^2} \theta(y, x^3 + 2^{-k} z) \varphi(2^k y)\ d\operatorname{vol}(y) \kappa(z)\ dz
\end{equation}
where $\kappa\colon \R \to \R$ is a smooth function supported on $[-1/2,1/2]$ with $\int_\R \kappa(z)\ dz = 1$.  The purpose of the additional averaging in the $z$ variable is so that $\tilde A_0$ obeys the kernel estimates \eqref{nij} in the definition of a reasonable operator.

It is easy to see that the sum defining $\tilde A_0(\theta)$ is absolutely convergent for $\theta$ in $C^\infty_c \cap \Lambda_0(\R^2 \times \R/\Z)$; indeed, the summands have size $O_\theta( 2^{-j} )$.  It is also easy to verify the relation \eqref{ao} with $\theta \in C^\infty_c \cap \Lambda_0(\R^2)$.  The adjoint map $\tilde A_0^*\colon  \Lambda_0(\R^2 \times \R/\Z) \to \Lambda_0(\R^2 \times \R/\Z)$ is given by the formula
$$ 
\tilde A_0^*(\theta)(y^1,y^2,y^3) \coloneqq \frac{2}{M} \sum_{k=0}^\infty 2^{2k} \varphi(2^k y^1, 2^k y^2) \int_\R \int_{\R^2} x^1 \eta(2^k x^2) \psi( 2^k x^1 ) \theta(x, y^3 - 2^{-k} z)\ dx \kappa(z)\ dz;
$$
again, one can check that the sum defining $\tilde A_0^*$ is absolutely convergent for $\theta \in C^\infty_c \cap \Lambda_0(\R^2 \times \R/\Z)$, that $\tilde A_0^*$ is the adjoint of $\tilde A_0$ in the sense of \eqref{adj}, and that $\tilde A_0^*$ extends $A_0^*$.  Finally it is clear from construction that $\tilde A_0$ and $\tilde A_0^*$ are both invariant with respect to translations in the $x^3$ direction.

Now we establish

\begin{proposition}\label{pass-ok} $\tilde A$ is a $100$-reasonable vector potential operator.
\end{proposition} 

\begin{proof}
We first prove \eqref{pass}.  As in the previous section, it will suffice to establish the slightly stronger bounds
$$
\| \tilde A \omega \|_{\dot H^{k+2}(\M)} \lesssim \| \omega \|_{\dot H^k(\M)}
$$
for all $0 \leq k \leq 100$.

The claim is clear for the last three components of \eqref{awdef}, so we focus on the first three components.  By duality (and commuting $\Delta^{-1}$ with $\partial_1, \partial_2$) it will suffice to show that
\begin{equation}\label{traq}
\| \tilde A_0 \partial_l \omega \|_{\dot H^{s}(\R^2 \times \R/\Z)} \lesssim \| \omega \|_{\dot H^s(\R^2 \times \R/\Z)}
\end{equation}
for all $-102 \leq s \leq 102$, $l=1,2$, and $\omega \in \dot H^s(\M)$ (dropping the requirement that $\omega$ be divergence-free).
For future reference we note that we will in fact gain an extra factor of $1/M$, and show that
\begin{equation}\label{traq-2}
\| \tilde A_0 \partial_l \omega \|_{\dot H^{s}(\R^2 \times \R/\Z)} \lesssim \frac{1}{M} \| \omega \|_{\dot H^s(\R^2 \times \R/\Z)}.
\end{equation}
Unwinding the definition of the Sobolev norms, it suffices to show that
$$ \| \Delta^{s/2} \tilde A_0 \partial_l \Delta^{-s/2} f \|_{L^2(\R^2 \times \R/\Z)} \lesssim \frac{1}{M} \| f \|_{L^2(\R^2 \times \R/\Z)}$$
for all $f \in L^2(\R^2 \times \R/\Z)$.  By Minkowski's inequality and translation invariance, it suffices to prove this with $\tilde A_0$ replaced by the variant operator $\tilde A'_0$ defined by
$$
\tilde A'_0(\theta)(x^1,x^2,x^3) \coloneqq\frac{2}{M} \sum_{k=0}^\infty 2^{2k} x^1 \gamma(2^k x^2) \psi( 2^k x^1 ) \int_{\R^2} \theta(y, x^3) \varphi(2^k y)\ d\operatorname{vol}(y)
$$
thus $\tilde A'_0$ simply applies the operator $A_0$ on each $x^3$ slice of $\R^2 \times \R/\Z$.
Taking Fourier coefficients in the $\R/\Z$ coordinate (noting that $A'_0$ and $\partial_i$ commute with this operation), it suffices to show the two-dimensional estimate
$$ \| (E + \Delta)^{s/2} A_0 \nabla (E+\Delta)^{-s/2} f \|_{L^2(\R^2)} \lesssim \frac{1}{M} \| f \|_{L^2(\R^2)}$$
for all $f \in L^2(\R^2)$ and $E \geq 0$, where $\Delta$ now denotes the Hodge Laplacian on $\R^2$ rather than $\R^2 \times \R/\Z$.

Fix $E \geq 0$.  By duality, it suffices to establish the bound
$$ |\langle A_0 \nabla (E+\Delta)^{-s/2} f , (E + \Delta)^{s/2} g \rangle| \lesssim \frac{1}{M} \| f \|_{L^2(\R^2)} \| g \|_{L^2(\R^2)}
$$
for $f,g \in L^2(\R^2)$.  By \eqref{stack} and integration by parts, the left-hand side is $-\frac{1}{M} \sum_{k=0}^\infty X_k Y_k$, where
$$
X_k \coloneqq 2^k \int_{\R^2} (E+\Delta)^{-s/2} f(y) (\nabla \varphi)(2^k y)\ d\operatorname{vol}(y)
$$
and
$$ Y_k \coloneqq 2^{k} \int_{\R^2} (E+\Delta)^{s/2} g(x)  2^k x^1 \gamma(2^k x^2) \psi( 2^k x^1 )\ d\operatorname{vol}(x).$$
The functions $y \mapsto \nabla \varphi(y)$ and $(x^1,x^2) \mapsto x^1 \gamma(x^2) \psi(x^1)$ are smooth and compactly supported, and orthogonal to all polynomials of degree up to $1000$, thus their Fourier transforms are Schwartz functions that vanish to order $1000$ at the origin.  From this, Plancherel's theorem, and Cauchy-Schwarz, we see that
$$
X_k \lesssim \sum_N (E + N^2)^{-s/2} \min( N/2^k, 2^k/N )^{1000} \| P_N f \|_{L^2(\R^2)}$$
and
$$
Y_k \lesssim \sum_M (E + M^2)^{-s/2} \min( M/2^k, 2^k/M )^{1000} \| P_M g \|_{L^2(\R^2)}$$
where $N,M$ range over the dyadic numbers $2^n, n \in \Z$, and $P_N$ denotes the Fourier projection to frequencies $N \leq |\xi| \leq 2N$.  Multiplying and summing in $k$ and using the hypothesis $|k| \leq 102$, we conclude that
$$ \sum_{k=0}^\infty X_k Y_k \lesssim \sum_N \sum_M \min( M/N, N/M )^{100} \| P_N f \|_{L^2(\R^2)} \| P_M g \|_{L^2(\R^2)}$$
and the claim now follows from Schur's test and Plancherel's theorem.

Now we prove \eqref{nij}.  We need to show that the integral kernel of $\tilde A$ obeys the bounds
$$
 |\nabla^i_x \nabla^j_y K(x,y)| \lesssim \max( |x-y|^{-i-j-1}, |x-y|^{-i-j} )
$$
for $0 \leq i,j \leq 100$ with $i+j \geq 1$.  The contribution of the last three components of $\tilde A$ in \eqref{awdef} are acceptable after differentiating \eqref{kdef} as in \eqref{kdef-deriv} (note here that it is important that $i+j \geq 1$).  It remains to control the kernel of the first three components.  This kernel on $\R^2 \times \R/\Z$ (and its derivatives) can be obtained by descent from the kernel of the corresponding operator on $\R^3$ (and its derivatives) by summing over cosets of $\{0\} \times \{0\} \times \Z$ as in \eqref{kdef}, \eqref{kdef-deriv}. Thus, if we let $\tilde K$ denote the kernel of the first three components of $\tilde A$ on $\R^3$, it will suffice to show that
$$
 |\nabla^i_x \nabla^j_y \tilde K(x,y)| \lesssim_{i,j} |x-y|^{-i-j-1}
$$
for $0 \leq i,j \leq 100$ with $i+j \geq 1$; the condition $i+j \geq 1$ is needed to ensure a convergent sum over the coset of $\{0\} \times \{0\} \times \Z$, but will not otherwise be needed henceforth.

By linearity and taking adjoints, it thus suffices to verify the above bound for the integral kernel of $\tilde A_0 \Delta^{-1} \partial_l$ on $\R^3$ for $l=1,2$.

From the Newton formula
$$ \Delta^{-1} f(w) = \frac{1}{4\pi} \int_{\R^3} \frac{f(y)}{|w-y|}\ d\operatorname{vol}(y)$$
on $\R^3$, we see that the kernel $\Delta^{-1} \partial_l$ is given by $L(w-y)$, where
$$ L(x) \coloneqq \frac{-1}{4\pi} \frac{x^l}{|x|^3};$$
also, from \eqref{tao}, the kernel $R(x,w)$ of $\tilde A_0$ is given by
$$ R(x,w) = \sum_{k=0}^\infty R_k(x,w)$$
where
$$ R_k(x,w) \coloneqq \frac{2}{M} 2^{3k} x^1 \gamma(2^k x^2) \psi(2^k x^1) \varphi(2^k (w^1,w^2) ) \kappa( 2^k (x^3-w^3) ).$$
Thus it will suffice to show that
\begin{equation}\label{lh}
 |\sum_{k=0}^\infty \nabla^i_x \nabla^j_y \int_{\R^3} R_k(x,w) L(w-y)\ d\operatorname{vol}(w)| \lesssim |x-y|^{-i-j-1}
\end{equation}
for $0 \leq i,j \leq 100$.

From the construction of $\gamma,\psi,\varphi,\kappa$ we see that $R_k(x,w)$ is supported on the region $|x-w| \leq 100 \times 2^{-k}$ and obeys the derivative bounds
\begin{equation}\label{rik}
 |\nabla^i_x \nabla^j_w R_k(x,w)| \lesssim 2^{(i+j+2)k}
\end{equation}
for $0 \leq i,j \leq 100$.  Also, from the moment conditions on $\varphi$ we see that for any $x \in \R^3$, the function $w \mapsto R_k(x,w)$ is orthogonal to any polynomial of degree at most $1000$.

Let us first consider the contribution to the left-hand side of \eqref{lh} of those $k$ for which
\begin{equation}\label{dink}
 |x-y| \geq 200 \times 2^{-k}.
\end{equation}
Then we have $|w-y| \gtrsim |x-y|$, and hence $|\nabla^m_y L(w-y)| \lesssim |x-y|^{-2-m}$ for any $0 \leq m \leq 1000$.  For each fixed $x \in \R^3$, and for $w$ in the support of $R_k(x,w)$, one can then use Taylor expansion to write $\nabla^j L(w-y)$ as a polynomial of degree at most $1000$, plus an error of size at most $O( (2^{-k}/|x-y|)^{500} |x-y|^{-2-j} )$ (say).  Using \eqref{rik} (with $j$ replaced by $0$), and the support of $R_k$, we conclude that
$$ |\nabla^i_x \nabla^j_y \int_{\R^3} R_k(x,w) L(w-y)\ dw| \lesssim 2^{(i+2)k} \times 2^{-3k} \times (2^{-k}/|x-y|)^{500} |x-y|^{-2-j}.$$
Summing over all $k$ obeying \eqref{dink}, we see that this contribution to the left-hand side of \eqref{lh} is acceptable.

It remains to treat the contribution of those $k$ for which \eqref{dink} fails.  In this case we integrate by parts to obtain the identity
$$ |\nabla^i_x \nabla^j_y \int_{\R^3} R_k(x,w) L(w-y)\ d\operatorname{vol}(w)|  = | \int_{\R^3} \nabla^i_x \nabla^j_w R_k(x,w) L(w-y)\ d\operatorname{vol}(w)|.$$
Applying \eqref{rik} and the support of $R_k$, we conclude that
$$ |\nabla^i_x \nabla^j_y \int_{\R^3} R_k(x,w) L(w-y)\ d\operatorname{vol}(w)| \lesssim 2^{(i+j+2)k} \int_{|w-x| \leq 100 \times 2^{-k}} |L(w-y)|\ d\operatorname{vol}(w).$$
Since \eqref{dink} fails, the condition $|w-x| \leq 100 \times 2^{-k}$ implies that $|w-y| \lesssim 2^{-k}$, and hence by the bound $|L(w-y)| \lesssim |w-y|^{-2}$, we have
$$ 2^{(i+j+2)k} \int_{|w-x| \leq 100 \times 2^{-k}} |L(w-y)|\ d\operatorname{vol}(w) \lesssim 2^{(i+j+1) k}.$$
Summing over all $k$ for which \eqref{dink} fails, we see that this contribution to \eqref{lh} is also acceptable.
\end{proof}

Next, we establish positive definiteness.

\begin{proposition}  For any $\omega \in C^\infty_c \cap B_2(\M)$, we hqave
\begin{equation}\label{Sam}
\int_{\M} \langle \omega, \tilde A \omega \rangle\ d \operatorname{vol} = \left(1 + O\left(\frac{1}{M}\right)\right) \| \omega \|_{\dot H^{-1}(\M)}^2.
\end{equation}
\end{proposition}

\begin{proof}  From Plancherel's theorem, the contribution of the last three terms of \eqref{awdef} to the left-hand side of \eqref{Sam} is precisely $\| \omega \|_{\dot H^{-1}(\M)}^2$.  By the Cauchy-Schwarz inequality and the triangle inequality, it thus suffices to establish the bounds
$$ \| \tilde A_0 u \|_{\dot H^1(\M)} \lesssim \frac{1}{M} \| u \|_{L^2(\M)}$$
and
$$ \| \tilde A^*_0 v \|_{L^2(\M)} \lesssim \frac{1}{M} \| v \|_{\dot H^{-1}(\M)}$$
for $u \in L^2(\M)$ and $v \in \dot H^{-1}(\M)$.  But this follows from \eqref{traq-2} and duality.
\end{proof}

Let $\theta_0:\R^2 \to \R$ be initial data of the type in Proposition \ref{sqgb}, and let $\tilde \theta_0\colon \M \to \R$ be the lift of $\theta_0$ to $\M$ defined by $\tilde \theta_0 \coloneqq \theta_0 \circ \pi$.  Following \eqref{omega-def}, we define the initial data $\omega_0 \in C^\infty_c \cap B_2(\M)$ by the formula
\begin{equation}\label{omeg0-def}
\omega_0 \coloneqq (\partial_1 \tilde \theta_0) dx^1 \wedge dx^3 + (\partial_2 \tilde \theta_0) dx^2 \wedge dx^3.
\end{equation}
We now claim (for $M$ sufficiently large) that Theorem \ref{second-blow} holds with this choice of initial data $\omega_0$ and with the operator $\tilde A$ constructed above as vector potential operator.  We have already verified that $\tilde A$ is $100$-reasonable, formally self-adjoint, and obeys \eqref{posdef} (if $M$ is sufficiently large depending on $\eps$).  Thus, the only way that Theorem \ref{second-blow} can still fail is if there is a solution $\omega \in X^{10,2}$, $u \in Y^{10,2}$ to the generalised Euler equations with vector potential operator $\tilde A$ and initial vorticity $\omega_0$ on the time interval $[0,1]$.

Suppose for contradiction that this is the case.
Obseve that $\omega_0$ is invariant with respect to translations in the $x^3$ direction, and that $\tilde A$ commutes with these translations.  Thus, if $\omega,u$ solve the generalised Euler equations with initial data $\omega_0$, then so do any translates of $\omega,u$ in the $x^3$ direction.  Applying the uniqueness component of Theorem \ref{lest}, we conclude that $\omega,u$ are invariant with respect to translations in the $x^3$ direction, thus
$$ \partial_3 \omega = 0; \quad \partial_3 u = 0.$$

We define the scalar field $\tilde \theta\colon [0,1] \times \M \to \R$ by solving the transport equation
\begin{equation}\label{ttt}
\partial_t \tilde \theta + {\mathcal L}_u \tilde \theta = 0
\end{equation}
with initial data $\tilde \theta = \tilde \theta_0$.  Since $u$ lies in $Y^{10,2}$ and $\tilde \theta_0$ is smooth and compactly supported, there is no difficulty defining $\tilde \theta$ uniquely, in such a way that it is continuously differentiable in both space and time, and compactly supported in space. Since $\tilde \theta$ and $u$ are invariant with respect to translations in the $x^3$ direction, $\tilde \theta$ is also.

We now can justify the formal ansatz \eqref{omega-def}:

\begin{proposition}\label{formal}  On $[0,1] \times \M$, we have
\begin{align*}
 \omega &= d \tilde \theta \wedge d x^3 \\
&= \partial_1 \tilde \theta dx^1 \wedge d x^3 + \partial_2 \tilde \theta dx^2 \wedge dx^3.
\end{align*}
\end{proposition}

\begin{proof}  Set $\alpha$ to be the $2$-form
$$ \alpha \coloneqq \omega - d \tilde \theta \wedge dx^3,$$
then $\alpha$ is continuously differentiable in space and time, and our task is to show that $\alpha(t)=0$ for all $t \in [0,1]$.  From \eqref{omeg0-def} we know that $\alpha(0)=0$.  We now use \eqref{vort-1}, \eqref{ttt} to compute
\begin{align*}
(\partial_t + {\mathcal L}_u) \alpha &= (\partial_t + {\mathcal L}_u) \alpha - (\partial_t + {\mathcal L}_u) (d \tilde \theta \wedge dx^3) \\
&= 0 - d 0 \wedge dx^3 - d \tilde \theta \wedge d {\mathcal L}_u x^3 \\
&= - d \tilde \theta \wedge d u^3.
\end{align*}
On the other hand, from \eqref{vort-2} and \eqref{awdef} we have
\begin{align*}
u^3 &= \partial_1 \partial_1 \Delta^{-1} \tilde A_0^* \omega_{12} + \partial_2 \partial_2 \Delta^{-1} \tilde A_0^* \omega_{12}  \\
&\quad + \partial_1 \Delta^{-1} \omega_{13} + \partial_2 \Delta^{-1} \omega_{23} \\
&= \tilde A_0^* \alpha_{12} + \partial_1 \Delta^{-1} \partial_1 \tilde \theta + \partial_2 \Delta^{-1} \partial_2 \tilde \theta \\
&\quad + \partial_1 \Delta^{-1} \alpha_{13} + \partial_2 \Delta^{-1} \alpha_{23} \\
&= (\tilde A_0^* \alpha_{12} + \partial_1 \Delta^{-1} \alpha_{13} + \partial_2 \Delta^{-1} \alpha_{23}) - \tilde \theta.
\end{align*}
Since $d\tilde \theta \wedge d\tilde \theta = 0$, we thus have
$$
(\partial_t + {\mathcal L}_u) \alpha = - d \tilde \theta \wedge d(\tilde A_0^* \alpha_{12} + \partial_1 \Delta^{-1} \alpha_{13} + \partial_2 \Delta^{-1} \alpha_{23}).$$
Taking inner products with $\alpha$ and integrating by parts (which can be justified as $\omega$ lies in $X^{10,2}$ and $\theta$ is continuously differentiable and compactly supported), we conclude that
$$
\partial_t \| \alpha \|_{L^2}^2 = - 2 \langle d \tilde \theta \wedge d(\tilde A_0^* \alpha_{12} + \partial_1 \Delta^{-1} \alpha_{13} + \partial_2 \Delta^{-1} \alpha_{23}), \alpha \rangle_{L^2(\M)}.$$
From the proof of Proposition \ref{pass-ok}, we know that $\tilde A_0^*$ maps $L^2(\M)$ to $\dot H^1(\M)$.  As $d\tilde \theta$ is bounded, we conclude that
$$
\partial_t \| \alpha \|_{L^2}^2 \lesssim_\theta \| \alpha\|_{L^2}^2$$
and hence from Gronwall's inequality we have $\alpha(t)=0$ for all $0 \leq t \leq 1$, as required.
\end{proof}

If we insert the above proposition back into \eqref{awdef}, we have
\begin{align*}
\tilde A \omega & \coloneqq - \tilde A_0 \Delta^{-1} (\partial_1 \partial_1 \tilde \theta + \partial_2 \partial_2 \tilde \theta ) \frac{d}{d x^1} \wedge \frac{d}{d x^2} \\
&\quad + \Delta^{-1} \partial_1 \tilde \theta \frac{d}{d x^1} \wedge \frac{d}{d x^3} \\
&\quad + \Delta^{-1} \partial_2 \tilde \theta \frac{d}{d x^2} \wedge \frac{d}{d x^3}. 
\end{align*}
The first term on the right-hand side simplifies to $\tilde A_0 \tilde \theta  \frac{d}{d x^1} \wedge \frac{d}{d x^2}$.  Taking divergences (and recalling that $\tilde \theta$ is constant in the $x^3$ direction), and using \eqref{vort-2}, we conclude that
$$
 u = \partial_2 (\tilde A_0 \tilde \theta) \frac{d}{d x^1} - \partial_1 (\tilde A_0 \tilde \theta) \frac{d}{d x^2} - \tilde \theta \frac{d}{d x^3}
$$
(cf. \eqref{u-def}).
The equation \eqref{ttt} then becomes
$$
\partial_t \tilde \theta + \partial_2 (\tilde A_0 \tilde \theta) \partial_1 \tilde \theta - \partial_1 (\tilde A_0 \tilde \theta) \partial_2 \tilde \theta = 0.$$
Since $\tilde \theta$ is constant in the $x^3$ direction, we can write $\tilde \theta = \theta \circ \pi$ for some continuously differentiable, compactly supported $\theta\colon [0,1] \times \R^2 \to \R$.  From \eqref{ao} we then have
$$
\partial_t \theta + \partial_2 (A_0 \theta) \partial_1 \theta - \partial_1 (A_0 \theta) \partial_2 \theta = 0.$$
But then $\theta$ contradicts Proposition \ref{sqgb} (with $u^1 \coloneqq \partial_2(A_0 \theta)$ and $u^1 \coloneqq -\partial_1(A_0 \theta)$), as required.

\begin{remark} Applying the above arguments with $A_0$ replaced by the SQG vector potential operator $\Delta^{-1/2}$, we obtain a rigorous connection between SQG and an explicit three dimensional generalised Euler equation.  Namely, if there exists a finite time blowup solution to SQG in $\R^2$ (with suitable decay at infinity), then there exists a finite time blowup solution to a generalised Euler equation in $\R^2 \times \R/\Z$ for an explicit vector potential operator $A$ that is a Fourier multiplier of order $-2$ which is self-adjoint and positive definite.
\end{remark}

\section{Removing the periodicity}\label{period-sec}

We now modify the arguments of the previous section to prove Theorem \ref{third-blow}.  Let $\M$ denote the Euclidean manifold that is represented in Cartesian coordinates by $\R^3$.  Whereas in previous sections we would use the notations $\M$ and $\R^3$ interchangeably, in this section we will take care to distinguish the manifold $\M$ from its Cartesian coordinate representation $\R^3$.  This is because we will be using a number of other coordinate systems for $\M$, such as cylindrical coordinates, in which the coordinate space is not $\R^3$.  More precisely, for any triple $(x^1,x^2,x^3) \in \R^3$ of real numbers, we let $(x^1,x^2,x^3)_{\operatorname{car}} \in \M$ denote the associated point on $\M$, thus the map $(x^1,x^2,x^3) \mapsto (x^1,x^2,x^3)_{\operatorname{car}}$ gives an isomorphism between $\R^3$ and $\M$; however we will not view this isomorphisms as an identification, keeping the point $(x^1,x^2,x^3)_{\operatorname{car}} \in \M$ and the triple $(x^1,x^2,x^3) \in \R^3$ conceptually distinct.

As mentioned in the introduction, the strategy is to try to embed $\R^2 \times \R/\Z$ (or more precisely, $\R^2 \times \R/\Z$ equipped with a constant coefficient Riemannian metric) into $\M$.  Clearly this cannot be done globally, and certainly not isometrically; however, it can be done locally, and nearly isometrically, by modifying the familiar cylindrical coordinates\footnote{We use $\alpha$ here instead of $\theta$ to denote the angular variable, as we will reserve the latter symbol for an active scalar field later in this section.} $(z,r,\alpha)_{\operatorname{cyl}}$ of $\M$, with $(z,r,\alpha) \in \R \times [0,+\infty) \times \R/2\pi \Z$, defined in terms of the Cartesian coordinate system $(x^1,x^2,x^3)_{\operatorname{car}}$ by the change of variables
$$ (x^1,x^2,x^3)_{\operatorname{car}} = (r \cos \alpha, r \sin \alpha, z)_{\operatorname{car}} = (z,r,\alpha)_{\operatorname{cyl}}.$$
Of course, the cylindrical coordinate system is singular at the $x^3$-axis 
\begin{equation}\label{axis}
\{ (x^1,x^2,x^3)_{\operatorname{car}}: x^1=x^2=0 \} = \{ (z,r,\alpha)_{\operatorname{cyl}}: r = 0 \},
\end{equation}
but let us ignore this singularity for the moment and work away from this axis, in which the map $(z,r,\alpha) \mapsto (z,r,\alpha)_{\operatorname{cyl}}$ becomes a diffeomorphism between (most of) $\R \times [0,+\infty) \times \R/2\pi \Z$ and (most of) $\M$.  In cylindical coordinates, the Euclidean first fundamental form 
$$d\eta^2 = (dx^1)^2 + (dx^2)^2 + (dx^3)^2$$
becomes
$$ d\eta^2 = dz^2 + dr^2 + r^2 d\alpha^2$$
while the volume form
$$ d\operatorname{vol} = dx^1 \wedge dx^2 \wedge dx^3 $$
becomes
$$ d\operatorname{vol} = r dz \wedge dr \wedge d\alpha.$$
Note that the first fundamental form and the volume element both have variable coefficients due to the factors of $r$.  In the latter case, we can rectify this by replacing the radial variable $r$ with the modified radial variable $y \coloneqq r^2/2$, thus introducing\footnote{This modified cylindrical coordinate system has been used previously to simplify the true Euler equations in the case of axisymmetric solutions with swirl; see \cite{benjamin}, \cite{turkington}.} a modified cylindrical coordinate system $(z,y,\alpha)_{\operatorname{mod}}$ with $(z,y,\alpha) \in \R \times [0,+\infty) \times \R/2\pi\Z$, defined through the change of variables
$$ (x^1,x^2,x^3)_{\operatorname{car}} = (\sqrt{2y} \cos \alpha, \sqrt{2y} \sin \alpha, z)_{\operatorname{car}} = (z,\sqrt{2y},\alpha)_{\operatorname{cyl}} = (z,y,\alpha)_{\operatorname{mod}}.$$
The volume form is now constant coefficient,
$$ d\operatorname{vol} = dz \wedge dy \wedge d\theta,$$
so in particular the Hodge star $*$ and codifferential $\delta$ look the same when written in $(z,y,\alpha)_{\operatorname{mod}}$ coefficients as they do in $(x^1,x^2,x^3)_{\operatorname{car}}$ coordinates.  However the first fundamental form remains variable coefficient:
$$ d\eta^2 = dz^2 + \frac{1}{2y} dy^2 + 2y d\alpha^2.$$
Nevertheless, we observe that the first fundamental form is \emph{approximately} constant coefficient when $y$ is large.  Indeed, let $\eps>0$ be the quantity in Theorem \ref{third-blow}.  If $M \geq 10^{10}$ is a large constant depending on $\eps$ to be chosen later, and we reparameterise the annular region 
\begin{equation}\label{ann}
 \{ (z,y,\alpha)_{\operatorname{mod}}: |y-M^2/2| < M^{3/2}; |z| < M^{1/2} \}
\end{equation}
in $\M$ using rescaled coordinates $(w^1,w^2,w^3)_{\operatorname{rsc}}$, with $(w^1,w^2,w^3)$ confined to the region 
$$Q \coloneqq (-M^{1/2},M^{1/2}) \times (-M^{1/2}, M^{1/2}) \times \R/2\pi M\Z,$$ 
defined by
$$ (z,y,\alpha)_{\operatorname{mod}} = (w^1, M^2/2 + M w^2, \frac{w^3}{M})_{\operatorname{mod}} = (w^1,w^2,w^3)_{\operatorname{rsc}} $$
or equivalently
$$  (w^1,w^2,w^3)_{\operatorname{rsc}} = \left( \sqrt{M^2/2 + M w^2} \cos \frac{w^3}{M}, \sqrt{M^2/2 + M w^2} \sin \frac{w^3}{M}, w^1\right)_{\operatorname{car}}$$
then the volume form is still constant coefficient in this region,
$$ d\operatorname{vol} = dw^1 \wedge dw^2 \wedge dw^3 $$
and the first fundamental form is almost Euclidean:
\begin{equation}\label{da}
 d\eta^2 = (dw^1)^2 + \left(1 + \frac{2 w^2}{M}\right)^{-1} (dw^2)^2 + \left(1 + \frac{2 w^2}{M}\right) (dw^3)^2.
\end{equation}
From this it is easy to see that the map $w \mapsto w_{\operatorname{rsc}}$ is a bilipschitz identification of $Q$ (with the Euclidean metric) with the region \eqref{ann}, where the bilipschitz constants are bounded uniformly in $M$.  It will later be convenient (mostly for notational reasons) to embed $Q$ as a subset of $\R^2 \times \R/2\pi M\Z$, but we do not attempt to identify the remaining portion of $\R^2 \times \R/2\pi M\Z$ with any portion of $\M$, thus leaving the $(w^1,w^2,w^3)$ coordinate system as a local coordinate system parameterising \eqref{ann} only.

In order to smoothly interpolate between the Euclidean structure on $\R^2 \times \R/2\pi M\Z$ and the Euclidean structure on $\R^3$, we will (for technical reasons) need a very gentle cutoff function $\varphi \in C^c(\M)$ supported in \eqref{ann} which is bounded by $1$ and small in $\dot H^1(\M)$, while remaining invariant with respect to rotations around the axis \eqref{axis}; this is possible due to the failure of the two-dimensional Sobolev embedding $\dot H^1 \not \subset L^\infty$.  More precisely, we set
$$ \varphi( (w^1, w^2, w^3)_{\operatorname{rsc}} ) \coloneqq h( w_1, w_2 )$$
in \eqref{ann}, with $\varphi$ vanishing outside of \eqref{ann}, where $h\colon \R^2 \to [0,1]$ is a smooth, spherically symmetric function supported on $B_{\R^2}(0, \sqrt{M})$ which equals $1$ on $B_{\R^2}(0,10^3)$, and is such that
$$ h(w) = 1 - \frac{\log |w|}{\log \sqrt{M}}$$
when $10^4 \leq |w| \leq \sqrt{M}/10$, with the derivative estimates
$$ |\nabla^j h(w)| \lesssim_j \frac{1}{\log M} \frac{1}{(1+|w|)^j}$$
for all $j \geq 1$ and $w \in \R^2$.  

Let $I\colon \Lambda_0(\M) \to \Lambda_0( \R^2 \times \R/2\pi M\Z )$ be the operator defined by
\begin{align*}
 I f( w^1, w^2, w^3 ) &\coloneqq (\varphi f)( (w^1, w^2, w^3)_{\operatorname{rsc}} )
\end{align*}
for $(w^1,w^2,w^3)$ in $Q$, with $If$ vanishing outside of this region.  
The adjoint operator $I^*\colon \Lambda_0( \R^2 \times \R/2\pi M\Z ) \to \Lambda_0(\M)$ is then given by the formula
$$ I^* f( (w^1, w^2, w^3)_{\operatorname{rsc}} ) = \varphi( (w^1, w^2, w^3)_{\operatorname{rsc}} ) f(w^1,w^2,w^3)$$
in the annulus \eqref{ann}, with $I^* f$ vanishing outside of this annulus.  (The fact that $I^*$ is the adjoint of $I$ follows from the fact that the volume form on $\M$ is given by $dw^1 \wedge dw^2 \wedge dw^3$ in \eqref{ann}, so there is no Jacobian factor.)

Let $\tilde A_0\colon \Lambda_0( \R^2 \times \R/2\pi M\Z ) \to \Lambda_0( \R^2 \times \R/2\pi M\Z )$ be the operator defined by \eqref{tao} (but now with the $x^3$ variable ranging in $\R/2\pi M\Z$ rather than $\R/\Z$).
We now define the operator $\tilde A\colon B_2( \M ) \to \Gamma^2(\M)$ by the formula
\begin{equation}\label{nbd}
\tilde A = I^* A' I + \tilde \eta^{-1} \Delta^{-1} - \varphi \tilde \eta^{-1} \Delta^{-1} \varphi 
\end{equation}
where the operator $A'\colon \Lambda_2(\R^2 \times \R/2\pi M\Z ) \to \Gamma^2(\R^2 \times \R/2\pi M\Z)$ is given by the formula
\begin{equation}\label{ap-def}
\begin{split}
A' \omega & \coloneqq - \tilde A_0 \Delta_w^{-1} (\partial_1 \omega_{13} + \partial_2 \omega_{23} ) \frac{d}{d w^1} \wedge \frac{d}{d w^2} \\
&\quad + \partial_1 \Delta_w^{-1} \tilde A_0^* \omega_{12} \frac{d}{d w^1} \wedge \frac{d}{d w^3} \\
&\quad + \partial_2 \Delta_w^{-1} \tilde A_0^* \omega_{12} \frac{d}{d w^2}\wedge \frac{d}{d w^3} \\
&\quad + \Delta_w^{-1} \omega_{13} \frac{d}{d w^1} \wedge \frac{d}{d w^3} \\
&\quad + \Delta_w^{-1} \omega_{23} \frac{d}{d w^2} \wedge \frac{d}{d w^3} \\
&\quad + \Delta_w^{-1} \omega_{12} \frac{d}{d w^1} \wedge \frac{d}{d w^2}
\end{split}
\end{equation}
where $\omega \in \Lambda_2(\R^2 \times \R/2\pi M\Z )$ is expressed in coordinates as
$$ \omega = \omega_{12} dw^1 \wedge dw^2 + \omega_{13} dw^1 \wedge dw^3 + \omega_{23} dw^2 \wedge dw^3,$$
and $\Delta_w$ denotes the Euclidean Laplacian on $\R^2 \times \R/2\pi M\Z$ (the reader should take care to \emph{not} confuse this with the Laplacian $\Delta$ on $\M$, although the two operators become close to each other in some sense when $M$ is large).  As in the previous section, we need to fix an inverse of $\Delta_w^{-1}$; for sake of concreteness we set
$$ \Delta_w^{-1} \omega(w) = \int_{\R^2 \times \R/2\pi M\Z} \omega(w') K_{2\pi M}(w-w')\ d\operatorname{vol}(w')$$
where
\begin{equation}\label{kdef-mix}
 K_{2\pi M}(w) \coloneqq \lim_{N \to \infty} \sum_{n=-N}^N \frac{1}{4\pi |\tilde w + (0,0,2\pi M n)|} - \frac{\log N}{2\pi}
\end{equation}
and $\tilde w$ is an arbitrary lift of $w$ from $\R^2 \times \R/2\pi M\Z$ to $\R^3$.

Informally, $\tilde A$ behaves like the true Euler vector potential $\tilde \eta^{-1} \Delta^{-1}$ away from \eqref{ann}, but inside the smaller region
$$ \{ (w^1,w^2,w^3)_{\operatorname{rsc}}: (w^1,w^2) \in B_{\R^2}(0,10^3) \}$$
it behaves (in $(w^1,w^2,w^3)$ coordinates) like the operator defined in \eqref{awdef}.

It is easy to see that $\tilde A$ is well defined on $C^\infty_c \cap B_2(\M)$ and formally self-adjoint.  Now we verify the further properties of $\tilde A$ needed for Theorem \ref{third-blow}.

\begin{proposition}\label{pass-ok2} $\tilde A$ is a $100$-reasonable vector potential operator.
\end{proposition} 

\begin{proof}  We begin with establishing \eqref{pass}.  Let $0 \leq k\leq 100$.  From standard elliptic estimates we see that the contribution of the $\tilde \eta^{-1} \Delta^{-1}$ term in \eqref{nbd} is acceptable.  Now we turn to the $\varphi \tilde \eta^{-1} \Delta^{-1} \varphi$ term.  If $\omega$ is bounded in $H^k(\M)$, then from the Leibniz rule and H\"older's inequality $\varphi \omega$ is bounded in both $H^k(\M)$ and $L^1(\M)$.  From Sobolev embedding we see that $\Delta^{-1} \varphi\omega$ is locally in $H^{k+2}(\M)$ (this can be seen for instance by breaking up $\varphi \omega$ into low frequency and high frequency components), with bounds that are allowed to depend on $M$.   From this and the Leibniz rule we see that $\varphi \tilde \eta^{-1} \Delta^{-1} \varphi \omega$ is bounded in $H^{k+2}(\M)$, and from this we see that the contribution of the $\varphi \tilde \eta^{-1} \Delta^{-1} \varphi$ is also acceptable.

To finish the proof of \eqref{pass}, it will suffice to show that
$$ \| I^* A' I \omega \|_{H^{k+2}(\M)} \lesssim_M \| \omega \|_{H^k(\M)}.$$
Changing variables to $(w^1,w^2,w^3)_{\operatorname{rsc}}$ coordinates, we see that it suffices to show that
$$ \| \varphi A' \omega \|_{H^{k+2}(\R^2 \times \R/2\pi M\Z)} \lesssim_M \| \omega \|_{H^k(\R^2 \times \R/2\pi M\Z)}$$
whenever $\omega$ is supported on the support of $\varphi$ (which by abuse of notation we now view as a function on $\R^2 \times \R/2\pi M\Z$).  The contribution of the $\Delta_w^{-1} \omega_{ij}$ terms in \eqref{ap-def} for $ij=13,23,12$ can be treated by the same argument used to control $\varphi \tilde \eta^{-1} \Delta^{-1} \varphi$.  It thus remains to show that
\begin{equation}\label{claim1}
 \| \varphi \tilde A_0 \Delta_w^{-1} \nabla \omega \|_{H^{k+2}(\R^2 \times \R/2\pi M\Z)} \lesssim_M \| \omega \|_{H^k(\R^2 \times \R/2\pi M\Z)}
\end{equation}
and
\begin{equation}\label{claim2}
\| \varphi \nabla \Delta_w^{-1} \tilde A_0^* \omega \|_{H^{k+2}(\R^2 \times \R/2\pi M\Z)} \lesssim_M \| \omega \|_{H^k(\R^2 \times \R/2\pi M\Z)}
\end{equation}
for scalar $\omega \in H^k(\R^2 \times \R/2\pi M\Z)$ supported in the support of $\varphi$.

If $\omega$ is bounded in $H^k(\R^2 \times \R/2\pi M\Z)$, then by \eqref{traq} (replacing $\R/\Z$ with $\R/2\pi M\Z$) we see that $\tilde A_0 \Delta_w^{-1} \nabla \omega$ is bounded in $\dot H^{k+2}( \R^2 \times \R/2\pi M\Z )$, but from \eqref{tao} we also see that this function is supported in $B_{\R^2 \times \R/2\pi M\Z}(0,100)$.  From this and the fundamental theorem of calculus we see that $\tilde A_0 \Delta_w^{-1} \nabla \omega$  is in fact bounded in $H^{k+2}( \R^2 \times \R/2\pi M\Z )$, giving \eqref{claim1}.  A similar argument gives \eqref{claim2}, completing the proof of \eqref{pass}.

Now we show \eqref{nij}.  From the explicit formula $\frac{1}{4\pi |x-y|}$ for the Newton potential kernel of $\Delta^{-1}$, we see that the contribution of the $\tilde \eta^{-1} \Delta^{-1}$ term in \eqref{nbd} is acceptable.  The remaining terms in \eqref{nbd} only give a contribution to the kernel when $x,y = O_M(1)$.  The contribution of $\varphi \tilde \eta^{-1} \Delta^{-1} \varphi $ can then be seen to also be acceptable by the Leibniz rule.  By further application of the Leibniz rule and the chain rule, it thus suffices to show that the kernel $K(w,w')$ of $A'$ obeys the estimates
$$ |\nabla^i_w \nabla^j_{w'} K(w,w')| \lesssim_{M}  |w-w'|^{-i-j-1}$$
whenever $0 \leq i,j \leq M$ with $i+j \geq 1$.  But this follows from the arguments used to prove Proposition \ref{pass-ok}.
\end{proof}

\begin{proposition}  For any $\omega \in C^\infty_c \cap B_2(\M)$, we have
\begin{equation}\label{Sam2}
\int_{\M} \langle \omega, \tilde A \omega \rangle\ d \operatorname{vol} = \left(1 + O\left(\frac{1}{\log M}\right)\right) \| \omega \|_{\dot H^{-1}(\M)}^2.
\end{equation}
\end{proposition}

\begin{proof}  
From Fourier analysis we may write $\omega = dv$ for some $v \in L^2 \cap \Lambda_1(\M)$ with
$$ \| v \|_{L^2(\M)} = \| \omega \|_{\dot H^{-1}(\M)}.$$
From integration by parts, we have
\begin{equation}\label{int1}
\int_{\M} \langle dv, \tilde \eta^{-1} \Delta^{-1}  dv \rangle\ d \operatorname{vol}= \| v \|_{L^2(\M)}^2
\end{equation}
so by \eqref{nbd} and the triangle inequality it suffices to show that
$$
\int_{\M} \langle dv, (I^* A' I - \varphi \tilde \eta^{-1} \Delta^{-1} \varphi) dv \rangle\ d \operatorname{vol} = O\left( \frac{1}{\log M}  \| v \|_{L^2(\M)}^2 \right).$$
From the Newton formula
$$ \Delta^{-1} f(x) = \int_\M \frac{f(y)}{4\pi |x-y|}\ d\operatorname{vol}(y) $$
one has
\begin{align*}
 \int_\M \langle dv, \varphi \tilde \eta^{-1} \Delta^{-1} \varphi dv \rangle\ d \operatorname{vol}
&= \int_\M \int_\M \frac{\langle \varphi dv(x), \tilde \eta^{-1} \varphi dv(y) \rangle}{4\pi|x-y|}\ d \operatorname{vol}(x) \ d \operatorname{vol}(y) \\
&= \int_{\R^2 \times \R/2\pi M\Z} \int_{\R^2 \times \R/2\pi M\Z} \frac{\langle Idv(w), (\tilde \eta')^{-1} Idv(w') \rangle}{4\pi |w_{\operatorname{rsc}} - w'_{\operatorname{rsc}}|}
\ d \operatorname{vol}(w) \ d \operatorname{vol}(w') 
\end{align*}
where $\eta'$ is the metric on the support of $I\omega$ formed by pulling back the Euclidean metric $\eta$, thus by \eqref{da}
\begin{equation}\label{dae}
 (d\eta')^2 = (dw^1)^2 + \left(1 + \frac{2 w^2}{M}\right)^{-1} (dw^2)^2 + \left(1 + \frac{2 w^2}{M}\right) (dw^3)^2.
\end{equation}
Meanwhile, from \eqref{ap-def} we have
\begin{align*}
\int_\M \langle dv, I^* A' I dv \rangle\ d\operatorname{vol} &= \int_{\R^2 \times \R/2\pi M\Z} \langle I dv, A' I dv \rangle\ d\operatorname{vol} \\
&= - 2 \int_{\R^2 \times \R/2\pi M\Z} \langle I dv, \tilde A_0 \Delta_w^{-1} (\partial_1 (Idv)_{13} + \partial_2 (Idv)_{23} ) \frac{d}{d w^1} \wedge \frac{d}{d w^2} \rangle\ d\operatorname{vol} \\
&\quad + \int_{\R^2 \times \R/2\pi M\Z} \int_{\R^2 \times \R/2\pi M\Z} \\
&\quad\quad K_{2\pi M}(w-w') \langle Idv(w), \tilde \eta^{-1} Idv(w') \rangle
\ d \operatorname{vol}(w) \ d \operatorname{vol}(w'),
\end{align*}
where by abuse of notation $\eta$ now also denotes the Euclidean metric on $\R^2 \times \R/2\pi M\Z$, and $K_{2\pi M}$ was defined in \eqref{kdef-mix}.  Thus by the triangle inequality it will suffice to establish the estimates
\begin{equation}\label{rg-1}
 \int_{\R^2 \times \R/2\pi M\Z} \left\langle I dv, \tilde A_0 \Delta_w^{-1} (\partial_1 (Idv)_{13} + \partial_2 (Idv)_{23} ) \frac{d}{d w^1} \wedge \frac{d}{d w^2} \right\rangle\ d\operatorname{vol} = O\left( 
\frac{1}{\log M} \| v \|_{L^2(\M)}^2 \right)
\end{equation}
and
\begin{equation}\label{rg-2}
\begin{split}
&\int_{\R^2 \times \R/2\pi M\Z} \int_{\R^2 \times \R/2\pi M\Z} \\
&\quad \frac{\langle Idv(w), (\tilde \eta')^{-1} Idv(w') \rangle}{4\pi |w_{\operatorname{rsc}} - w'_{\operatorname{rsc}}|} - K_{2\pi M}(w-w') \langle Idv(w), \tilde \eta^{-1} Idv(w') \rangle\\
&\quad 
\ d \operatorname{vol}(w) \ d \operatorname{vol}(w') 
=  O( 
\frac{1}{\log M} \| v \|_{L^2(\M)}^2 ).
\end{split}
\end{equation}
The bound \eqref{rg-1} follows easily from \eqref{traq-2} (with the factor of $\frac{1}{\log M}$ improved to $\frac{1}{M}$), so we turn to \eqref{rg-2}.  Forming the tensor kernel
$$ L(w,w') \coloneqq \frac{1}{4\pi |w_{\operatorname{rsc}} - w'_{\operatorname{rsc}}|} \tilde \eta'(w')^{-1} - K_{2\pi M}(w-w') \tilde \eta(w')^{-1}$$
we see from integration by parts, the chain rule and duality that it suffices to prove the operator norm bound
\begin{equation}\label{dos}
 \left\| \delta \int_{\R^2 \times \R/2\pi M\Z} \varphi(\cdot) L(\cdot,w') \varphi(w') dv(w')\ d\operatorname{vol}(w')\right \|_{L^2(\R^2 \times \R/2\pi M\Z)}
\lesssim \frac{1}{\log M} \| v \|_{L^2(\R^2 \times \R/2\pi M\Z)}
\end{equation}
for all $v \in L^2 \cap \Lambda_1(\R^2 \times \R/2\pi M\Z)$.

To prove this estimate, we first claim the kernel estimates
\begin{equation}\label{kernel}
 |\nabla_w^i \nabla_{w'}^j L(w,w')| \lesssim_{i,j} \frac{1}{\sqrt{M}} \frac{1}{|w-w'|^{i+j}} \left(\frac{1}{\sqrt{M}} + \frac{1}{|w-w'|}\right)
\end{equation}
for distinct $w,w' \in Q$ and $i,j \geq 0$.

From \eqref{dae} we have
$$ |\tilde \eta'(w')^{-1} - \tilde \eta(w')| \lesssim \frac{1}{\sqrt{M}}$$
for $w' \in Q$, and more generally
$$ \left|\nabla_{w'}^j (\tilde \eta'(w')^{-1} - \tilde \eta(w'))\right| \lesssim_j \frac{1}{\sqrt{M}} \frac{1}{M^j}$$
for $w' \in Q$ and $j \geq 0$.  Also, from many applications of the chain rule one has
$$ \left|\nabla_w^i \nabla_{w'}^j \frac{1}{4\pi |w_{\operatorname{rsc}} - w'_{\operatorname{rsc}}|}\right| \lesssim_{i,j} \frac{1}{|x-y|^{1+i+j}}$$
for $w,w' \in Q$ and $i,j \geq 0$, and hence by the product rule
$$ \left|\nabla_w^i \nabla_{w'}^j (\frac{1}{4\pi |w_{\operatorname{rsc}} - w'_{\operatorname{rsc}}|} (\tilde \eta'(w')^{-1} - \tilde \eta(w')^{-1})\right| \lesssim_{i,j} \frac{1}{\sqrt{M}} \frac{1}{|w-w'|^{1+i+j}}.$$
Thus by the triangle inequality it suffices to show that
$$ \left|\nabla_w^i \nabla_{w'}^j (\frac{1}{4\pi |w_{\operatorname{rsc}} - w'_{\operatorname{rsc}}|} \tilde \eta(w')^{-1} - K_{2\pi M}(w-w') \tilde \eta(w')^{-1})\right| \lesssim_{i,j} \frac{1}{\sqrt{M}} \frac{1}{|w-w'|^{i+j}} \left(\frac{1}{\sqrt{M}} + \frac{1}{|w-w'|}\right).$$
As $\tilde \eta$ is constant coefficient, we can drop the $\tilde \eta(w')^{-1}$ factor, thus we reduce to establishing
\begin{equation}\label{ttt-r}
 \left|\nabla_w^i \nabla_{w'}^j (\frac{1}{4\pi |w_{\operatorname{rsc}} - w'_{\operatorname{rsc}}|} - K_{2\pi M}(w-w'))\right| \lesssim_{i,j} \frac{1}{\sqrt{M}} \frac{1}{|w-w'|^{i+j}} \left(\frac{1}{\sqrt{M}} + \frac{1}{|w-w'|}\right)
\end{equation}
for $w,w' \in Q$ and $i,j \geq 0$.

We first dispose of the case where $w,w'$ are very far apart in the sense that $|w-w'| \geq M/10$.  From \eqref{kdef-mix} (and recalling that the $w^1$ and $w^2$ components of $w,w'$ are $O(\sqrt{M})$) we see that
$$ |K_{2\pi M}(w-w')| \lesssim \frac{1}{M}$$
which by the harmonicity of $K_{2\pi M}$ implies that
$$ |\nabla_w^i \nabla_{w'}^j K_{2\pi M}(w-w')| \lesssim_{i,j} \frac{1}{M^{1+i+j}}$$
Similarly, as the map $w \mapsto w_{\operatorname{rsc}}$ is bilipschitz with all derivatives bounded, we have
$$ \left|\frac{1}{4\pi |w_{\operatorname{rsc}} - w'_{\operatorname{rsc}}|}\right| \lesssim \frac{1}{M}$$
and more generally
$$ \left|\nabla_w^i \nabla_{w'}^j \frac{1}{4\pi |w_{\operatorname{rsc}} - w'_{\operatorname{rsc}}|}\right| \lesssim_{i,j} \frac{1}{M^{1+i+j}},$$
and so \eqref{ttt-r} follows from the triangle inequality in this case.

Henceforth we suppose that $|w-w'| < M/10$.  From \eqref{kdef-mix} we now have
$$ \left|K_{2\pi M}(w-w') - \frac{1}{4\pi |w-w'|}\right| \lesssim \frac{1}{M}$$
which by harmonicity implies
$$ \left|\nabla_w^i \nabla_{w'}^j (K_{2\pi M}(w-w') - \frac{1}{4\pi |w-w'|}) \right| \lesssim_{i,j} \frac{1}{M} \frac{1}{|w-w'|^{i+j}}.$$
Thus by the triangle inequality, it suffices to show that

\begin{equation}\label{wawa}
 \left|\nabla_w^i \nabla_{w'}^j \left(\frac{1}{|w_{\operatorname{rsc}} - w'_{\operatorname{rsc}}|} - \frac{1}{|w - w'|}\right)\right| \lesssim_{i,j} \frac{1}{\sqrt{M}} \frac{1}{|w-w'|^{i+j}} \left(\frac{1}{\sqrt{M}} + \frac{1}{|w-w'|}\right).
\end{equation}
We divide into two cases, depending on whether $|w-w'|$ is less than $\sqrt{M}$ or not.  First suppose that $|w-w'| \leq \sqrt{M}$, thus $w,w'$ both lie in $B_{\R^2 \times \R/2\pi M\Z}(w_0,\sqrt{M})$ for some $w_0 \in Q$.  Let $B$ denote the convex region
$$ B \coloneqq \{ u \in B_{\R^3}(0,1): w_0 + \sqrt{M} u \in Q \},$$
and let $f\colon B \to \R^3$ be the map
\begin{equation}\label{fdef}
 f( u ) \coloneqq \frac{ (w_0 + \sqrt{M} u)_{\operatorname{rsc}} - (w_0)_{\operatorname{rsc}} }{ \sqrt{M} }
\end{equation}
then it is easy to see that $f$ is bilipschitz on $B$ with constants comparable to $1$, and from Taylor expansion we see that 
$$ \nabla^i_u (f(u) - u) = O_i\left(\frac{1}{\sqrt{M}}\right)$$
on $B$ for all $i \geq 0$.  In particular, we have for distinct $u,v \in B$ that $|f(u)-f(v)|$ is comparable to $|u-v|$, and from several applications of the chain rule (and writing $f(u) = u + \frac{1}{\sqrt{M}} g(u)$ for some function $g$ with all derivatives bounded on $B$) we have
$$ \nabla^i_u \nabla^j_v \left(\frac{1}{|f(u)-f(v)|} - \frac{1}{|u-v|}\right) = O_{i,j}\left( \frac{1}{\sqrt{M} |u-v|^{1+i+j}} \right)$$
for $i,j \geq 0$.  Setting $w = w_0 + \sqrt{M} u$ and $w' = w_0 + \sqrt{M} v$, we obtain \eqref{wawa} when $|w-w'| \leq \sqrt{M}$.

To complete the proof of \eqref{kernel}, we need to establish \eqref{wawa} in the case $\sqrt{M} < |w-w'| \leq M/10$.  Set $R \coloneqq |w-w'|/\sqrt{M}$, then $1 \leq R \leq \sqrt{M}/10$, and $w,w'$ both lie in $B_{\R^2 \times \R/2\pi M\Z}(w_0,R \sqrt{M})$ for some $w_0 \in Q$.  Setting $B_R$ to be the convex region
$$ B_R \coloneqq \{ u \in B_{\R^3}(0,R): w_0 + \sqrt{M} u \in Q \}$$
and defining $f$ by \eqref{fdef} as before, one has from Taylor expansion that
$$ f(u) = u + O\left( \frac{R^2}{\sqrt{M}} \right) $$
and more generally
$$ \nabla^i_u (f(u) - u) = O_i(\frac{R^{2-i}}{\sqrt{M}} )$$
on $B_R$ for all $i \geq 0$.  As before, $|f(u)-f(v)|$ is comparable to $|u-v|$. By many applications of the chain rule, we have
$$ \nabla^i_u \nabla^j_v \left(\frac{1}{|f(u)-f(v)|} - \frac{1}{|u-v|}\right) = O_{i,j}\left( \frac{R}{\sqrt{M}} |u-v|^{-1-i-j}\right )$$
for $i,j \geq 0$.  Setting $w = w_0 + \sqrt{M} u$ and $w' = w_0 + \sqrt{M} v$, so that $|u-v|$ is comparable to $R$, we obtain \eqref{wawa} when $\sqrt{M} < |w-w'| \leq M/10$.

This completes the proof of \eqref{kernel} in all cases.  We now return to the proof of \eqref{dos}.
Let $\pi\colon \R^2 \times \R/2\pi M\Z \to \R^2$ denote the projection map $\pi\colon (w^1,w^2,w^3) \mapsto (w^1,w^2)$.
We smoothly partition $L = L_1 + L_2$, where $L_1(w,w')$ is the ``local'' part of $L(w,w')$ smoothly restricted to the region where $|w-w'| \leq \log M \min( |\pi(w)|, |\pi(w')| )$, and $L_1$ is the ``global'' part, restricted to the region where $|w-w'| \gg \log M \min( |\pi(w)|, |\pi(w')|)$.  More explicitly, we can set
$$ L_1(w,w') \coloneqq L(w,w') \chi\left( \frac{w-w'}{|\pi(w)| \log M} \right) \chi\left( \frac{w-w'}{|\pi(w')| \log M} \right)$$
where $\chi\colon \R^2 \to [0,1]$ is a smooth function supported on $B_{\R^2}(0,1)$ that equals one on $B_{\R^2}(0,1/2)$, and set $L_2 \coloneqq L - L_1$.  By the triangle inequality, it thus suffices to establish the bounds
\begin{equation}\label{t1}
 \left\| \delta \int_{\R^2 \times \R/2\pi M\Z} \varphi(\cdot) L_l(\cdot,w') \varphi(w') dv(w')\ d\operatorname{vol}(w') \right\|_{L^2(\R^2 \times \R/2\pi M\Z)}
\lesssim \frac{1}{\log M} \| v \|_{L^2(\R^2 \times \R/2\pi M\Z)}
\end{equation}
for $l=1,2$.  

In the $l=1$ case, we note that as $L_l$ is supported in the regime where $|w-w'| \lesssim |\pi(w)| \log M, |\pi(w')| \log M \lesssim \sqrt{M} \log M$, and we see from \eqref{kernel} and the product rule that we have the Calder\'on-Zygmund bounds
$$ |\nabla_w \nabla_{w'}(\varphi(w) L_l(w,w') \varphi(w'))| \lesssim \frac{\log^{O(1)} M}{\sqrt{M}} \frac{1}{|w-w'|^{3}}$$
and
$$ |\nabla_{w,w'} \nabla_w \nabla_{w'}(\varphi(w) L_l(w,w') \varphi(w'))| \lesssim \frac{\log^{O(1)} M}{\sqrt{M}} \frac{1}{|w-w'|^{4}}$$
for $w \neq w'$.  Also, the operator that maps $v$ to
$$ \delta \int_{\R^2 \times \R/2\pi M\Z} \varphi(\cdot) L_l(\cdot,w') \varphi(w') dv(w')\ d\operatorname{vol}(w') $$
clearly annihilates the constant function $1$, as does its adjoint.  Applying the $T(1)$ theorem of David and Journ\'e \cite{david}, we obtain the $l=1$ case of \eqref{t1} (with the $\frac{1}{\log M}$ factor improved to $\frac{\log^{O(1)} M}{\sqrt{M}}$).

Now we handle the $l=2$ case.  From \eqref{kernel} and the product rule, we have the bounds
\begin{align*}
 |\nabla_w \nabla_{w'}(\varphi(w) L_l(w,w') \varphi(w'))| &\lesssim \frac{1}{\sqrt{M}} 
\left(\frac{1}{\sqrt{M}} + \frac{1}{|w-w'|}\right) \\
&\quad \times \frac{1}{\log^2 M} \frac{1}{1+|\pi(w)|} \frac{1}{1+|\pi(w')|}
\end{align*}
since any factor of $\frac{1}{|w-w'|}$ that comes when a derivative falls on $L_l$ can be replaced instead by $\frac{1}{\log M} \frac{1}{1+|\pi(w)|}$ or $\frac{1}{\log M} \frac{1}{1+|\pi(w')|}$.  In particular, we have the estimates
$$
\int_{\R^2 \times \R/2\pi M \Z}  |\nabla_w \nabla_{w'}(\varphi(w) L_l(w,w') \varphi(w'))|\ \frac{d\operatorname{vol}(w')}{1 + |\pi(w')|} \lesssim \frac{1}{\log M} \frac{1}{1+|\pi(w)|}
$$
for all $w$, and
$$
\int_{\R^2 \times \R/2\pi M \Z}  |\nabla_w \nabla_{w'}(\varphi(w) L_l(w,w') \varphi(w'))|\ \frac{d\operatorname{vol}(w)}{1 + |\pi(w)|} \lesssim \frac{1}{\log M} \frac{1}{1+|\pi(w')|}
$$
for all $w'$.  The $l=2$ case of \eqref{t1} then follows from the weighted Schur test (after integrating by parts to move all derivatives onto $\varphi(w) L_l(w,w') \varphi(w')$).
\end{proof}

Now we prove Theorem \ref{third-blow}.  As in the previous section, let $\theta_0:\R^2 \to \R$ be initial data of the type in Proposition \ref{sqgb}; we can choose $\theta_0$ so that it is supported in the ball $B_{\R^2}(0,200)$.  Let $\tilde \theta_0\colon \M \to \R$ be the lift of $\theta_0$ to $\M$ defined by setting
$$ \tilde \theta_0( (w^1,w^2,w^3)_{\operatorname{rsc}} ) \coloneqq \theta_0( w^1,w^2 )$$
in the region \eqref{ann}, with $\tilde \theta_0$ vanishing outside of \eqref{ann}.  Clearly $\tilde \theta$ is smooth and supported in the set $\{ (w^1,w^2,w^3)_{\operatorname{rsc}}: |w^1|, |w^2| \leq 200 \}$.  We define the initial data $\omega_0 \in C^\infty_c \cap B_2(\M)$ by the formula
\begin{equation}\label{omeg1-def}
\omega_0\coloneqq d\tilde \theta_0 \wedge dw^3 = d(\tilde \theta_0 dw^3),
\end{equation}
noting that the $1$-form $dw^3$ is well-defined on the support of $\tilde \theta_0$.  This is clearly a closed $2$-form.
We now claim (for $M$ sufficiently large) that Theorem \ref{third-blow} holds with this choice of initial data $\omega_0$ and with the operator $\tilde A$ constructed above as vector potential operator.  We have already verified that $\tilde A$ is $100$-reasonable, formally self-adjoint, and obeys \eqref{posdef-2} (if $M$ is sufficiently large depending on $\eps$).  Thus, the only way that Theorem \ref{third-blow} can still fail is if there is a solution $\omega \in X^{10,2}$, $u \in Y^{10,2}$ to the generalised Euler equations with vector potential operator $\tilde A$ and initial vorticity $\omega_0$ on the time interval $[0,1]$.

Suppose for contradiction that this is the case.
Obseve that $\omega_0$ is invariant with respect to rotations around the $x^3$ axis \eqref{axis} (which, in the region \eqref{ann}, corresponds to translations in the $w^3$ direction), and that $\tilde A$ commutes with these rotations.  Thus, if $\omega,u$ solve the generalised Euler equations with initial data $\omega_0$, then so do any rotations of $\omega,u$ around the $x^3$ axis.  Applying the uniqueness component of Theorem \ref{lest}, we conclude that $\omega,u$ are invariant with respect to rotations around the $x^3$ axis.  In particular, in the region \eqref{ann}, we have
$$ \partial_3 \omega = 0; \quad \partial_3 u = 0$$
in the $(w^1,w^2,w^3)_{\operatorname{rsc}}$ coordinate system.

We define the scalar field $\tilde \theta\colon [0,1] \times \M \to \R$ by solving the transport equation \eqref{ttt} 
with initial data $\tilde \theta = \tilde \theta_0$.  Again, there is no difficulty defining $\tilde \theta$, and it is continuously differentiable in both space and time, and compactly supported in space. Since $\tilde \theta$ and $u$ are invariant with respect to rotations around the $x^3$ axis, $\tilde \theta$ is also.

Next, we claim that $\omega$ stays well within the region \eqref{ann} and obeys the analogue of Proposition \ref{formal}:

\begin{proposition}  For each $0 \leq t \leq 1$, let $\Omega(t)$ be the subset of \eqref{ann} defined by
$$ \Omega(t) \coloneqq \{ (w^1,w^2,w^3)_{\operatorname{rsc}}: |w^1|, |w^2| \leq 300+t \}.$$
Then $\omega(t)$ and $\tilde \theta(t)$ are supported in $\Omega(t)$ for all $0 \leq t \leq 1$, and we have
\begin{align*}
 \omega &= d \tilde \theta \wedge d w^3 \\
&= \partial_1 \tilde \theta dw^1 \wedge d w^3 + \partial_2 \tilde \theta dw^2 \wedge dw^3.
\end{align*}
\end{proposition}

\begin{proof}  We again use the barrier method.  Since $u$ lies in $Y^{10,2}$, it is bounded, and hence $\omega$ is transported at bounded speed.  Suppose the first claim fails, thus $\omega(t)$ or $\tilde \theta(t)$ is not supported in $\Omega(t)$ for some $0 \leq t \leq 1$.  Let $0 \leq T \leq 1$ be the infimum of all the times $t$ in which $\omega(t)$ or $\tilde \theta(t)$ is not supported in $\Omega(t)$.  Since this is a closed condition, we have $T < 1$.  Since $\omega(0)$ and $\tilde \theta(0)$ are supported in the interior of $\Omega(0)$ and is transported at bounded speed, we have $T > 0$.  

For times $t \in [0,T]$, set $\alpha$ to be the $2$-form
$$ \alpha \coloneqq \omega - d \tilde \theta \wedge dw^3,$$
then $\alpha$ is continuously differentiable in space and time and supported in \eqref{ann}.  From \eqref{omeg0-def} we know that $\alpha(0)=0$.  As in the proof of Proposition \ref{formal}, we use \eqref{vort-1}, \eqref{ttt} to compute
\begin{align*}
(\partial_t + {\mathcal L}_u) \alpha &= (\partial_t + {\mathcal L}_u) \alpha - (\partial_t + {\mathcal L}_u) (d \tilde \theta \wedge dw^3) \\
&= 0 - d 0 \wedge dw^3 - d \tilde \theta \wedge d {\mathcal L}_u w^3 \\
&= - d \tilde \theta \wedge d u^3
\end{align*}
where $u^3$ is the $\frac{d}{dw^3}$ component of $u$.
On the other hand, from \eqref{vort-2} and \eqref{nbd}, \eqref{ap-def} (noting that $\tilde \eta$ equals $1$ on the support of $\omega$ or $\tilde \theta$ for times in $[0,T]$) we have
\begin{align*}
u^3 &= \partial_1 \partial_1 \Delta_w^{-1} \tilde A_0^* \omega_{12} + \partial_2 \partial_2 \Delta_w^{-1} \tilde A_0^* \omega_{12}  \\
&\quad + \partial_1 \Delta_w^{-1} \omega_{13} + \partial_2 \Delta_w^{-1} \omega_{23} \\
&= \tilde A_0^* \alpha_{12} + \partial_1 \Delta^{-1} \partial_1 \tilde \theta + \partial_2 \Delta_w^{-1} \partial_2 \tilde \theta \\
&\quad + \partial_1 \Delta^{-1} \alpha_{13} + \partial_2 \Delta_w^{-1} \alpha_{23} \\
&= (\tilde A_0^* \alpha_{12} + \partial_1 \Delta_w^{-1} \alpha_{13} + \partial_2 \Delta_w^{-1} \alpha_{23}) - \tilde \theta
\end{align*}
where all derivatives and components are with respect to the $(w^1,w^2,w^3)_{\operatorname{rsc}}$ coordinate system.  Repeating the arguments in Theorem \ref{formal} verbatim, we thus have $\alpha(t)=0$ for all $0 \leq t \leq T$.  Thus we have
$$ \omega = d \tilde \theta \wedge dw^3$$
for times $0 \leq t \leq T$.  In particular, $\omega_{12}$ vanishes.

By continuity and a compactness argument, there must exist a point $(w^1,w^2,w^3)_{\operatorname{rsc}}$ on the boundary of $\Omega(T)$ which is also on the boundary of the support of $\omega(T)$ or $\tilde \theta(T)$, but such that the support of $\omega(t)$ or $\tilde \theta(t)$ escapes $\Omega(t)$ in any given neighbourhood of $(w^1,w^2,w^3)_{\operatorname{rsc}}$ for times $t>T$ arbitrarily close to $T$.

Now we compute the vector potential $\tilde A \omega$ and velocity field $u$ at time $T$ and in a sufficiently small neighbourhood of this point $(w^1,w^2,w^3)_{\operatorname{rsc}}$.  In this neighbourhood and on the support of $\omega(T)$, the cutoff $\tilde \eta$ equals $1$, and from \eqref{tao} we see that $A_0$ and $A_0^*$ vanish in this neighbourhood.  From \eqref{nbd}, \eqref{ap-def} and the vanishing of $\omega_{12}$ we thus have
$$
\tilde A \omega = \Delta_w^{-1} \omega_{13} \frac{d}{d w^1} \wedge \frac{d}{d w^3} + \Delta_w^{-1} \omega_{23} \frac{d}{d w^2} \wedge \frac{d}{d w^3}$$
in this neighbourhood, after abusing notation and identifying this neighbourhood with a subset of $\R^2 \times \R/2\pi M\Z$.  Since $\omega_{13}, \omega_{23}$ are invariant with respect to translations in the $w^3$ neighbourhood, we conclude that $\delta \tilde A \omega$ has vanishing $\frac{d}{dw^1}$ and $\frac{d}{dw^2}$ components in this neighbourhood, thus the velocity field $u$ is parallel to $\frac{d}{dw^3}$.  But since $\Omega(T)$ is invariant in the $\frac{d}{dw^3}$ direction and is expanding outwards in the other two directions, we see from the transport equations for $\omega(t)$ and $\tilde \theta(t)$ that for $t>T$ sufficiently close to $T$, $\omega(t)$ and $\tilde \theta(t)$ are supported inside $\Omega(t)$ in this neighbourhood, contradicting the construction of $(w^1,w^2,w^3)_{\operatorname{rsc}}$.  Thus $\omega(t)$ and $\tilde \theta(t)$ are supported in $\Omega(t)$ for all $0 \leq t \leq 1$.  Repeating the above arguments we then obtain the second claim of the proposition.
\end{proof}

If we insert the above proposition back into \eqref{nbd}, \eqref{ap-def}, noting again that $\tilde \eta$ equals $1$ on the support of $\omega$ or $\tilde \theta$, we have
\begin{align*}
\tilde A \omega & \coloneqq - \tilde A_0 \Delta_w^{-1} (\partial_1 \partial_1 \tilde \theta + \partial_2 \partial_2 \tilde \theta ) \frac{d}{d w^1} \wedge \frac{d}{d w^2} \\
&\quad + \Delta_w^{-1} \partial_1 \tilde \theta \frac{d}{d w^1} \wedge \frac{d}{d w^3} \\
&\quad + \Delta_w^{-1} \partial_2 \tilde \theta \frac{d}{d w^2} \wedge \frac{d}{d w^3}. 
\end{align*}
Repeating the arguments of the previous section verbatim, we see that (in $(w^1,w^2,w^3)_{\operatorname{rsc}}$ coordinates) $\tilde \theta$ is the lift of a continuously differentiable, compactly supported function $\theta\colon [0,1] \times \R^2 \to \R$ that contradicts Proposition \ref{sqgb}, as required.

\end{document}